\documentclass[11pt]{article}

\usepackage[a4paper, total={5.9in, 9in}]{geometry}
\usepackage{emptypage}
\usepackage[utf8]{inputenc}
\usepackage[T1]{fontenc}
\usepackage{domitian}
\usepackage[english]{babel}
\usepackage{amsmath, amssymb}
\usepackage{amsthm}
\usepackage{longtable}
\usepackage{placeins}
\usepackage[font=small,labelfont=bf]{caption}

\usepackage{algorithm}
\usepackage{algpseudocode}
\algrenewcommand\algorithmicrequire{\textbf{Input:}}
\algrenewcommand\algorithmicensure{\textbf{Output:}}

\usepackage{enumerate}
\usepackage{xcolor}
\RequirePackage[colorlinks,allcolors=blue]{hyperref}
\RequirePackage{graphicx}
\usepackage[nottoc]{tocbibind}

\usepackage{titlesec}
\titleformat{\section}{\large\bfseries}{\thesection}{1em}{}

\usepackage[round]{natbib} 
\let\oldcite\cite
\renewcommand{\cite}[1]{\mbox{\oldcite{#1}}}

\newtheorem{theorem}{Theorem}
\newtheorem{lemma}{Lemma}
\newtheorem{corollary}{Corollary}
\newtheorem{proposition}{Proposition}

\theoremstyle{definition}

\newtheorem{remark}{Remark}
\newtheorem*{remark*}{Remark}

\numberwithin{equation}{section}

\title{Detecting practically significant dependencies in metric spaces via distance correlations
}
\author{Holger Dette\footnote{Faculty of Mathematics, Ruhr-Universität Bochum, \href{mailto:holger.dette@rub.de}{holger.dette@rub.de}} \and Marius Kroll\footnote{Faculty of Mathematics, Ruhr-Universität Bochum, \href{mailto:marius.kroll@rub.de}{marius.kroll@rub.de}}}
\begin{document}
\maketitle

\begin{abstract}
  We take a different look at the problem of testing the independence of two metric-space-valued random variables using the distance correlation. Instead of testing if the distance correlation vanishes exactly, we are interested in the hypothesis that it does not exceed a certain threshold. Our testing problem is motivated by the observation that in many cases it is more reasonable to test for a practically significant dependency since it is rare that a hypothesis of perfect independence is exactly satisfied. This point of view  also reflects statistical practice, where one often classifies the strength of the association in categories such as `small', `medium' and `large' and the precise definitions depend on the specific application. To address these problems we develop a pivotal test for the hypothesis that the distance correlation between two random variables does not exceed a pre-specified threshold $\Delta$. We also determine a minimum value $\hat \Delta_\alpha$ from the data such that the hypothesis is rejected for all $\Delta \leq \hat \Delta_\alpha$ at controlled type I error $\alpha$. This quantity can be interpreted as a measure of evidence against the hypothesis that the distance correlation is less or equal than $\Delta$. The new test is applicable to processes taking values in separable metric spaces of strong negative type, covering Euclidean as well as functional data. We do not assume independent observations, and instead prove our results for absolutely regular sample generating processes, which includes many time series such as ARMA and GARCH models. Our approach is based on a new functional limit theorem for the sequential distance correlation process, and can also be used to construct confidence intervals for the distance correlation without the need for resampling.
  \\[5mm]

\end{abstract}

\section{Introduction}
\label{sec:introduction}

One of the fundamental problems in statistics is to  measure and  test for a  statistical association between two quantities of interest. Besides the popular Pearson’s correlation coefficient \citep{Pearson1920}, Spearman’s  $\rho $ \citep{spearman:1904} and Kendall’s $\tau$ \citep{Kendall1938} numerous other measures have been proposed in the literature for two-dimensional data, see \cite{Hoeffding1948} and \cite{blum1961} for some further early references as well as \cite{gretton2008}, \cite{Heller2012}, \cite{Dette.2012}, \cite{bergsma2014}, \cite{Albert2014}, \cite{Geenens.2020} and \cite{ Chatterjee.2020} for recent references. Besides these and other measures of dependence the {\it distance covariance} and {\it distance correlation} have found considerable interest since their original introduction by \cite{szekely_rizzo_bakirov:2007} and \cite{szekely:2009, szekely:2012}. These authors discussed the distance covariance in the context of Euclidean data and \cite{lyons:2013} generalised the theory to separable metric spaces.

The distance covariance proves exceedingly useful as a tool in  statistics. Even the early work by \cite{szekely_rizzo_bakirov:2007} includes critical values for a test of independence based on the distance covariance. \cite{szekely:2013} construct a $t$-test for independence of high-dimensional data; \cite{dehling_et_al:2020} test for independence of time-continuous empirical processes; \cite{davis_et_al:2018} consider the distance covariance in the context of time series, resulting in a goodness-of-fit test; \cite{fokianos:2017} develop a Box-Ljung-type test for serial independence; \cite{zhou:2012} defines  the auto  distance covariance and correlation for time series under the assumption of physical dependence; \cite{betkendehling:2022} and \cite{betkendehlingkroll:2022}  propose independence tests for long-range and  short-range dependent time series. The importance of the distance correlation  for measuring dependencies between non-Euclidean data becomes  apparent from the work \cite{sejdinovic:2013} who  prove that the distance covariance and the Hilbert-Schmidt independence criterion (HSIC) as introduced by \cite{gretton:2005} are in a certain sense equivalent. This   provides a new point of view on  the rich literature in this field. For some examples we refer to the work of \cite{gretton:2005,gretton2008} as well as \cite{gretton_györfi:2010}, who construct HSIC-independence-tests; \cite{zhang_et_al:2008} who apply the HSIC to non-i.i.d.\@ data; \cite{wang_et_al:2021} who test for independence between time series; \cite{sen_sen:2014} who propose a goodness-of-fit test  for  linear models. A literature survey of the distance covariance is given in \cite{edelman_et_al:2019} and a recent article on independence testing via distance covariance is \cite{dettetang2024},  who propose a distance covariance (and a corresponding independence test) which does not require the existence of moments.

These papers differ in several aspects, such as the type of data, the kind of serial dependence that is assumed (if any), the test statistics
and methods to obtain critical values. However, a common feature of all cited references that construct independence tests consists in the fact that these are proposed for the hypothesis of exact independence, which is  characterized by a vanishing distance covariance or correlation. We argue that there are many applications where it is in fact clear from the scientific problem that the association between two quantities might be small but cannot be exactly $0$ (which corresponds to complete independence), and one is in fact interested in testing for a practically significant deviation from independence.
For example, \cite{yaoetal2017} considered the problem
of  the mutual independence of a part of the components of a high-dimensional vector using sums of the corresponding pairwise empirical distance covariances. Their test rejects the null hypothesis of vanishing distance covariances, although the authors observe  that   the dependencies are relatively weak, see   Section 7 in in this reference.

In the present paper we therefore take a different point of view on the problem of testing independence
of metric-space-valued data via the distance covariance $\mathrm{dcov}(X,Y) $ and correlation $\mathrm{dcor}(X,Y)  $ between two random vectors, say  $X$ and $Y$ (the precise definition of the distance correlation will be given in equation \eqref{det1} in Section \ref{sec2}).
Instead of testing for exact independence we propose to test the hypothesis that the distance correlation (covariance) does not exceed a given threshold, say $\Delta$, that is
\begin{equation}
  \label{det2}
  H_0^{\rm rel}:
  \mathrm{dcor}(X,Y)  \leq \Delta
  ~~ \text{versus} ~~ H_
  1^{\rm rel}:  \mathrm{dcor}(X,Y)  > \Delta ~.
\end{equation}
Note that these  hypotheses include the ``classical'' null hypothesis of independence, which is obtained by the choice $\Delta=0$. However in this paper we focus on the case
$\Delta >0$, assuming that it  is clear from the scientific context that there is some (potentially small)  association between $X$ and $Y$, and that the real question is if  this association is practically significant.
This point of  view is in line with \cite{tukey1991}, who  argues in the context of  comparisons of means that {\it  `all we know about the world teaches us that the
    effects of A and B are always different  -- in some
    decimal place -- for any A and B.  Thus asking ``Are
    the effects different?'' is foolish.'}  Translated to the problem of independence testing this means that it might not be reasonable to ask if $X$ and $Y$ are dependent (as there usually exists a dependence but its  strength might be very small).

Developing a statistical tests for the hypotheses \eqref{det2} is non-trivial as one has to control the level under a composite null hypothesis. Consequently, standard tools such as asymptotic inference, permutation or bootstrap tests are not easily applicable for testing hypotheses of the form \eqref{det2}. For example, the asymptotic variance of the common distance correlation estimator in the case of dependent components has a very complicated structure and is hard to estimate. Similarly, for implementation of powerful bootstrap tests one has to generate data under the restriction
$\mathrm{dcor}(X,Y) = \Delta$. In the present paper
we will develop  a very simple and
pivotal test for these hypotheses. Interestingly, this test turns out to exhibit some monotonicity properties that allow us -- even in cases where it is difficult to define a threshold before the experiment -- to
identify the minimum value of $\Delta $
for which
$H_0$ in \eqref{det2} is not rejected at a controlled type I error. This value can be interpreted as a measure of evidence against the null hypothesis in \eqref{det2} (see Remark \ref{rem:visual_inspection}(a) below). As a consequence we are able to provide a data-adaptive way of choosing the threshold in \eqref{det2}.

Our approach is applicable to strictly stationary time series with components taking values in metric spaces, making it feasible to analyse the dependence between data with a potentially complex structure. We work under a short range dependent framework; specifically, we assume that the data generating process is $\beta$-mixing with a polynomial rate of decay. In particular, this covers many commonly used time series models. Besides these contributions from the statistical side this work makes also several contributions from a theoretical point of view. In particular we develop an invariance principle for the sequential process of the empirical distance covariance and correlation based on dependent random variables taking values in metric spaces, which is the basic tool for the pivotal inference proposed here.

\section{Preliminaries}
\subsection{Distance covariance and correlation in metric spaces}
\label{sec2}

Suppose that $(\mathcal{X}, d_\mathcal{X})$ and $(\mathcal{Y}, d_\mathcal{Y})$ are two separable metric spaces equipped with their Borel-$\sigma$-algebras. For an $\mathcal{X} \times \mathcal{Y}$-valued random variable $(X,Y) $ with distribution $\mathbb{P}^{X,Y}$ whose marginals $\mathbb{P}^{X}$ on $\mathcal{X}$ and $\mathbb{P}^{Y}$ on $\mathcal{Y}$ have finite first moments we define the distance covariance of $(X,Y)$ as
$$
  \mathrm{dcov}(X,Y) :=
  \mathrm{dcov}(\mathbb{P}^{X,Y}) :=
  \int h'(z_1, \ldots, z_6) ~  \mathrm{d}\mathbb{P}^{X,Y} (z_1 ) \ldots \mathbb{P}^{X,Y} (z_6),
$$
where $z_i := (x_i, y_i) \in \mathcal{X} \times \mathcal{Y}$ ($i=1, \ldots , 6$), and the kernel $h'$ is defined by
\begin{align}
  \begin{split}
    \label{eq:kern_h'}
    h'(z_1, \ldots, z_6) := & \left[ d_\mathcal{X}(x_1, x_2) - d_\mathcal{X}(x_1, x_3) - d_\mathcal{X}(x_2, x_4) + d_\mathcal{X}(x_3, x_4) \right]              \\
                            & \quad \cdot \left[ d_\mathcal{Y}(y_1, y_2) - d_\mathcal{Y}(y_1, y_5) - d_\mathcal{Y}(y_2, y_6) + d_\mathcal{Y}(y_5, y_6) \right].
  \end{split}
\end{align}
This definition follows that of \cite{lyons:2013}, which differs from that of \cite{szekely_rizzo_bakirov:2007} by a power of $2$; taking the square root of what we call $\mathrm{dcov}(X,Y)$ results in what in the Euclidean case \citep{szekely_rizzo_bakirov:2007} is called $\mathcal{V}(X,Y)$. The same holds for the distance correlation, to be introduced below. This is a purely conventional difference, but it should be kept in mind if one wishes to write code based on our methods. For instance, the commonly used R-package \texttt{energy} \citep{energy} implements the distance covariance and distance correlation in the sense of \cite{szekely_rizzo_bakirov:2007}.

A key property of the distance covariance is that it vanishes if and only if $X$ and $Y$ are independent: $\mathrm{dcov}(X,Y) = 0$ if and only if $\mathbb{P}^{X,Y} = \mathbb{P}^{X}\otimes \mathbb{P}^{Y}$. This property holds provided that the underlying metric spaces $(\mathcal{X}, d_\mathcal{X})$ and $(\mathcal{Y}, d_\mathcal{Y})$ are of so-called strong negative type. The precise definition is technical and can be found in Section 3 of \cite{lyons:2013}. For the purposes of this article, it is enough to point out that many spaces which are useful in practice are of strong negative type. For instance, separable Hilbert spaces \citep[Theorem 3.16 in ][]{lyons:2013}, separable ultrametric spaces \citep{timan_vestfried:1983} and hyperbolic spaces \citep{lyons:2014} are all of strong negative type. There is also the broader class of spaces of negative type (without the `strong'), and these can be turned into spaces of strong negative type by raising their metrics to the $r$-th power for some $0 < r < 1$ \citep[Remark 3.19 in][]{lyons:2013}. Spaces of negative type include $L_p$ spaces for $1 < p \leq 2$, and the trick of raising the metric to some $0 < r < 1$ also has the nice side effect of resolving potential issues if $X$ and $Y$ do not satisfy the necessary moment assumptions. We refer to \cite{lyons:2013} for more information about the connection between the distance covariance and spaces of (strong) negative type.

A closely related quantity to  $\mathrm{dcov}(X,Y)$ is the distance correlation between two random variables $X$ and $Y$, which  is defined as
$$
  \mathrm{dcor}(X,Y) := \mathrm{dcor}(\mathbb{P}^{X,Y}) :=  \frac{\mathrm{dcov}(X,Y)}{\sqrt{\mathrm{dcov(X,X) \, \mathrm{dcov}(Y,Y)}}},
$$
whenever  $\mathrm{dcov}(X,X), \mathrm{dcov}(Y,Y) > 0$. By \cite{lyons:2013}, this is the case if neither $X$ nor $Y$ are almost surely constant. If one of the two random variables is almost surely constant, then the distance correlation is set to $0$. It is obvious that the characterisation of independence via the distance covariance carries over to the distance correlation, i.e.\@ $\mathrm{dcor}(X,Y) = 0$ if and only if $X$ and $Y$ are independent. Furthermore, the distance correlation is bounded by $1$, with equality if and only if there exists a continuous function $f : (\mathcal{X}, d_\mathcal{X}) \to (\mathcal{Y}, d_\mathcal{Y})$ such that $Y = f(X)$ almost surely and $d_\mathcal{X}(x,x') = \mathrm{const} \cdot d_\mathcal{Y}(f(x), f(x'))$ for all $x,x'$ in the support of the distribution of $X$, which we denote by $\mathcal{L}(X)$. If $\mathcal{X}$ and $\mathcal{Y}$ are Euclidean spaces, the function $f$ has the form $f(x) = a + \beta Cx$ for some vector $a$, a non-zero constant $\beta \in \mathbb{R}$ and an orthogonal matrix $C$. See \cite{lyons:2013} for the general result and \cite{szekely_rizzo_bakirov:2007} for the Euclidean special case.

Given a sample $(X_1, Y_1), \ldots, (X_n, Y_n)$ from a strictly stationary process $(X_k, Y_k)_{k \in \mathbb{N}}$, a natural estimator for the (usually unknown) value $\mathrm{dcov}(X,Y)$ is the empirical distance covariance
\begin{align}
  \label{det1}
  \mathrm{dcov}_n(X,Y) =\mathrm{dcov}(\hat{\mathbb{P}}_n) = \int h' ~\mathrm{d}\hat{\mathbb{P}}_n^6 = n^{-6} \sum_{1 \leq i_1, \ldots, i_6 \leq n} h'[(X_{i_1}, Y_{i_1}), \ldots, (X_{i_6}, Y_{i_6})],
\end{align}
where $\hat{\mathbb{P}}_n$ denotes the empirical measure of $(X_1, Y_1), \ldots, (X_n, Y_n)$, i.e.\@ the measure defined by $\hat{\mathbb{P}}_n(A) = n^{-1} \sum_{i=1}^n \textbf{1}\{(X_i, Y_i) \in A\}$, and $h'$ is the kernel from Eq.\@ \eqref{eq:kern_h'}. Under the assumption of ergodicity and finite first moments, it holds that this estimator converges almost surely to the true value of the distance covariance,
$$
  \mathrm{dcov}_n(X,Y)  \xrightarrow[n \to \infty]{a.s.}\mathrm{dcov}(X,Y).
$$
In this generality, this result can be found in \cite{kroll:2022}, though earlier results of this kind are given by \cite{lyons:2013} and \cite{janson:2021} under the assumption of i.i.d.\@ data in metric spaces and by \cite{davis_et_al:2018} for strictly stationary and mixing data in Euclidean space.

By definition, the empirical distance covariance is a $V$-statistic of order $6$ and kernel  $h'$ as defined in Eq.\@ \eqref{eq:kern_h'}. Thus, using asymptotic theory for $V$-statistics it can be  shown \citep{kroll:2023} that under some technical assumptions (most importantly, $(4+\varepsilon)$-moments and absolute regularity), it holds that
\begin{align}
  \label{det6}
  \sqrt{n}(\mathrm{dcov}_n(X,Y)   -\mathrm{dcov}(X,Y) ) \xrightarrow[n \to \infty]{\mathcal{D}} \mathcal{N}\left(0, \sigma^2\right)
\end{align}
with $\sigma^2 > 0$ if and only if $\mathrm{dcov}(\mathbb{P}^{X,Y}) > 0$ (otherwise, the rate of convergence is of order $1/n$).  A similar result will be derived in the present paper for a corresponding estimator, say  $\mathrm{dcor}_n(X,Y) $, for the distance correlation. In principle such results can be used to construct an asymptotic level $\alpha$ test for the hypotheses \eqref{det2} by rejecting $H_0$ for large values of $\mathrm{dcor}_n(X,Y) $. However such a test would require a precise estimation of the asymptotic variance
which is difficult for infinite dimensional dependent data. In the present paper we circumvent these problems and construct a pivotal test for the hypotheses \eqref{det2} using a sequential version of $\mathrm{dcor}_n(X,Y) $.

\subsection{Absolutely regular processes}
Throughout this article, we operate under the assumption that the $(X_1, Y_1), \ldots, (X_n, Y_n)$ are not necessarily i.i.d.\@, but may instead display some sort of serial dependence. We assume that the data generating process $(X_k, Y_k)_{k \in \mathbb{N}}$ is strictly stationary and $\beta$-mixing. $\beta$-mixing, also known as absolute regularity, falls in the regime of so-called short range dependence, in which the dependence between two observations $(X_k, Y_k)$ and $(Y_{k+l}, Y_{k+l})$ decreases as the lag $l$ increases. On a qualitative level, this behaviour is present in many time series describing natural phenomena. On a more technical level, the $\beta$-mixing coefficient of two $\sigma$-algebras $\mathcal{A}$ and $\mathcal{B}$ is defined as
$$
  \beta(\mathcal{A}, \mathcal{B}) = \sup \frac{1}{2}\sum_{i=1}^I \sum_{j=1}^J \left|\mathbb{P}(A_i \cap B_j) - \mathbb{P}(A_i) \mathbb{P}(B_j)\right|,
$$
where the supremum is taken over all finite partitions $A_1, \ldots, A_I \in \mathcal{A}$ and $B_1, \ldots, B_J \in \mathcal{B}$, $I,J \in \mathbb{N}$. A stochastic process $(U_k)_{k \in \mathbb{N}}$ is called $\beta$-mixing or absolutely regular if
$$
  \beta(n) = \sup_{j \in \mathbb{N}} \beta\left(\mathcal{F}_1^j, \mathcal{F}_{j+n}^\infty\right) \xrightarrow[n \to \infty]{} 0,
$$
where $\mathcal{F}_i^j$ denotes the $\sigma$-algebra generated by $U_i, \ldots, U_j$. The $\beta$-mixing condition was originally introduced by \cite{volkonskii_rozanov:1959, volkonskii_rozanov:1961}; for a detailed compilation of results on $\beta$-mixing sequences, especially in comparison with other mixing conditions such as $\alpha$-, $\phi$- or $\psi$-mixing, we refer to \cite{bradley:2007}, particularly Chapter 3. For the purposes of this article, it is enough to note that the $\beta$-mixing assumption is relatively mild, and that many important time series models are $\beta$-mixing. Time series models for which the $\beta$-mixing condition has been verified include ARMA models \citep{mokkadem:1988} and GARCH models \citep{carrasco_chen:2002}, but also more elementary models such as irreducible aperiodic countable-state Markov chains \citep[Theorem 7.7 in][]{bradley:2007}. Trivially, i.i.d.\@ sequences are also $\beta$-mixing. \cite{doukhan:mixing} contains more examples and alternative characterizations of different mixing conditions, including $\beta$-mixing. Additionally, $\beta$-mixing coefficients can be estimated even in the absence of model assumptions \citep{mcdonald_et_al:2015}.

\section{Pivotal inference for distance covariance and correlation }
\label{sec:mains}
Suppose that $(X_k, Y_k)$, $1 \leq k \leq n$, are copies of some generic random vector $(X,Y) \in \mathcal{X} \times \mathcal{Y}$, where $\mathcal{X}$ and $\mathcal{Y}$ are separable metric spaces of strong negative type. For any $n \in \mathbb{N}$, let $\mathrm{dcov}_n(X,Y)$ and $\mathrm{dcor}_n(X,Y)$ denote the empirical distance covariance and distance correlation, respectively, based on the sample $(X_1, Y_1), \ldots, (X_n, Y_n)$.
Recalling the definition of $\mathrm{dcov}_n(X,Y) $ in \eqref{det1} we define for $\lambda \in [0,1]$
\begin{align}
  \label{det4}
  \mathrm{dcov}_{\lfloor n\lambda\rfloor}(X,Y) :=\mathrm{dcov}(\hat{\mathbb{P}}_{\lfloor n\lambda\rfloor}),
\end{align}
as the sequential  empirical distance covariance of the sample $(X_1, Y_1), \ldots, (X_{\lfloor n\lambda\rfloor}, Y_{\lfloor n\lambda\rfloor})$, where $\hat{\mathbb{P}}_{\lfloor n\lambda\rfloor}$ is  the empirical measure of $(X_1, Y_1), \ldots, (X_{\lfloor n\lambda\rfloor}, Y_{\lfloor n\lambda\rfloor})$. $\lfloor \cdot \rfloor$ denotes the floor function, i.e.\@ for $\alpha \in \mathbb{R}$, $\lfloor \alpha \rfloor$ is the largest integer not exceeding $\alpha$. Similarly, we define  for $\lambda \in [0,1]$ the measures $\hat{\mathbb{P}}_{\lfloor n\lambda\rfloor}^{X,X}$ and  $\hat{\mathbb{P}}_{\lfloor n\lambda\rfloor}^{Y,Y}$
as the empirical distributions of the samples $(X_1, X_1), \ldots, (X_{\lfloor n\lambda\rfloor}, X_{\lfloor n\lambda\rfloor})$ and $(Y_1, Y_1), \ldots, (Y_{\lfloor n\lambda\rfloor}, Y_{\lfloor n\lambda\rfloor})$, respectively, and consider the corresponding distance variances
\begin{align}
  \label{det4a}
  \mathrm{dcov}_{\lfloor n\lambda\rfloor}(X, X) :=\mathrm{dcov}(\hat{\mathbb{P}}_{\lfloor n\lambda\rfloor}^{X,X}), \quad
  \mathrm{dcov}_{\lfloor n\lambda\rfloor}(Y, Y) :=\mathrm{dcov}(\hat{\mathbb{P}}_{\lfloor n\lambda\rfloor}^{Y,Y}).
\end{align}
The main theoretical tool  for the construction of a pivotal test for the hypotheses \eqref{det2} is the following limiting result, which gives the weak convergence of the $3$-dimensional sequential process defined by \eqref{det4} and \eqref{det4a}. The space $\ell^\infty[0,1]$ denotes the space of all bounded functions defined on $[0,1]$, equipped with the supremum norm, and the symbol $\rightsquigarrow$ denotes weak convergence in the sense of \cite{wellner_vandervaart:1996}.

\begin{theorem}
  \label{thm:dcov_prozesskonvergenz}
  For each $\lambda \in [0,1]$ define the random vector
  $$
    X_n(\lambda) := \frac{\lfloor n \lambda \rfloor}{n}
    \begin{pmatrix}
      \mathrm{dcov}_{\lfloor n \lambda \rfloor}(X,Y) \\ \mathrm{dcov}_{\lfloor n \lambda\rfloor}(X,X) \\ \mathrm{dcov}_{\lfloor n \lambda\rfloor}(Y,Y)
    \end{pmatrix}, \qquad X(\lambda) := \lambda
    \begin{pmatrix}
      \mathrm{dcov}(X,Y) \\ \mathrm{dcov}(X,X) \\ \mathrm{dcov}(Y,Y)
    \end{pmatrix},
  $$
  and consider the processes $X_n := (X_n(\lambda))_{0 \leq \lambda \leq 1}$ and $X := (X(\lambda))_{0 \leq \lambda \leq 1}$. Suppose that $X$ and $Y$ are not almost surely constant and have finite $(4+\varepsilon)$-moments for some $\varepsilon > 0$. Assume further that the sample generating process $(X_k, Y_k)_{k \in \mathbb{N}}$ is strictly stationary and absolutely regular with mixing rate $\beta(n) = \mathcal{O}\left(n^{-r}\right)$ for some $r > 6 + 24/\varepsilon$. Then it holds that
  $$
    \sqrt{n}(X_n - X) \rightsquigarrow \Gamma^\frac{1}{2} W
  $$
  in $(\ell^\infty[0,1])^3$, where $W = (B_1, B_2, B_3)^\top$ for three independent Brownian motions $B_1, B_2, B_3$ (restricted to the unit interval). The matrix $\Gamma$ is a covariance matrix depending on the distribution of the sample generating process and has the following property: If $\mathrm{dcov}(X,Y) = 0$, then the first row and the first column of $\Gamma^{1/2}$ consist entirely of $0$-entries.
\end{theorem}

The most important part of Theorem \ref{thm:dcov_prozesskonvergenz} is the sequential limit theorem for the distance covariance, which can be obtained by projecting $X_n$ and $X$ onto their first coordinates. A similar result can be obtained for the distance correlation. For this, define
\begin{align}
  \label{det7}
  \mathrm{dcor}_{\lfloor n\lambda\rfloor}(X,Y) & := \mathrm{dcor} (\hat{\mathbb{P}}_{\lfloor n\lambda\rfloor})
\end{align}
as the sequential  empirical distance correlation of the sample $(X_1, Y_1), \ldots, (X_{\lfloor n\lambda\rfloor}, Y_{\lfloor n\lambda\rfloor})$.

\begin{corollary}
  \label{cor:dcor_processkonvergenz}
  Fix some $r > 1/2$ and define the processes $Z_n, Z \in \ell^\infty[0,1]$ by
  $$
    Z_n(\lambda) = \lambda^r \mathrm{dcor}_{\lfloor n\lambda\rfloor}(X,Y)
  $$
  and $Z(\lambda) = \lambda^r \mathrm{dcor}(X,Y)$. Then, under the assumptions of Theorem \ref{thm:dcov_prozesskonvergenz}, it holds that
  $$
    \sqrt{n}(Z_n - Z) \rightsquigarrow L
  $$
  in $\ell^\infty[0,1]$, where the limit process $L$ is defined by $L(\lambda) = \lambda^{r-1} \langle v, \Gamma^{1/2}W(\lambda)\rangle$, $\Gamma^{1/2}W$ is the limit process from Theorem \ref{thm:dcov_prozesskonvergenz}, and
  $$
    v = \begin{pmatrix}
      \frac{1}{\sqrt{\mathrm{dcov}(X,X) \mathrm{dcov}(Y,Y)}} \\ - \frac{\mathrm{dcov}(X,Y)}{2\sqrt{\mathrm{dcov}(X,X)^3\mathrm{dcov}(Y,Y)}} \\ - \frac{\mathrm{dcov}(X,Y)}{2\sqrt{\mathrm{dcov}(X,X)\mathrm{dcov}(Y,Y)^3}}
    \end{pmatrix}.
  $$
  In particular, $L$ is tight, mean-zero and Gaussian.
\end{corollary}
The fact that the limiting process $L$ is an element of $\ell^\infty[0,1]$ is a consequence of the law of the iterated logarithm for the time-inversed Brownian motion, which states that the sample paths of a Brownian motion $B = (B(\lambda))_{\lambda \geq 0}$ are almost surely asymptotically bounded by  $\sqrt{\lambda \log (1/\lambda)}$ for $\lambda \downarrow 0$ \citep[see, for instance, Theorem 2.8.1 in][]{shorack_wellner:empirical_processes}.

Next we introduce a normalizing factor for the statistic
$\mathrm{dcov}_{ n}(X,Y) $, which is defined by
\begin{align*}
  V_{n,\mathrm{dcov}} & := \left\{\int \left(\lambda^2 \mathrm{dcov}_{\lfloor n\lambda\rfloor}(X,Y) - \lambda^2 \mathrm{dcov}_n(X,Y)\right)^2 ~\mathrm{d}\gamma(\lambda)\right\}^\frac{1}{2},
\end{align*}
where $\gamma$ is some arbitrary but fixed finite measure on the unit interval. Similarly, define
\begin{align}
  V_{n,\mathrm{dcor}} & := \left\{\int \left(\lambda^2 \mathrm{dcor}_{\lfloor n\lambda\rfloor}(X,Y) - \lambda^2 \mathrm{dcor}_n(X,Y)\right)^2 ~\mathrm{d}\gamma(\lambda)\right\}^\frac{1}{2}
  \label{hol100}
\end{align}
as a normalizing factor for the statistic $ \mathrm{dcor}_{ n }(X,Y)$.

\begin{theorem}
  \label{thm:schwache_konvergenz_dcov}
  Under the assumptions of Theorem \ref{thm:dcov_prozesskonvergenz}, there exists a constant $\sigma > 0$ and a standard Brownian motion $(B(\lambda))_{\lambda \geq 0}$ such that
  $$
    \sqrt{n}\begin{pmatrix}
      \mathrm{dcov}_n(X,Y) - \mathrm{dcov}(X,Y) \\ V_{n,\mathrm{dcov}}
    \end{pmatrix}
    \xrightarrow[n \to \infty]{\mathcal{D}}
    \sigma \begin{pmatrix}
      B(1) \\ \left\{\int \lambda^2 (B(\lambda) - \lambda B(1))^2 ~\mathrm{d} \gamma(\lambda)\right\}^\frac{1}{2}
    \end{pmatrix}.
  $$
\end{theorem}

\begin{theorem}
  \label{thm:schwache_konvergenz_dcor}
  Under the assumptions of Theorem \ref{thm:dcov_prozesskonvergenz}, there exists a constant $\tau \geq 0$ and a standard Brownian motion $(B(\lambda))_{\lambda \geq 0}$ such that
  $$
    \sqrt{n}\begin{pmatrix}
      \mathrm{dcor}_n(X,Y) - \mathrm{dcor}(X,Y) \\ V_{n,\mathrm{dcor}}
    \end{pmatrix}
    \xrightarrow[n \to \infty]{\mathcal{D}}
    \tau \begin{pmatrix}
      B(1) \\ \left\{\int \lambda^2 (B(\lambda) - \lambda B(1))^2 ~\mathrm{d} \gamma(\lambda)\right\}^\frac{1}{2}
    \end{pmatrix}.
  $$
\end{theorem}

These results can be used to derive a test for the hypothesis \eqref{det2}
of a practically significant deviation from independence measured by distance correlation or distance covariance. To be precise, let $w_{1-\alpha}$  denote the $(1-\alpha)$-quantile of the distribution of the random variable
\begin{equation}
  \label{eq:definition_W}
  W := \frac{B(1)}{\left\{\int \lambda^2 (B(\lambda) - \lambda B(1))^2 ~\mathrm{d} \gamma(\lambda)\right\}^\frac{1}{2}},
\end{equation}
then, if $\tau >0$,  it follows from Theorem \ref{thm:schwache_konvergenz_dcor} and the continuous mapping theorem that
\begin{align}
  \label{det8}
  \frac{\mathrm{dcor}_n(X,Y) - \mathrm{dcor}(X,Y)}{V_{n,\mathrm{dcor}}}
  \xrightarrow[n \to \infty]{\mathcal{D}}  W.
\end{align}
Therefore, we propose  to reject
the null hypothesis in \eqref{det2}, whenever
\begin{align}
  \label{det5}
  \frac{ \mathrm{dcor}_n(X,Y) - \Delta}{V_{n,\mathrm{dcor}}} > w_{1-\alpha}~.
\end{align}
The following result shows that this decision rule defines an  asymptotic level $\alpha$ test for the hypotheses \eqref{det2}, which is consistent against all alternatives.

\begin{proposition}
  \label{cor:test_dcov}
  Let $\Delta > 0$ be fixed and assume that the assumptions of Theorem \ref{thm:schwache_konvergenz_dcov} are satisfied. If the limiting variance from Theorem \ref{thm:schwache_konvergenz_dcor} satisfies $\tau > 0$, then the test defined by \eqref{det5} has asymptotic level $\alpha$ and is consistent. More precisely, we have
  $$
    \lim_{n \to \infty} \mathbb{P}\left(\frac{ \mathrm{dcor}_n(X,Y) - \Delta}{V_{n,\mathrm{dcor}}} > w_{1-\alpha}\right) = \begin{cases}0 \qquad &\textrm{if } \mathrm{dcor}(X,Y) < \Delta, \\ \alpha \qquad &\textrm{if } \mathrm{dcor}(X,Y) = \Delta, \\ 1 \qquad &\textrm{if } \mathrm{dcor}(X,Y) > \Delta.\end{cases}
  $$
\end{proposition}

\begin{remark} ~
  \label{rem2}
  {\rm
    It is possible for the limiting variance $\tau$ from Theorem \ref{thm:schwache_konvergenz_dcor} to be $0$. For instance, if $Y = f(X)$ almost surely for some isometric function $f$, it holds that $\mathrm{dcor}_n(X,Y) - \mathrm{dcor}(X,Y) = 1-1= 0$ for almost all $n$ almost surely, and hence $\tau = 0$. While we suspect that in most `real' data examples, the limiting variance $\tau$ will be strictly greater than $0$, this can also be checked by the practitioner directly. For this, they can use the fact that, with $B$ denoting a Brownian motion,
    \begin{equation}
      \label{eq:visual_inspection}
      \sqrt{n}\left[\lambda^2 \mathrm{dcor}_{\lfloor n\lambda\rfloor}(X,Y) - \lambda^2  \mathrm{dcor}_n(X,Y) \right]_{0 \leq \lambda \leq 1} \rightsquigarrow \tau \left[\lambda B(\lambda) - \lambda^2B(1)\right]_{0 \leq \lambda \leq 1}
    \end{equation}
    in $\ell^\infty[0,1]$. We establish this convergence in the proof of Theorem \ref{thm:schwache_konvergenz_dcor}. To be sure that the condition $\tau > 0$ is satisfied, one can therefore first perform a visual inspection of the sample path of the process on the left-hand side. If it is not (approximately) constant, we can be confident that $\tau > 0$, and then proceed with the proposed test. One advantage of this method is that the left-hand side is exactly the integrand in $V_{n,\mathrm{dcor}}$, and so the added computational cost is minimal. This technique  works very reliably in simulations, and   an illustration of it is given in Figure \ref{fig:visual_inspection}. The data in that figure were simulated according to a multivariate autoregressive model; the precise model will be introduced later in Section \ref{sec:finite_sample}, Eq.\@ \eqref{eq:var_model}.}

\end{remark}

\begin{figure}[t]
  \fontsize{8}{10}\selectfont

  \begingroup%
  \makeatletter%
  \providecommand\color[2][]{%
    \errmessage{(Inkscape) Color is used for the text in Inkscape, but the package 'color.sty' is not loaded}%
    \renewcommand\color[2][]{}%
  }%
  \providecommand\transparent[1]{%
    \errmessage{(Inkscape) Transparency is used (non-zero) for the text in Inkscape, but the package 'transparent.sty' is not loaded}%
    \renewcommand\transparent[1]{}%
  }%
  \providecommand\rotatebox[2]{#2}%
  \newcommand*\fsize{\dimexpr\f@size pt\relax}%
  \newcommand*\lineheight[1]{\fontsize{\fsize}{#1\fsize}\selectfont}%
  \ifx\svgwidth\undefined%
    \setlength{\unitlength}{\textwidth}%
    \ifx\svgscale\undefined%
      \relax%
    \else%
      \setlength{\unitlength}{\unitlength * \real{\svgscale}}%
    \fi%
  \else%
    \setlength{\unitlength}{\svgwidth}%
  \fi%
  \global\let\svgwidth\undefined%
  \global\let\svgscale\undefined%
  \makeatother%
  \begin{picture}(1,0.45454545)%
    \lineheight{1}%
    \setlength\tabcolsep{0pt}%
    \put(0,0){\includegraphics[width=\unitlength,page=1]{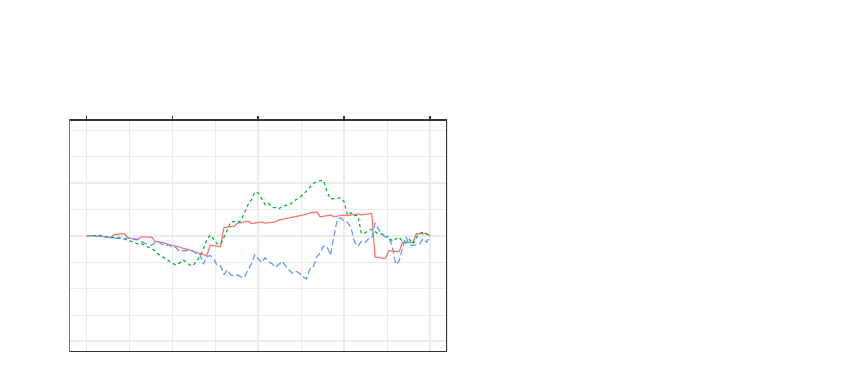}}%
    \put(0.10067272,0.32399999){\color[rgb]{0.30196078,0.30196078,0.30196078}\makebox(0,0)[t]{\lineheight{1.25}\smash{\begin{tabular}[t]{c}0.00\end{tabular}}}}%
    \put(0.20061817,0.32399999){\color[rgb]{0.30196078,0.30196078,0.30196078}\makebox(0,0)[t]{\lineheight{1.25}\smash{\begin{tabular}[t]{c}0.25\end{tabular}}}}%
    \put(0.30054546,0.32399999){\color[rgb]{0.30196078,0.30196078,0.30196078}\makebox(0,0)[t]{\lineheight{1.25}\smash{\begin{tabular}[t]{c}0.50\end{tabular}}}}%
    \put(0.40049092,0.32399999){\color[rgb]{0.30196078,0.30196078,0.30196078}\makebox(0,0)[t]{\lineheight{1.25}\smash{\begin{tabular}[t]{c}0.75\end{tabular}}}}%
    \put(0.50043634,0.32399999){\color[rgb]{0.30196078,0.30196078,0.30196078}\makebox(0,0)[t]{\lineheight{1.25}\smash{\begin{tabular}[t]{c}1.00\end{tabular}}}}%
    \put(0.07172727,0.0517091){\color[rgb]{0.30196078,0.30196078,0.30196078}\makebox(0,0)[rt]{\lineheight{1.25}\smash{\begin{tabular}[t]{r}-1.0\end{tabular}}}}%
    \put(0.07172727,0.11305453){\color[rgb]{0.30196078,0.30196078,0.30196078}\makebox(0,0)[rt]{\lineheight{1.25}\smash{\begin{tabular}[t]{r}-0.5\end{tabular}}}}%
    \put(0.07172728,0.17438182){\color[rgb]{0.30196078,0.30196078,0.30196078}\makebox(0,0)[rt]{\lineheight{1.25}\smash{\begin{tabular}[t]{r}0.0\end{tabular}}}}%
    \put(0.07172728,0.23572728){\color[rgb]{0.30196078,0.30196078,0.30196078}\makebox(0,0)[rt]{\lineheight{1.25}\smash{\begin{tabular}[t]{r}0.5\end{tabular}}}}%
    \put(0.07172728,0.29705454){\color[rgb]{0.30196078,0.30196078,0.30196078}\makebox(0,0)[rt]{\lineheight{1.25}\smash{\begin{tabular}[t]{r}1.0\end{tabular}}}}%
    \put(0,0){\includegraphics[width=\unitlength,page=2]{plot_visual_inspection_svg-tex.pdf}}%
    \put(0.0342,0.18010909){\rotatebox{90}{\makebox(0,0)[t]{\lineheight{1.25}\smash{\begin{tabular}[t]{c}Integrand value\end{tabular}}}}}%
    \put(0.52041822,0.02378182){\makebox(0,0)[rt]{\lineheight{1.25}\smash{\begin{tabular}[t]{r}$\rho = 0.5(\tau \approx 0.64)$\end{tabular}}}}%
    \put(0,0){\includegraphics[width=\unitlength,page=3]{plot_visual_inspection_svg-tex.pdf}}%
    \put(0.56032726,0.32399999){\color[rgb]{0.30196078,0.30196078,0.30196078}\makebox(0,0)[t]{\lineheight{1.25}\smash{\begin{tabular}[t]{c}0.00\end{tabular}}}}%
    \put(0.66027271,0.32399999){\color[rgb]{0.30196078,0.30196078,0.30196078}\makebox(0,0)[t]{\lineheight{1.25}\smash{\begin{tabular}[t]{c}0.25\end{tabular}}}}%
    \put(0.76019997,0.32399999){\color[rgb]{0.30196078,0.30196078,0.30196078}\makebox(0,0)[t]{\lineheight{1.25}\smash{\begin{tabular}[t]{c}0.50\end{tabular}}}}%
    \put(0.86014543,0.32399999){\color[rgb]{0.30196078,0.30196078,0.30196078}\makebox(0,0)[t]{\lineheight{1.25}\smash{\begin{tabular}[t]{c}0.75\end{tabular}}}}%
    \put(0.96009089,0.32399999){\color[rgb]{0.30196078,0.30196078,0.30196078}\makebox(0,0)[t]{\lineheight{1.25}\smash{\begin{tabular}[t]{c}1.00\end{tabular}}}}%
    \put(0.5303818,0.34427272){\makebox(0,0)[t]{\lineheight{1.25}\smash{\begin{tabular}[t]{c}$\lambda$\end{tabular}}}}%
    \put(0.98007269,0.02378182){\makebox(0,0)[rt]{\lineheight{1.25}\smash{\begin{tabular}[t]{r}$\rho = 0.99(\tau \approx 0.03)$\end{tabular}}}}%
    \put(0,0){\includegraphics[width=\unitlength,page=4]{plot_visual_inspection_svg-tex.pdf}}%
    \put(0.39914545,0.4018){\makebox(0,0)[lt]{\lineheight{1.25}\smash{\begin{tabular}[t]{l}$n$\end{tabular}}}}%
    \put(0,0){\includegraphics[width=\unitlength,page=5]{plot_visual_inspection_svg-tex.pdf}}%
    \put(0.48572726,0.40323636){\makebox(0,0)[lt]{\lineheight{1.25}\smash{\begin{tabular}[t]{l}25\end{tabular}}}}%
    \put(0.55536366,0.40323636){\makebox(0,0)[lt]{\lineheight{1.25}\smash{\begin{tabular}[t]{l}100\end{tabular}}}}%
    \put(0.63416365,0.40323636){\makebox(0,0)[lt]{\lineheight{1.25}\smash{\begin{tabular}[t]{l}400\end{tabular}}}}%
  \end{picture}%
  \endgroup%

  \caption{An illustration of the visual inspection method described in Remark \ref{rem:visual_inspection} to determine whether $\tau > 0$. The plots show the sample paths of the process on the left-hand side of Eq.\@ \eqref{eq:visual_inspection} under model \eqref{eq:var_model} for two different values of $\rho$. The two panels are examples of cases where $\tau > 0$ (left panel) and where $\tau \approx 0$ (right panel), and even at small sample sizes these two cases can clearly be distinguished. The values of $\tau$ given in the captions of the plots were obtained by Monte Carlo simulations.}
  \label{fig:visual_inspection}
\end{figure}

\begin{remark}
  \label{rem:visual_inspection}
  {\rm
    A crucial step in our approach is the choice of the threshold $\Delta$ defining the relevant hypotheses in \eqref{det2}. This choice  raises  the  question when a  distance covariance or correlation is  considered as {\it practically significant}, and the answer depends sensitively on the particular problem under consideration. For the commonly used dependence measures, this discussion  has a long history in  applied statistics and
    is  related to the specification of the effect size \citep[see, for example,][]{cohen1988}, which is often used  to obtain a better interpretation of $p$-values for comparing sample means.
    Several authors argue that such a definition should be done on a finer scale and   transfer this concept to classify for their studies the strength of the association in (two or) three categories ``small'' ($\Delta  \leq  d_\star \leq \Delta_1$), ``medium'' ($\Delta_1 <  d_\star \leq \Delta_2$) or ``large''  ($\Delta_2 <  d_\star \leq 1$), where
    $d_\star$  denotes a dependence measure taking values in the interval $[0,1]$. For example using $\Delta_2$ and $d_\star = \mathrm{dcor}(X,Y) $
    means that one is testing for a ``large''  association between $X$ and $Y$ measured with respect to the distance correlation. As pointed out before,
    the exact definition of the classes varies between different disciplines and the considered measure. Exemplarily, we refer to the recent work of \cite{Tsaparas2006}, \cite{Brydges2019} and \cite{Lovakov2021}
    who give recommendations for Pearson's correlation for studies in gerontology research and social psychology, and  to  \cite{Huang2022}, who investigate this matter for  Spearman's $\rho$ in the context of cancer mortality. Other works
    discussing these issues from different perspectives are
    \cite{Hemphill2003,Bosco2015} and \cite{Quintana2016}, and a common aspect in all these references consists in the fact that the authors define a threshold $\Delta$ for their studies, which should be exceeded to consider an association as {\it practically significant}.
    As the distance  correlation is a rather new dependence measure such thresholds have to be developed in the future for different applications. However,  for situations where such a choice is difficult,  we discuss two alternatives. \\
    \begin{itemize}
      \item[(a)] The hypotheses in \eqref{det2} are nested. By definition it is clear that  a
            rejection of $H_0$ in \eqref{det2}
            by the test \eqref{det5}
            for $\Delta= \Delta_0$ also yields rejection of $H_{0}$
            for all  $\Delta < \Delta_0$.
            By the sequential rejection principle, we may simultaneously test the  hypotheses  in \eqref{det2} for different $\Delta \geq 0$
            starting at $\Delta  = 0$ and
            increasing  $\Delta $ to
            find the minimum value of $\Delta $, say
            \begin{align*}
              \hat \Delta_\alpha
               & :=\min \big \{\Delta \ge 0 \,|  \,
              { \mathrm{dcor}_n(X,Y) } \leq   \Delta + w_{1-\alpha} {V_{n,\mathrm{dcor}}}     \big  \} \\
               & = \max \big \{ 0,
              { \mathrm{dcor}_n(X,Y) } -  w_{1-\alpha} {V_{n,\mathrm{dcor}}}
              \big \}
            \end{align*}
            for which
            $H_0$  in \eqref{det2} is not rejected.
            In particular,  as the null hypothesis
            is accepted  for all  thresholds
            $ \Delta \geq  \hat \Delta_\alpha $ and rejected for
            $ \Delta <   \hat \Delta_\alpha $, the quantity  $\hat \Delta_\alpha $ could be interpreted as a measure of evidence against the null hypothes in \eqref{det2}, i.e.  the distance correlation is less or equal than $\Delta$,    at a controlled type I error. \\
      \item[(b)] Moreover, if $\mathrm{dcor}(X,Y)>0$
            it is easy to see from \eqref{det8} that the interval
            \begin{equation}
              \label{eq:confidence_interval}
              I_n := \big [ \mathrm{dcor}_n(X,Y)  +  w_{\alpha/2} {V_{n,\mathrm{dcor}}}, ~\mathrm{dcor}_n(X,Y) +  w_{1-\alpha/2} {V_{n,\mathrm{dcor}}} \big ]
            \end{equation}
            defines an asymptotic $(1-\alpha)$-confidence interval for $\mathrm{dcor}(X,Y)$. We investigate the finite sample performance of these confidence intervals in Section \ref{sec:finite_sample}. To the best of our knowledge, this is the first non-resampling based method to obtain confidence intervals for the distance correlation (or distance covariance), and it may be of interest even for i.i.d.\@ data due to its low computational costs.
    \end{itemize}
  }
\end{remark}

\begin{remark} \label{remboot}
  \rm
  \begin{itemize}
    \item[(a)]
          In finite sample studies, which are not displayed for the sake of brevity, it can be observed that the test
          \eqref{det5} is not very sensitive with respect to the choice of the  measure $\gamma$ used in the normalizing statistic $V_{n,\mathrm{dcor}}$ in \eqref{hol100},  and we
          give a heuristic argument for this observation.
          Let us make the dependence of the denominator and the quantile of the random variable ${W}$ explicit by introducing the notation
          $$
            {V} (\gamma) = \left\{ \int_0^1 \lambda^2[{B} (\lambda) - \lambda {B} (1) ]^2  \,\mathrm{d}\gamma (\lambda) \right\}^{1/2}
          $$
          and $w_{1- \alpha} ({B}(1) / {V} (\gamma))$.
          Using these notations and Theorem \ref{thm:schwache_konvergenz_dcor}, the probability of rejection by the test \eqref{det5} can be approximated by
          \begin{align*}
             & \mathbb{P}   \left(
            \mathrm{dcor}_n(X,Y)  > \Delta  + w_{1-\alpha} V_{n,\mathrm{dcor}}  \right)
            \\
             & \quad \approx \mathbb{P}   \left(
            B (1) > {\sqrt{n} (\Delta - \mathrm{dcor}(X,Y)) \over \tau }  +
            w_{1-\alpha} \Big ({ {B} (1) \over {V} (\gamma) } \Big ) {V} (\gamma) \right).
          \end{align*}
          The right hand side is not very sensitive with respect to the measure $\gamma$, as for a fixed constant $c\in\mathbb R$, it is true that $c w_{1-\alpha} (B(1) / c) = w_{1-\alpha} (B(1))$.
    \item[(b)]   The application of resampling procedures in the context of testing relevant hypotheses is a non-trivial problem. To be precise, consider the bootstrap, and note that for its application one has to generate data under the {\bf composite} null hypothesis \eqref{det2}. Thus one has to find a distribution for $(X,Y)$ such
          $\mathrm{dcor}(X,Y)  \leq \Delta$ , which is a non-trivial task. A simple choice is to generate data under the independence assumption, but such a test would be extremely conservative with low power. In order to improve the power one would have to generate data at the `boundary  of the hypotheses', that is $\mathrm{dcor}(X,Y)  = \Delta $, but this problem seems to be intrinsically difficult. This difficulty is not a particular feature of the distance correlation, but applies to any dependence measure. Thus self-normalization, as developed here, is a useful and computationally efficient  method to address these problems.
  \end{itemize}

\end{remark}
\begin{remark}       ~~~
  {\rm \begin{enumerate}[(a)]
    \item By a similar argument to Proposition \ref{cor:test_dcov} we can construct a consistent asymptotic level $\alpha$
          test for the hypotheses
          $$
            H_0 : \mathrm{dcov}(X,Y) \leq \Delta \quad \textrm{vs.} \quad H_1 : \mathrm{dcov}(X,Y) > \Delta,
          $$
          which rejects the null hypothesis whenever $\mathrm{dcov}_n(X,Y) >  \Delta  + w_{1-\alpha} V_{n,\mathrm{dcov}}$. Since the limiting variance in Theorem \ref{thm:schwache_konvergenz_dcov} is always strictly greater than $0$, no additional assumptions are needed for this test.
    \item  Closely related to this decision problem are the hypotheses
          \begin{align}
            \label{det3}
            H_0^{\rm sim}:
            \mathrm{dcor}(X,Y)  \geq \Delta
            ~~ \text{versus} ~~ H_
            1^{\rm sim}:  \mathrm{dcor}(X,Y)  <  \Delta ~
          \end{align}
          (here the upper index `sim' indicates that a rejection of $H_0$ in \eqref{det3} means that the association is `similar' to independence.)
          Hypotheses of this type have found considerable attention in the biostatistics literature, where  the problem  is referred to as bioequivalence testing
          \citep[see, for example, the monograph of][]{wellek2010testing}.
          In the present context  rejecting \eqref{det3} allows to work under approximate independence at a controlled type I error.
          A test for these  hypotheses  can easily be developed in the same spirit. In this case the null hypothesis in $H_0^{\rm sim}$ is rejected, whenever
          \begin{align}
            \label{det33}
            \frac{ \mathrm{dcor}_n(X,Y) - \Delta}{V_{n,\mathrm{dcor}}} < w_{\alpha},
          \end{align}
          and this test has asymptotic level $\alpha$ and is consistent, that is.
          $$
            \lim_{n \to \infty} \mathbb{P}\left(\frac{ \mathrm{dcor}_n(X,Y) - \Delta}{V_{n,\mathrm{dcor}}} < w_{\alpha}\right) = \begin{cases}0 \qquad &\textrm{if } \mathrm{dcor}(X,Y) >  \Delta, \\ \alpha \qquad &\textrm{if } \mathrm{dcor}(X,Y) = \Delta, \\ 1 \qquad &\textrm{if } \mathrm{dcor}(X,Y) < \Delta.\end{cases}
          $$

          Moreover,  we can also determine a minimal value $\hat \Delta_{\alpha ,{\rm sim}}$   such that $H_0^{\rm sim}$ is rejected for all $\Delta > \hat \Delta_{\alpha ,{\rm sim}}$  at a controlled type I error $\alpha$.
  \end{enumerate}
  }
\end{remark}

\section{Behaviour Under Independence}

The results in Section \ref{sec:mains} all assume that $\mathrm{dcov}(X,Y) > 0$, or equivalently, that $X$ and $Y$ are not independent. On a conceptual level, we motivated this by the observation that nothing in the real world is ever truly independent. Nevertheless, the behaviour of our self-normalized statistic under perfect independence of $X$ and $Y$ is an interesting problem. In this case, Theorem \ref{thm:dcov_prozesskonvergenz} tells us that the first row and column of the covariance matrix $\Gamma$ and therefore also its square root $\Gamma^{1/2}$ consist only of $0$-entries. Next, the limiting process in Corollary \ref{cor:dcor_processkonvergenz} is essentially determined by the factor
$$
  \langle \nu, \Gamma^{1/2}W\rangle = \langle \Gamma^{1/2} \nu, W\rangle,
$$
where $\nu = (\nu_1, \nu_2, \nu_3)^\top$ is defined in the statement of that corollary. Importantly, under independence of $X$ and $Y$, it holds that $\nu_2 = \nu_3 = 0$ since then $\mathrm{dcov}(X,Y) = 0$. But since the first row and column of $\Gamma^{1/2}$ only consist of zeroes, this means that $\Gamma^{1/2} \nu = (0,0,0)^\top$, and so the limiting process in Corollary \ref{cor:dcor_processkonvergenz} is $0$ almost surely. This implies that the limiting variance $\tau$ in Theorem \ref{thm:schwache_konvergenz_dcor} is $0$. While this means that our self-normalizing procedure cannot be used if $X$ and $Y$ are perfectly independent, it also means that the visual inspection method from Remark \ref{rem:visual_inspection} can be used to detect whether this is the case. Therefore, even if a practitioner does not share our sentiment that perfect independence never occurs in practice, they have a practical tool to check if our self-normalization method is applicable.

Can our process convergence results be recovered under perfect independence? Perhaps a different scaling would lead to a non-degenerate covariance matrix $\Gamma$, allowing us to reconstruct analogues of Theorems \ref{thm:schwache_konvergenz_dcov} and \ref{thm:schwache_konvergenz_dcor} under independence. Unfortunately, while convergence to a non-degenerate limit can be obtained through a different scaling, the form of the limit is such that this cannot be used to construct pivotal test statistics. This is illustrated by the following result.

\begin{theorem}
  \label{thm:independence_weak_convergence}
  Define the processes $Q_n \in \ell^\infty[0,1]$, $n \in \mathbb{N}$, by
  $$
    Q_n(\lambda) = \left(\frac{\lfloor n\lambda\rfloor}{n}\right)^2 \mathrm{dcov}_{\lfloor n\lambda \rfloor}(X,Y).
  $$
  Suppose that $X$ and $Y$ are independent and have finite $(4+\varepsilon)$-moments for some $\varepsilon > 0$. Assume further that the sample generating process $(X_k,Y_k)_{k \in \mathbb{N}}$ is strictly stationary and absolutely regular with mixing rate $\beta(n) = \mathcal{O}\left(n^{-r}\right)$ for some $r > 6 + 24/\varepsilon$. Then it holds that
  \begin{equation}
    \label{eq:independence_limit}
    n Q_n \rightsquigarrow 15 \sum_{k=1}^\infty \mu_k W_k^2
  \end{equation}
  in $\ell^\infty[0,1]$, where $(\mu_k)_{k \in \mathbb{N}}$ is a sequence of non-negative numbers defined below, and $(W_k)_{k \in \mathbb{N}}$ is a sequence of centred Gaussian processes with
  \begin{equation}
    \label{eq:independence_covariance}
    \mathrm{Cov}\left[W_i(\lambda_1), W_j(\lambda_2)\right] = (\lambda_1 \land \lambda_2) \lim_{n \to \infty} \frac{1}{n} \sum_{s,t=1}^n \mathrm{Cov}\left[\varphi_i(Z_s), \varphi_j(Z_t)\right],
  \end{equation}
  where we use the notation $Z_s = (X_s, Y_s)$ for $s = 1, \ldots, n$. The objects $(\mu_k, \varphi_k)$ from Eqs.\@ \eqref{eq:independence_limit} and \eqref{eq:independence_covariance} are pairs of non-negative eigenvalues and matching eigenfunctions of the integral operator
  \begin{align*}
    T : L_2\left(\mathbb{P}^{X,Y}\right) & \to L_2\left(\mathbb{P}^{X,Y}\right),                                                                              \\
    f                                    & \mapsto \left[z \mapsto \int h_2\left(z, z'; \mathbb{P}^{X,Y}\right) f(z') ~\mathrm{d}\mathbb{P}^{X,Y}(z')\right],
  \end{align*}
  where, writing $z = (x,y) \in \mathcal{X} \times \mathcal{Y}$ and $z' = (x', y')\in \mathcal{X} \times \mathcal{Y}$,
  \begin{align}
    \begin{split}
      \label{eq:identity_h2_delta_theta}
      h_2\left(z, z'; \mathbb{P}^{X,Y}\right) & =\frac{1}{15}\left[ d_\mathcal{X}(x,x') - \int d_\mathcal{X}(x,x_0) ~\mathrm{d}\mathbb{P}^X(x_0) - \int d_\mathcal{X}(x',x_0) ~\mathrm{d}\mathbb{P}^X(x_0)\right.   \\
                                              & \phantom{=\frac{1}{15}\Big[ d_\mathcal{X}(x,x')\Big]}\left.+ \int d_\mathcal{X}(x_0,x_1) ~\mathrm{d}\left(\mathbb{P}^X \otimes \mathbb{P}^X\right) (x_0,x_1)\right] \\
                                              & \quad\cdot \left[ d_\mathcal{Y}(y,y') - \int d_\mathcal{Y}(y,y_0) ~\mathrm{d}\mathbb{P}^Y(y_0) - \int d_\mathcal{Y}(y',y_0) ~\mathrm{d}\mathbb{P}^Y(y_0)\right.     \\
                                              & \phantom{\quad\cdot \Big[ d_\mathcal{Y}(y,y')\Big]}\left.+ \int d_\mathcal{Y}(y_0,y_1) ~\mathrm{d}\left(\mathbb{P}^Y \otimes \mathbb{P}^Y\right) (y_0,y_1)\right].
    \end{split}
  \end{align}
\end{theorem}

While Theorem \ref{thm:independence_weak_convergence} only makes a statement about the process convergence of the sequential empirical distance covariance $\mathrm{dcov}_{\lfloor n\lambda\rfloor}(X,Y)$, a joint convergence result similar to Theorem \ref{thm:dcov_prozesskonvergenz}, which includes the sequential distance variances $\mathrm{dcov}_{\lfloor n\lambda\rfloor}(X,X)$ and $\mathrm{dcov}_{\lfloor n\lambda\rfloor}(Y,Y)$, could also be derived. However, this would be useless for the purposes of pivotal inference for the distance correlation, as results similar to Theorems \ref{thm:schwache_konvergenz_dcov} and \ref{thm:schwache_konvergenz_dcor} are impossible if $X$ and $Y$ are independent. The reason for this is the form of the limiting process in Eq.\@ \eqref{eq:independence_limit}, which is not a linear transformation of some Gaussian process. It depends on the sample generating process in a way which is more intricate than in the dependent case, and which makes a `canceling out' trick as in Eq.\@ \eqref{det8} impossible. Therefore, while we could in theory use Theorem \ref{thm:independence_weak_convergence} to construct test statistics similar to those in Section \ref{sec:mains}, their limiting distributions would not be pivotal. The problem ultimately lies in the fact that the empirical distance covariance is a degenerate V-statistic under perfect independence.

\section{Finite Sample Properties}
\label{sec:finite_sample}

In this section we investigate the performance of the proposed test in a variety of  models. As the reference measure $\gamma$ in the definition of the test statistic $V_{n,\mathrm{dcor}}$ we use the discrete uniform distribution concentrated on the points $j/20$, $j = 1, \ldots, 19$. The quantiles of the corresponding random variable $W$ defined in \eqref{eq:definition_W} are approximated by their empirical analogues based on $10^6$ independently sampled realisations of $W$. These quantiles are given in Table \ref{tab:quantile}. To determine the empirical rejection rates and empirical covering rates given below, we first generate data $(X_k,Y_k)$, $k = 1, \ldots, n$ according to different data generating models. Each of these models features a tuning parameter $\rho \in [0,1]$, which for the purpose of our simulation study we vary through the grid $0.00, 0.01, \ldots, 0.99$. For every different combination of data generating model, tuning parameter $\rho$ and sample size $n$, we generate $1000$ independent simulation runs from which we calculate the empirical rejection rate and empirical covering rate given below. Since the role played by the parameter $\rho$ varies between data generating models, we do not give the values of $\rho$ directly, but instead replace them by approximations of the corresponding population values of $\mathrm{dcor}(X,Y)$. In the following plots, these approximations are plotted on the $x$-axes under the label `Distance Correlation'; they were obtained by averaging $100$ independent simulations of $\mathrm{dcor}_n(X,Y)$ with $n = 1000$ for each specific combination of data generating model and tuning parameter $\rho$ \citep[$\mathrm{dcor}_n(X,Y)$ is a strongly consistent estimator of $\mathrm{dcor}(X,Y)$ as a consequence of Theorem 1 in][]{kroll:2022}. In all instances, the hypothesis tests were performed at a nominal level of $\alpha = 5\%$, and the confidence intervals were constructed at a nominal covering rate of $95\%$.

\begin{table}[b]
  \centering
  \begin{tabular}{r|c c c}
    $1-\alpha$     & $90 \%$ & 9$5 \%$ & $99 \%$ \\\hline
    $w_{1-\alpha}$ & $7.13$  & $9.89$  & $16.40$
  \end{tabular}
  \caption{\it Empirical $(1-\alpha)$-quantiles $w_{1-\alpha}$ (based on $10^6$ independent realisations) of the distribution of the random variable $W$ defined in \eqref{eq:definition_W} with $\gamma$ chosen as the discrete uniform distribution concentrated on the points  $j/20$, $j = 1, \ldots, 19$. The corresponding $\alpha$-quantiles are obtained by the symmetry of of the distribution of $W$, that is $w_{\alpha}= - w_{1-\alpha}$.}
  \label{tab:quantile}
\end{table}

\begin{figure}[t]
  \fontsize{8}{10}\selectfont

  \begingroup%
  \makeatletter%
  \providecommand\color[2][]{%
    \errmessage{(Inkscape) Color is used for the text in Inkscape, but the package 'color.sty' is not loaded}%
    \renewcommand\color[2][]{}%
  }%
  \providecommand\transparent[1]{%
    \errmessage{(Inkscape) Transparency is used (non-zero) for the text in Inkscape, but the package 'transparent.sty' is not loaded}%
    \renewcommand\transparent[1]{}%
  }%
  \providecommand\rotatebox[2]{#2}%
  \newcommand*\fsize{\dimexpr\f@size pt\relax}%
  \newcommand*\lineheight[1]{\fontsize{\fsize}{#1\fsize}\selectfont}%
  \ifx\svgwidth\undefined%
    \setlength{\unitlength}{\textwidth}%
    \ifx\svgscale\undefined%
      \relax%
    \else%
      \setlength{\unitlength}{\unitlength * \real{\svgscale}}%
    \fi%
  \else%
    \setlength{\unitlength}{\svgwidth}%
  \fi%
  \global\let\svgwidth\undefined%
  \global\let\svgscale\undefined%
  \makeatother%
  \begin{picture}(1,0.69565217)%
    \lineheight{1}%
    \setlength\tabcolsep{0pt}%
    \put(0,0){\includegraphics[width=\unitlength,page=1]{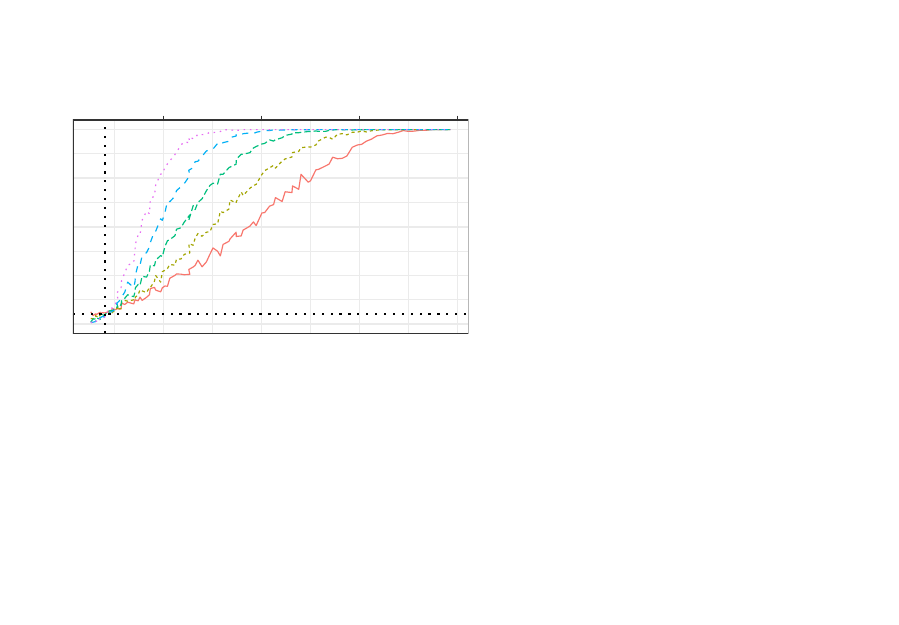}}%
    \put(0.18215652,0.5707826){\color[rgb]{0.30196078,0.30196078,0.30196078}\makebox(0,0)[t]{\lineheight{1.25}\smash{\begin{tabular}[t]{c}0.25\end{tabular}}}}%
    \put(0.29116521,0.5707826){\color[rgb]{0.30196078,0.30196078,0.30196078}\makebox(0,0)[t]{\lineheight{1.25}\smash{\begin{tabular}[t]{c}0.50\end{tabular}}}}%
    \put(0.40017392,0.5707826){\color[rgb]{0.30196078,0.30196078,0.30196078}\makebox(0,0)[t]{\lineheight{1.25}\smash{\begin{tabular}[t]{c}0.75\end{tabular}}}}%
    \put(0.5091652,0.5707826){\color[rgb]{0.30196078,0.30196078,0.30196078}\makebox(0,0)[t]{\lineheight{1.25}\smash{\begin{tabular}[t]{c}1.00\end{tabular}}}}%
    \put(0.07243478,0.32965218){\color[rgb]{0.30196078,0.30196078,0.30196078}\makebox(0,0)[rt]{\lineheight{1.25}\smash{\begin{tabular}[t]{r}0.00\end{tabular}}}}%
    \put(0.07243478,0.38372174){\color[rgb]{0.30196078,0.30196078,0.30196078}\makebox(0,0)[rt]{\lineheight{1.25}\smash{\begin{tabular}[t]{r}0.25\end{tabular}}}}%
    \put(0.07243478,0.4377913){\color[rgb]{0.30196078,0.30196078,0.30196078}\makebox(0,0)[rt]{\lineheight{1.25}\smash{\begin{tabular}[t]{r}0.50\end{tabular}}}}%
    \put(0.07243478,0.49186087){\color[rgb]{0.30196078,0.30196078,0.30196078}\makebox(0,0)[rt]{\lineheight{1.25}\smash{\begin{tabular}[t]{r}0.75\end{tabular}}}}%
    \put(0.07243478,0.54593044){\color[rgb]{0.30196078,0.30196078,0.30196078}\makebox(0,0)[rt]{\lineheight{1.25}\smash{\begin{tabular}[t]{r}1.00\end{tabular}}}}%
    \put(0,0){\includegraphics[width=\unitlength,page=2]{plot_VAR_svg-tex.pdf}}%
    \put(0.52144346,0.30384349){\makebox(0,0)[rt]{\lineheight{1.25}\smash{\begin{tabular}[t]{r}$\Delta = 0.1$\end{tabular}}}}%
    \put(0,0){\includegraphics[width=\unitlength,page=3]{plot_VAR_svg-tex.pdf}}%
    \put(0.64165219,0.5707826){\color[rgb]{0.30196078,0.30196078,0.30196078}\makebox(0,0)[t]{\lineheight{1.25}\smash{\begin{tabular}[t]{c}0.25\end{tabular}}}}%
    \put(0.75066088,0.5707826){\color[rgb]{0.30196078,0.30196078,0.30196078}\makebox(0,0)[t]{\lineheight{1.25}\smash{\begin{tabular}[t]{c}0.50\end{tabular}}}}%
    \put(0.85966956,0.5707826){\color[rgb]{0.30196078,0.30196078,0.30196078}\makebox(0,0)[t]{\lineheight{1.25}\smash{\begin{tabular}[t]{c}0.75\end{tabular}}}}%
    \put(0.96866083,0.5707826){\color[rgb]{0.30196078,0.30196078,0.30196078}\makebox(0,0)[t]{\lineheight{1.25}\smash{\begin{tabular}[t]{c}1.00\end{tabular}}}}%
    \put(0.98093909,0.30384349){\makebox(0,0)[rt]{\lineheight{1.25}\smash{\begin{tabular}[t]{r}$\Delta = 0.2$\end{tabular}}}}%
    \put(0,0){\includegraphics[width=\unitlength,page=4]{plot_VAR_svg-tex.pdf}}%
    \put(0.07243478,0.04855655){\color[rgb]{0.30196078,0.30196078,0.30196078}\makebox(0,0)[rt]{\lineheight{1.25}\smash{\begin{tabular}[t]{r}0.00\end{tabular}}}}%
    \put(0.07243478,0.1026261){\color[rgb]{0.30196078,0.30196078,0.30196078}\makebox(0,0)[rt]{\lineheight{1.25}\smash{\begin{tabular}[t]{r}0.25\end{tabular}}}}%
    \put(0.07243478,0.15669566){\color[rgb]{0.30196078,0.30196078,0.30196078}\makebox(0,0)[rt]{\lineheight{1.25}\smash{\begin{tabular}[t]{r}0.50\end{tabular}}}}%
    \put(0.07243478,0.21076522){\color[rgb]{0.30196078,0.30196078,0.30196078}\makebox(0,0)[rt]{\lineheight{1.25}\smash{\begin{tabular}[t]{r}0.75\end{tabular}}}}%
    \put(0.07243478,0.26483478){\color[rgb]{0.30196078,0.30196078,0.30196078}\makebox(0,0)[rt]{\lineheight{1.25}\smash{\begin{tabular}[t]{r}1.00\end{tabular}}}}%
    \put(0,0){\includegraphics[width=\unitlength,page=5]{plot_VAR_svg-tex.pdf}}%
    \put(0.52144346,0.0227478){\makebox(0,0)[rt]{\lineheight{1.25}\smash{\begin{tabular}[t]{r}$\Delta = 0.3$\end{tabular}}}}%
    \put(0,0){\includegraphics[width=\unitlength,page=6]{plot_VAR_svg-tex.pdf}}%
    \put(0.53097391,0.59017391){\makebox(0,0)[t]{\lineheight{1.25}\smash{\begin{tabular}[t]{c}Distance Correlation\end{tabular}}}}%
    \put(0.03271304,0.30271304){\rotatebox{90}{\makebox(0,0)[t]{\lineheight{1.25}\smash{\begin{tabular}[t]{c}Rejection Rate\end{tabular}}}}}%
    \put(0.98093909,0.0227478){\makebox(0,0)[rt]{\lineheight{1.25}\smash{\begin{tabular}[t]{r}$\Delta = 0.4$\end{tabular}}}}%
    \put(0,0){\includegraphics[width=\unitlength,page=7]{plot_VAR_svg-tex.pdf}}%
    \put(0.33445217,0.6452){\makebox(0,0)[lt]{\lineheight{1.25}\smash{\begin{tabular}[t]{l}$n$\end{tabular}}}}%
    \put(0,0){\includegraphics[width=\unitlength,page=8]{plot_VAR_svg-tex.pdf}}%
    \put(0.41726955,0.64657391){\makebox(0,0)[lt]{\lineheight{1.25}\smash{\begin{tabular}[t]{l}25\end{tabular}}}}%
    \put(0.48387828,0.64657391){\makebox(0,0)[lt]{\lineheight{1.25}\smash{\begin{tabular}[t]{l}50\end{tabular}}}}%
    \put(0.55050436,0.64657391){\makebox(0,0)[lt]{\lineheight{1.25}\smash{\begin{tabular}[t]{l}100\end{tabular}}}}%
    \put(0.62587827,0.64657391){\makebox(0,0)[lt]{\lineheight{1.25}\smash{\begin{tabular}[t]{l}200\end{tabular}}}}%
    \put(0.70123477,0.64657391){\makebox(0,0)[lt]{\lineheight{1.25}\smash{\begin{tabular}[t]{l}400\end{tabular}}}}%
  \end{picture}%
  \endgroup%

  \caption{\it Simulation results for the model \eqref{eq:var_model} with varying parameter $\rho$. The $x$-axes give the corresponding values for $\mathrm{dcor}(X,Y)$. Empirical rejection rates of the test \eqref{det5} for the hypotheses \eqref{det2} for different thresholds $\Delta$ and sample sizes $n$. The dotted lines represent $\Delta$ (vertical) and the nominal level $\alpha = 5\%$ (horizontal).}
  \label{fig:var_model}
\end{figure}

\begin{figure}[t]
  \fontsize{8}{10}\selectfont

  \begingroup%
  \makeatletter%
  \providecommand\color[2][]{%
    \errmessage{(Inkscape) Color is used for the text in Inkscape, but the package 'color.sty' is not loaded}%
    \renewcommand\color[2][]{}%
  }%
  \providecommand\transparent[1]{%
    \errmessage{(Inkscape) Transparency is used (non-zero) for the text in Inkscape, but the package 'transparent.sty' is not loaded}%
    \renewcommand\transparent[1]{}%
  }%
  \providecommand\rotatebox[2]{#2}%
  \newcommand*\fsize{\dimexpr\f@size pt\relax}%
  \newcommand*\lineheight[1]{\fontsize{\fsize}{#1\fsize}\selectfont}%
  \ifx\svgwidth\undefined%
    \setlength{\unitlength}{\textwidth}%
    \ifx\svgscale\undefined%
      \relax%
    \else%
      \setlength{\unitlength}{\unitlength * \real{\svgscale}}%
    \fi%
  \else%
    \setlength{\unitlength}{\svgwidth}%
  \fi%
  \global\let\svgwidth\undefined%
  \global\let\svgscale\undefined%
  \makeatother%
  \begin{picture}(1,0.43478261)%
    \lineheight{1}%
    \setlength\tabcolsep{0pt}%
    \put(0,0){\includegraphics[width=\unitlength,page=1]{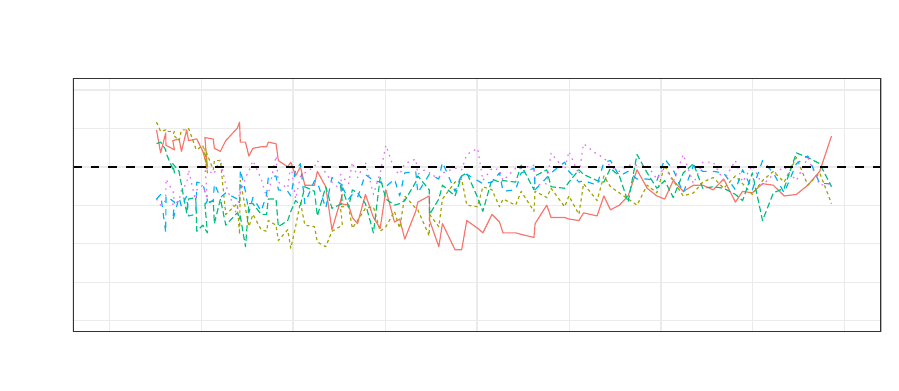}}%
    \put(0.07243478,0.07264349){\color[rgb]{0.30196078,0.30196078,0.30196078}\makebox(0,0)[rt]{\lineheight{1.25}\smash{\begin{tabular}[t]{r}0.85\end{tabular}}}}%
    \put(0.07243478,0.15817391){\color[rgb]{0.30196078,0.30196078,0.30196078}\makebox(0,0)[rt]{\lineheight{1.25}\smash{\begin{tabular}[t]{r}0.90\end{tabular}}}}%
    \put(0.07243478,0.24372174){\color[rgb]{0.30196078,0.30196078,0.30196078}\makebox(0,0)[rt]{\lineheight{1.25}\smash{\begin{tabular}[t]{r}0.95\end{tabular}}}}%
    \put(0.07243478,0.32925217){\color[rgb]{0.30196078,0.30196078,0.30196078}\makebox(0,0)[rt]{\lineheight{1.25}\smash{\begin{tabular}[t]{r}1.00\end{tabular}}}}%
    \put(0,0){\includegraphics[width=\unitlength,page=2]{plot_ci_VAR_svg-tex.pdf}}%
    \put(0.12191304,0.04575652){\color[rgb]{0.30196078,0.30196078,0.30196078}\makebox(0,0)[t]{\lineheight{1.25}\smash{\begin{tabular}[t]{c}0.00\end{tabular}}}}%
    \put(0.32645218,0.04575652){\color[rgb]{0.30196078,0.30196078,0.30196078}\makebox(0,0)[t]{\lineheight{1.25}\smash{\begin{tabular}[t]{c}0.25\end{tabular}}}}%
    \put(0.53097391,0.04575652){\color[rgb]{0.30196078,0.30196078,0.30196078}\makebox(0,0)[t]{\lineheight{1.25}\smash{\begin{tabular}[t]{c}0.50\end{tabular}}}}%
    \put(0.73551307,0.04575652){\color[rgb]{0.30196078,0.30196078,0.30196078}\makebox(0,0)[t]{\lineheight{1.25}\smash{\begin{tabular}[t]{c}0.75\end{tabular}}}}%
    \put(0.94003481,0.04575652){\color[rgb]{0.30196078,0.30196078,0.30196078}\makebox(0,0)[t]{\lineheight{1.25}\smash{\begin{tabular}[t]{c}1.00\end{tabular}}}}%
    \put(0.53097391,0.02365218){\makebox(0,0)[t]{\lineheight{1.25}\smash{\begin{tabular}[t]{c}Distance Correlation\end{tabular}}}}%
    \put(0.03271304,0.20641739){\rotatebox{90}{\makebox(0,0)[t]{\lineheight{1.25}\smash{\begin{tabular}[t]{c}Covering Rate\end{tabular}}}}}%
    \put(0,0){\includegraphics[width=\unitlength,page=3]{plot_ci_VAR_svg-tex.pdf}}%
    \put(0.33445217,0.38433043){\makebox(0,0)[lt]{\lineheight{1.25}\smash{\begin{tabular}[t]{l}$n$\end{tabular}}}}%
    \put(0,0){\includegraphics[width=\unitlength,page=4]{plot_ci_VAR_svg-tex.pdf}}%
    \put(0.41726955,0.38570435){\makebox(0,0)[lt]{\lineheight{1.25}\smash{\begin{tabular}[t]{l}25\end{tabular}}}}%
    \put(0.48387828,0.38570435){\makebox(0,0)[lt]{\lineheight{1.25}\smash{\begin{tabular}[t]{l}50\end{tabular}}}}%
    \put(0.55050436,0.38570435){\makebox(0,0)[lt]{\lineheight{1.25}\smash{\begin{tabular}[t]{l}100\end{tabular}}}}%
    \put(0.62587827,0.38570435){\makebox(0,0)[lt]{\lineheight{1.25}\smash{\begin{tabular}[t]{l}200\end{tabular}}}}%
    \put(0.70123477,0.38570435){\makebox(0,0)[lt]{\lineheight{1.25}\smash{\begin{tabular}[t]{l}400\end{tabular}}}}%
  \end{picture}%
  \endgroup%

  \caption{\it Empirical covering rates of the confidence interval \eqref{eq:confidence_interval} with $\alpha = 0.95$ and data generated according to Eq.\@ \eqref{eq:var_model}. The dashed line represents the nominal level of $95 \%$.}
  \label{fig:ci_VAR}
\end{figure}

\begin{figure}[t]
  \fontsize{8}{10}\selectfont

  \begingroup%
  \makeatletter%
  \providecommand\color[2][]{%
    \errmessage{(Inkscape) Color is used for the text in Inkscape, but the package 'color.sty' is not loaded}%
    \renewcommand\color[2][]{}%
  }%
  \providecommand\transparent[1]{%
    \errmessage{(Inkscape) Transparency is used (non-zero) for the text in Inkscape, but the package 'transparent.sty' is not loaded}%
    \renewcommand\transparent[1]{}%
  }%
  \providecommand\rotatebox[2]{#2}%
  \newcommand*\fsize{\dimexpr\f@size pt\relax}%
  \newcommand*\lineheight[1]{\fontsize{\fsize}{#1\fsize}\selectfont}%
  \ifx\svgwidth\undefined%
    \setlength{\unitlength}{\textwidth}%
    \ifx\svgscale\undefined%
      \relax%
    \else%
      \setlength{\unitlength}{\unitlength * \real{\svgscale}}%
    \fi%
  \else%
    \setlength{\unitlength}{\svgwidth}%
  \fi%
  \global\let\svgwidth\undefined%
  \global\let\svgscale\undefined%
  \makeatother%
  \begin{picture}(1,0.69565217)%
    \lineheight{1}%
    \setlength\tabcolsep{0pt}%
    \put(0,0){\includegraphics[width=\unitlength,page=1]{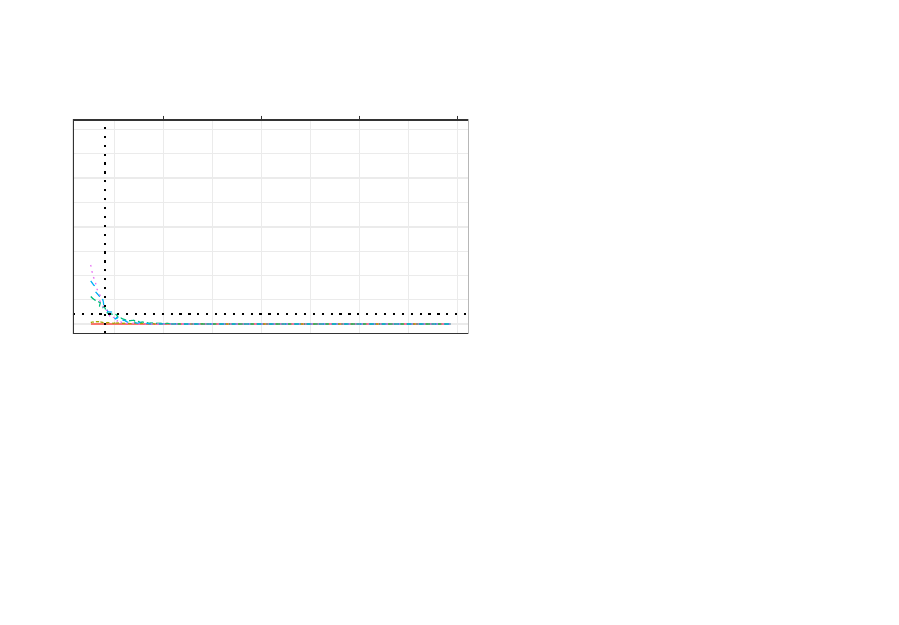}}%
    \put(0.18215652,0.5707826){\color[rgb]{0.30196078,0.30196078,0.30196078}\makebox(0,0)[t]{\lineheight{1.25}\smash{\begin{tabular}[t]{c}0.25\end{tabular}}}}%
    \put(0.29116521,0.5707826){\color[rgb]{0.30196078,0.30196078,0.30196078}\makebox(0,0)[t]{\lineheight{1.25}\smash{\begin{tabular}[t]{c}0.50\end{tabular}}}}%
    \put(0.40017392,0.5707826){\color[rgb]{0.30196078,0.30196078,0.30196078}\makebox(0,0)[t]{\lineheight{1.25}\smash{\begin{tabular}[t]{c}0.75\end{tabular}}}}%
    \put(0.5091652,0.5707826){\color[rgb]{0.30196078,0.30196078,0.30196078}\makebox(0,0)[t]{\lineheight{1.25}\smash{\begin{tabular}[t]{c}1.00\end{tabular}}}}%
    \put(0.07243478,0.32965218){\color[rgb]{0.30196078,0.30196078,0.30196078}\makebox(0,0)[rt]{\lineheight{1.25}\smash{\begin{tabular}[t]{r}0.00\end{tabular}}}}%
    \put(0.07243478,0.38372174){\color[rgb]{0.30196078,0.30196078,0.30196078}\makebox(0,0)[rt]{\lineheight{1.25}\smash{\begin{tabular}[t]{r}0.25\end{tabular}}}}%
    \put(0.07243478,0.4377913){\color[rgb]{0.30196078,0.30196078,0.30196078}\makebox(0,0)[rt]{\lineheight{1.25}\smash{\begin{tabular}[t]{r}0.50\end{tabular}}}}%
    \put(0.07243478,0.49186087){\color[rgb]{0.30196078,0.30196078,0.30196078}\makebox(0,0)[rt]{\lineheight{1.25}\smash{\begin{tabular}[t]{r}0.75\end{tabular}}}}%
    \put(0.07243478,0.54593044){\color[rgb]{0.30196078,0.30196078,0.30196078}\makebox(0,0)[rt]{\lineheight{1.25}\smash{\begin{tabular}[t]{r}1.00\end{tabular}}}}%
    \put(0,0){\includegraphics[width=\unitlength,page=2]{plot_VAR_switched_svg-tex.pdf}}%
    \put(0.52144346,0.30384349){\makebox(0,0)[rt]{\lineheight{1.25}\smash{\begin{tabular}[t]{r}$\Delta = 0.1$\end{tabular}}}}%
    \put(0,0){\includegraphics[width=\unitlength,page=3]{plot_VAR_switched_svg-tex.pdf}}%
    \put(0.64165219,0.5707826){\color[rgb]{0.30196078,0.30196078,0.30196078}\makebox(0,0)[t]{\lineheight{1.25}\smash{\begin{tabular}[t]{c}0.25\end{tabular}}}}%
    \put(0.75066088,0.5707826){\color[rgb]{0.30196078,0.30196078,0.30196078}\makebox(0,0)[t]{\lineheight{1.25}\smash{\begin{tabular}[t]{c}0.50\end{tabular}}}}%
    \put(0.85966956,0.5707826){\color[rgb]{0.30196078,0.30196078,0.30196078}\makebox(0,0)[t]{\lineheight{1.25}\smash{\begin{tabular}[t]{c}0.75\end{tabular}}}}%
    \put(0.96866083,0.5707826){\color[rgb]{0.30196078,0.30196078,0.30196078}\makebox(0,0)[t]{\lineheight{1.25}\smash{\begin{tabular}[t]{c}1.00\end{tabular}}}}%
    \put(0.98093909,0.30384349){\makebox(0,0)[rt]{\lineheight{1.25}\smash{\begin{tabular}[t]{r}$\Delta = 0.2$\end{tabular}}}}%
    \put(0,0){\includegraphics[width=\unitlength,page=4]{plot_VAR_switched_svg-tex.pdf}}%
    \put(0.07243478,0.04855655){\color[rgb]{0.30196078,0.30196078,0.30196078}\makebox(0,0)[rt]{\lineheight{1.25}\smash{\begin{tabular}[t]{r}0.00\end{tabular}}}}%
    \put(0.07243478,0.1026261){\color[rgb]{0.30196078,0.30196078,0.30196078}\makebox(0,0)[rt]{\lineheight{1.25}\smash{\begin{tabular}[t]{r}0.25\end{tabular}}}}%
    \put(0.07243478,0.15669566){\color[rgb]{0.30196078,0.30196078,0.30196078}\makebox(0,0)[rt]{\lineheight{1.25}\smash{\begin{tabular}[t]{r}0.50\end{tabular}}}}%
    \put(0.07243478,0.21076522){\color[rgb]{0.30196078,0.30196078,0.30196078}\makebox(0,0)[rt]{\lineheight{1.25}\smash{\begin{tabular}[t]{r}0.75\end{tabular}}}}%
    \put(0.07243478,0.26483478){\color[rgb]{0.30196078,0.30196078,0.30196078}\makebox(0,0)[rt]{\lineheight{1.25}\smash{\begin{tabular}[t]{r}1.00\end{tabular}}}}%
    \put(0,0){\includegraphics[width=\unitlength,page=5]{plot_VAR_switched_svg-tex.pdf}}%
    \put(0.52144346,0.0227478){\makebox(0,0)[rt]{\lineheight{1.25}\smash{\begin{tabular}[t]{r}$\Delta = 0.3$\end{tabular}}}}%
    \put(0,0){\includegraphics[width=\unitlength,page=6]{plot_VAR_switched_svg-tex.pdf}}%
    \put(0.53097391,0.59017391){\makebox(0,0)[t]{\lineheight{1.25}\smash{\begin{tabular}[t]{c}Distance Correlation\end{tabular}}}}%
    \put(0.03271304,0.30271304){\rotatebox{90}{\makebox(0,0)[t]{\lineheight{1.25}\smash{\begin{tabular}[t]{c}Rejection Rate\end{tabular}}}}}%
    \put(0.98093909,0.0227478){\makebox(0,0)[rt]{\lineheight{1.25}\smash{\begin{tabular}[t]{r}$\Delta = 0.4$\end{tabular}}}}%
    \put(0,0){\includegraphics[width=\unitlength,page=7]{plot_VAR_switched_svg-tex.pdf}}%
    \put(0.33445217,0.6452){\makebox(0,0)[lt]{\lineheight{1.25}\smash{\begin{tabular}[t]{l}$n$\end{tabular}}}}%
    \put(0,0){\includegraphics[width=\unitlength,page=8]{plot_VAR_switched_svg-tex.pdf}}%
    \put(0.41726955,0.64657391){\makebox(0,0)[lt]{\lineheight{1.25}\smash{\begin{tabular}[t]{l}25\end{tabular}}}}%
    \put(0.48387828,0.64657391){\makebox(0,0)[lt]{\lineheight{1.25}\smash{\begin{tabular}[t]{l}50\end{tabular}}}}%
    \put(0.55050436,0.64657391){\makebox(0,0)[lt]{\lineheight{1.25}\smash{\begin{tabular}[t]{l}100\end{tabular}}}}%
    \put(0.62587827,0.64657391){\makebox(0,0)[lt]{\lineheight{1.25}\smash{\begin{tabular}[t]{l}200\end{tabular}}}}%
    \put(0.70123477,0.64657391){\makebox(0,0)[lt]{\lineheight{1.25}\smash{\begin{tabular}[t]{l}400\end{tabular}}}}%
  \end{picture}%
  \endgroup%

  \caption{\it Empirical rejection rates of the test \eqref{det33} for the hypotheses \eqref{det3} for different thresholds $\Delta$ and sample sizes $n$. The dotted lines represent $\Delta$ (vertical) and the nominal level $\alpha = 5\%$ (horizontal). The data are simulated according to the VAR model \eqref{eq:var_model}.}
  \label{fig:var_model_alternative}
\end{figure}

\begin{figure}[t]
  \fontsize{8}{10}\selectfont

  \begingroup%
  \makeatletter%
  \providecommand\color[2][]{%
    \errmessage{(Inkscape) Color is used for the text in Inkscape, but the package 'color.sty' is not loaded}%
    \renewcommand\color[2][]{}%
  }%
  \providecommand\transparent[1]{%
    \errmessage{(Inkscape) Transparency is used (non-zero) for the text in Inkscape, but the package 'transparent.sty' is not loaded}%
    \renewcommand\transparent[1]{}%
  }%
  \providecommand\rotatebox[2]{#2}%
  \newcommand*\fsize{\dimexpr\f@size pt\relax}%
  \newcommand*\lineheight[1]{\fontsize{\fsize}{#1\fsize}\selectfont}%
  \ifx\svgwidth\undefined%
    \setlength{\unitlength}{\textwidth}%
    \ifx\svgscale\undefined%
      \relax%
    \else%
      \setlength{\unitlength}{\unitlength * \real{\svgscale}}%
    \fi%
  \else%
    \setlength{\unitlength}{\svgwidth}%
  \fi%
  \global\let\svgwidth\undefined%
  \global\let\svgscale\undefined%
  \makeatother%
  \begin{picture}(1,0.69565217)%
    \lineheight{1}%
    \setlength\tabcolsep{0pt}%
    \put(0,0){\includegraphics[width=\unitlength,page=1]{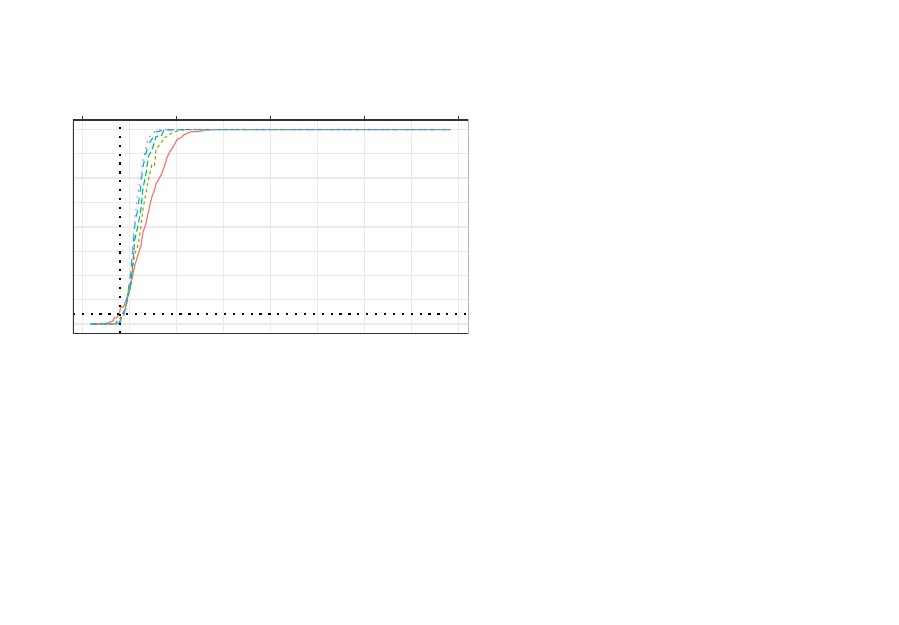}}%
    \put(0.09184348,0.5707826){\color[rgb]{0.30196078,0.30196078,0.30196078}\makebox(0,0)[t]{\lineheight{1.25}\smash{\begin{tabular}[t]{c}0.00\end{tabular}}}}%
    \put(0.19645217,0.5707826){\color[rgb]{0.30196078,0.30196078,0.30196078}\makebox(0,0)[t]{\lineheight{1.25}\smash{\begin{tabular}[t]{c}0.25\end{tabular}}}}%
    \put(0.30106087,0.5707826){\color[rgb]{0.30196078,0.30196078,0.30196078}\makebox(0,0)[t]{\lineheight{1.25}\smash{\begin{tabular}[t]{c}0.50\end{tabular}}}}%
    \put(0.40566955,0.5707826){\color[rgb]{0.30196078,0.30196078,0.30196078}\makebox(0,0)[t]{\lineheight{1.25}\smash{\begin{tabular}[t]{c}0.75\end{tabular}}}}%
    \put(0.51027827,0.5707826){\color[rgb]{0.30196078,0.30196078,0.30196078}\makebox(0,0)[t]{\lineheight{1.25}\smash{\begin{tabular}[t]{c}1.00\end{tabular}}}}%
    \put(0.07243478,0.32965218){\color[rgb]{0.30196078,0.30196078,0.30196078}\makebox(0,0)[rt]{\lineheight{1.25}\smash{\begin{tabular}[t]{r}0.00\end{tabular}}}}%
    \put(0.07243478,0.38372174){\color[rgb]{0.30196078,0.30196078,0.30196078}\makebox(0,0)[rt]{\lineheight{1.25}\smash{\begin{tabular}[t]{r}0.25\end{tabular}}}}%
    \put(0.07243478,0.4377913){\color[rgb]{0.30196078,0.30196078,0.30196078}\makebox(0,0)[rt]{\lineheight{1.25}\smash{\begin{tabular}[t]{r}0.50\end{tabular}}}}%
    \put(0.07243478,0.49186087){\color[rgb]{0.30196078,0.30196078,0.30196078}\makebox(0,0)[rt]{\lineheight{1.25}\smash{\begin{tabular}[t]{r}0.75\end{tabular}}}}%
    \put(0.07243478,0.54593044){\color[rgb]{0.30196078,0.30196078,0.30196078}\makebox(0,0)[rt]{\lineheight{1.25}\smash{\begin{tabular}[t]{r}1.00\end{tabular}}}}%
    \put(0,0){\includegraphics[width=\unitlength,page=2]{plot_fourier_sparse_largesamples_svg-tex.pdf}}%
    \put(0.52144346,0.30384349){\makebox(0,0)[rt]{\lineheight{1.25}\smash{\begin{tabular}[t]{r}$\Delta = 0.1$\end{tabular}}}}%
    \put(0,0){\includegraphics[width=\unitlength,page=3]{plot_fourier_sparse_largesamples_svg-tex.pdf}}%
    \put(0.55133911,0.5707826){\color[rgb]{0.30196078,0.30196078,0.30196078}\makebox(0,0)[t]{\lineheight{1.25}\smash{\begin{tabular}[t]{c}0.00\end{tabular}}}}%
    \put(0.65594785,0.5707826){\color[rgb]{0.30196078,0.30196078,0.30196078}\makebox(0,0)[t]{\lineheight{1.25}\smash{\begin{tabular}[t]{c}0.25\end{tabular}}}}%
    \put(0.76055653,0.5707826){\color[rgb]{0.30196078,0.30196078,0.30196078}\makebox(0,0)[t]{\lineheight{1.25}\smash{\begin{tabular}[t]{c}0.50\end{tabular}}}}%
    \put(0.86516522,0.5707826){\color[rgb]{0.30196078,0.30196078,0.30196078}\makebox(0,0)[t]{\lineheight{1.25}\smash{\begin{tabular}[t]{c}0.75\end{tabular}}}}%
    \put(0.9697739,0.5707826){\color[rgb]{0.30196078,0.30196078,0.30196078}\makebox(0,0)[t]{\lineheight{1.25}\smash{\begin{tabular}[t]{c}1.00\end{tabular}}}}%
    \put(0.98093909,0.30384349){\makebox(0,0)[rt]{\lineheight{1.25}\smash{\begin{tabular}[t]{r}$\Delta = 0.2$\end{tabular}}}}%
    \put(0,0){\includegraphics[width=\unitlength,page=4]{plot_fourier_sparse_largesamples_svg-tex.pdf}}%
    \put(0.07243478,0.04855655){\color[rgb]{0.30196078,0.30196078,0.30196078}\makebox(0,0)[rt]{\lineheight{1.25}\smash{\begin{tabular}[t]{r}0.00\end{tabular}}}}%
    \put(0.07243478,0.1026261){\color[rgb]{0.30196078,0.30196078,0.30196078}\makebox(0,0)[rt]{\lineheight{1.25}\smash{\begin{tabular}[t]{r}0.25\end{tabular}}}}%
    \put(0.07243478,0.15669566){\color[rgb]{0.30196078,0.30196078,0.30196078}\makebox(0,0)[rt]{\lineheight{1.25}\smash{\begin{tabular}[t]{r}0.50\end{tabular}}}}%
    \put(0.07243478,0.21076522){\color[rgb]{0.30196078,0.30196078,0.30196078}\makebox(0,0)[rt]{\lineheight{1.25}\smash{\begin{tabular}[t]{r}0.75\end{tabular}}}}%
    \put(0.07243478,0.26483478){\color[rgb]{0.30196078,0.30196078,0.30196078}\makebox(0,0)[rt]{\lineheight{1.25}\smash{\begin{tabular}[t]{r}1.00\end{tabular}}}}%
    \put(0,0){\includegraphics[width=\unitlength,page=5]{plot_fourier_sparse_largesamples_svg-tex.pdf}}%
    \put(0.52144346,0.0227478){\makebox(0,0)[rt]{\lineheight{1.25}\smash{\begin{tabular}[t]{r}$\Delta = 0.3$\end{tabular}}}}%
    \put(0,0){\includegraphics[width=\unitlength,page=6]{plot_fourier_sparse_largesamples_svg-tex.pdf}}%
    \put(0.53097391,0.59017391){\makebox(0,0)[t]{\lineheight{1.25}\smash{\begin{tabular}[t]{c}Distance Correlation\end{tabular}}}}%
    \put(0.03271304,0.30271304){\rotatebox{90}{\makebox(0,0)[t]{\lineheight{1.25}\smash{\begin{tabular}[t]{c}Rejection Rate\end{tabular}}}}}%
    \put(0.98093909,0.0227478){\makebox(0,0)[rt]{\lineheight{1.25}\smash{\begin{tabular}[t]{r}$\Delta = 0.4$\end{tabular}}}}%
    \put(0,0){\includegraphics[width=\unitlength,page=7]{plot_fourier_sparse_largesamples_svg-tex.pdf}}%
    \put(0.32132173,0.6452){\makebox(0,0)[lt]{\lineheight{1.25}\smash{\begin{tabular}[t]{l}$n$\end{tabular}}}}%
    \put(0,0){\includegraphics[width=\unitlength,page=8]{plot_fourier_sparse_largesamples_svg-tex.pdf}}%
    \put(0.40413914,0.64657391){\makebox(0,0)[lt]{\lineheight{1.25}\smash{\begin{tabular}[t]{l}200\end{tabular}}}}%
    \put(0.47951305,0.64657391){\makebox(0,0)[lt]{\lineheight{1.25}\smash{\begin{tabular}[t]{l}400\end{tabular}}}}%
    \put(0.55488695,0.64657391){\makebox(0,0)[lt]{\lineheight{1.25}\smash{\begin{tabular}[t]{l}600\end{tabular}}}}%
    \put(0.6302435,0.64657391){\makebox(0,0)[lt]{\lineheight{1.25}\smash{\begin{tabular}[t]{l}800\end{tabular}}}}%
    \put(0.70561741,0.64657391){\makebox(0,0)[lt]{\lineheight{1.25}\smash{\begin{tabular}[t]{l}1000\end{tabular}}}}%
  \end{picture}%
  \endgroup%

  \caption{\it Simulation results for the model \eqref{eq:fourier_model} with covariance matrix \eqref{eq:sparse} and varying parameter $\rho$. The $x$-axes give the corresponding values for $\mathrm{dcor}(X,Y)$. Empirical rejection rates of the test \eqref{det5} for the hypotheses \eqref{det2} for different thresholds $\Delta$ and sample sizes $n$. The dotted lines represent $\Delta$ (vertical) and the nominal level $\alpha = 5\%$ (horizontal).}
  \label{fig:fourier_modell1}
\end{figure}

\begin{figure}[t]
  \fontsize{8}{10}\selectfont

  \begingroup%
  \makeatletter%
  \providecommand\color[2][]{%
    \errmessage{(Inkscape) Color is used for the text in Inkscape, but the package 'color.sty' is not loaded}%
    \renewcommand\color[2][]{}%
  }%
  \providecommand\transparent[1]{%
    \errmessage{(Inkscape) Transparency is used (non-zero) for the text in Inkscape, but the package 'transparent.sty' is not loaded}%
    \renewcommand\transparent[1]{}%
  }%
  \providecommand\rotatebox[2]{#2}%
  \newcommand*\fsize{\dimexpr\f@size pt\relax}%
  \newcommand*\lineheight[1]{\fontsize{\fsize}{#1\fsize}\selectfont}%
  \ifx\svgwidth\undefined%
    \setlength{\unitlength}{\textwidth}%
    \ifx\svgscale\undefined%
      \relax%
    \else%
      \setlength{\unitlength}{\unitlength * \real{\svgscale}}%
    \fi%
  \else%
    \setlength{\unitlength}{\svgwidth}%
  \fi%
  \global\let\svgwidth\undefined%
  \global\let\svgscale\undefined%
  \makeatother%
  \begin{picture}(1,0.69565217)%
    \lineheight{1}%
    \setlength\tabcolsep{0pt}%
    \put(0,0){\includegraphics[width=\unitlength,page=1]{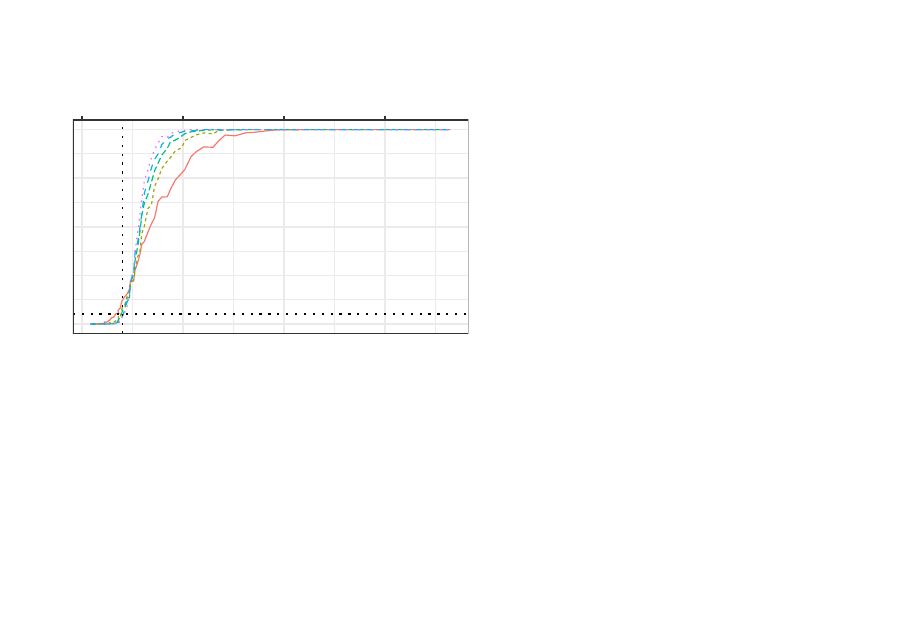}}%
    \put(0.09118261,0.5707826){\color[rgb]{0.30196078,0.30196078,0.30196078}\makebox(0,0)[t]{\lineheight{1.25}\smash{\begin{tabular}[t]{c}0.00\end{tabular}}}}%
    \put(0.20368696,0.5707826){\color[rgb]{0.30196078,0.30196078,0.30196078}\makebox(0,0)[t]{\lineheight{1.25}\smash{\begin{tabular}[t]{c}0.25\end{tabular}}}}%
    \put(0.3161913,0.5707826){\color[rgb]{0.30196078,0.30196078,0.30196078}\makebox(0,0)[t]{\lineheight{1.25}\smash{\begin{tabular}[t]{c}0.50\end{tabular}}}}%
    \put(0.42871303,0.5707826){\color[rgb]{0.30196078,0.30196078,0.30196078}\makebox(0,0)[t]{\lineheight{1.25}\smash{\begin{tabular}[t]{c}0.75\end{tabular}}}}%
    \put(0.07243478,0.32965218){\color[rgb]{0.30196078,0.30196078,0.30196078}\makebox(0,0)[rt]{\lineheight{1.25}\smash{\begin{tabular}[t]{r}0.00\end{tabular}}}}%
    \put(0.07243478,0.38372174){\color[rgb]{0.30196078,0.30196078,0.30196078}\makebox(0,0)[rt]{\lineheight{1.25}\smash{\begin{tabular}[t]{r}0.25\end{tabular}}}}%
    \put(0.07243478,0.4377913){\color[rgb]{0.30196078,0.30196078,0.30196078}\makebox(0,0)[rt]{\lineheight{1.25}\smash{\begin{tabular}[t]{r}0.50\end{tabular}}}}%
    \put(0.07243478,0.49186087){\color[rgb]{0.30196078,0.30196078,0.30196078}\makebox(0,0)[rt]{\lineheight{1.25}\smash{\begin{tabular}[t]{r}0.75\end{tabular}}}}%
    \put(0.07243478,0.54593044){\color[rgb]{0.30196078,0.30196078,0.30196078}\makebox(0,0)[rt]{\lineheight{1.25}\smash{\begin{tabular}[t]{r}1.00\end{tabular}}}}%
    \put(0,0){\includegraphics[width=\unitlength,page=2]{plot_fourier_full_largesamples_svg-tex.pdf}}%
    \put(0.52144346,0.30384349){\makebox(0,0)[rt]{\lineheight{1.25}\smash{\begin{tabular}[t]{r}$\Delta = 0.1$\end{tabular}}}}%
    \put(0,0){\includegraphics[width=\unitlength,page=3]{plot_fourier_full_largesamples_svg-tex.pdf}}%
    \put(0.55067828,0.5707826){\color[rgb]{0.30196078,0.30196078,0.30196078}\makebox(0,0)[t]{\lineheight{1.25}\smash{\begin{tabular}[t]{c}0.00\end{tabular}}}}%
    \put(0.66318258,0.5707826){\color[rgb]{0.30196078,0.30196078,0.30196078}\makebox(0,0)[t]{\lineheight{1.25}\smash{\begin{tabular}[t]{c}0.25\end{tabular}}}}%
    \put(0.77568694,0.5707826){\color[rgb]{0.30196078,0.30196078,0.30196078}\makebox(0,0)[t]{\lineheight{1.25}\smash{\begin{tabular}[t]{c}0.50\end{tabular}}}}%
    \put(0.8882087,0.5707826){\color[rgb]{0.30196078,0.30196078,0.30196078}\makebox(0,0)[t]{\lineheight{1.25}\smash{\begin{tabular}[t]{c}0.75\end{tabular}}}}%
    \put(0.98093909,0.30384349){\makebox(0,0)[rt]{\lineheight{1.25}\smash{\begin{tabular}[t]{r}$\Delta = 0.2$\end{tabular}}}}%
    \put(0,0){\includegraphics[width=\unitlength,page=4]{plot_fourier_full_largesamples_svg-tex.pdf}}%
    \put(0.07243478,0.04855655){\color[rgb]{0.30196078,0.30196078,0.30196078}\makebox(0,0)[rt]{\lineheight{1.25}\smash{\begin{tabular}[t]{r}0.00\end{tabular}}}}%
    \put(0.07243478,0.1026261){\color[rgb]{0.30196078,0.30196078,0.30196078}\makebox(0,0)[rt]{\lineheight{1.25}\smash{\begin{tabular}[t]{r}0.25\end{tabular}}}}%
    \put(0.07243478,0.15669566){\color[rgb]{0.30196078,0.30196078,0.30196078}\makebox(0,0)[rt]{\lineheight{1.25}\smash{\begin{tabular}[t]{r}0.50\end{tabular}}}}%
    \put(0.07243478,0.21076522){\color[rgb]{0.30196078,0.30196078,0.30196078}\makebox(0,0)[rt]{\lineheight{1.25}\smash{\begin{tabular}[t]{r}0.75\end{tabular}}}}%
    \put(0.07243478,0.26483478){\color[rgb]{0.30196078,0.30196078,0.30196078}\makebox(0,0)[rt]{\lineheight{1.25}\smash{\begin{tabular}[t]{r}1.00\end{tabular}}}}%
    \put(0,0){\includegraphics[width=\unitlength,page=5]{plot_fourier_full_largesamples_svg-tex.pdf}}%
    \put(0.52144346,0.0227478){\makebox(0,0)[rt]{\lineheight{1.25}\smash{\begin{tabular}[t]{r}$\Delta = 0.3$\end{tabular}}}}%
    \put(0,0){\includegraphics[width=\unitlength,page=6]{plot_fourier_full_largesamples_svg-tex.pdf}}%
    \put(0.53097391,0.59017391){\makebox(0,0)[t]{\lineheight{1.25}\smash{\begin{tabular}[t]{c}Distance Correlation\end{tabular}}}}%
    \put(0.03271304,0.30271304){\rotatebox{90}{\makebox(0,0)[t]{\lineheight{1.25}\smash{\begin{tabular}[t]{c}Rejection Rate\end{tabular}}}}}%
    \put(0.98093909,0.0227478){\makebox(0,0)[rt]{\lineheight{1.25}\smash{\begin{tabular}[t]{r}$\Delta = 0.4$\end{tabular}}}}%
    \put(0,0){\includegraphics[width=\unitlength,page=7]{plot_fourier_full_largesamples_svg-tex.pdf}}%
    \put(0.32132173,0.6452){\makebox(0,0)[lt]{\lineheight{1.25}\smash{\begin{tabular}[t]{l}$n$\end{tabular}}}}%
    \put(0,0){\includegraphics[width=\unitlength,page=8]{plot_fourier_full_largesamples_svg-tex.pdf}}%
    \put(0.40413914,0.64657391){\makebox(0,0)[lt]{\lineheight{1.25}\smash{\begin{tabular}[t]{l}200\end{tabular}}}}%
    \put(0.47951305,0.64657391){\makebox(0,0)[lt]{\lineheight{1.25}\smash{\begin{tabular}[t]{l}400\end{tabular}}}}%
    \put(0.55488695,0.64657391){\makebox(0,0)[lt]{\lineheight{1.25}\smash{\begin{tabular}[t]{l}600\end{tabular}}}}%
    \put(0.6302435,0.64657391){\makebox(0,0)[lt]{\lineheight{1.25}\smash{\begin{tabular}[t]{l}800\end{tabular}}}}%
    \put(0.70561741,0.64657391){\makebox(0,0)[lt]{\lineheight{1.25}\smash{\begin{tabular}[t]{l}1000\end{tabular}}}}%
  \end{picture}%
  \endgroup%

  \caption{\it Simulation results for the model \eqref{eq:fourier_model} with covariance matrix \eqref{eq:full} and varying parameter $\rho$. The $x$-axes give the corresponding values for $\mathrm{dcor}(X,Y)$. Empirical rejection rates of the test \eqref{det5} for the hypotheses \eqref{det2} for different thresholds $\Delta$ and sample sizes $n$. The dotted lines represent $\Delta$ (vertical) and the nominal level $\alpha = 5\%$ (horizontal).}
  \label{fig:fourier_modell_voll2}
\end{figure}

\begin{figure}[t]
  \fontsize{8}{10}\selectfont

  \begingroup%
  \makeatletter%
  \providecommand\color[2][]{%
    \errmessage{(Inkscape) Color is used for the text in Inkscape, but the package 'color.sty' is not loaded}%
    \renewcommand\color[2][]{}%
  }%
  \providecommand\transparent[1]{%
    \errmessage{(Inkscape) Transparency is used (non-zero) for the text in Inkscape, but the package 'transparent.sty' is not loaded}%
    \renewcommand\transparent[1]{}%
  }%
  \providecommand\rotatebox[2]{#2}%
  \newcommand*\fsize{\dimexpr\f@size pt\relax}%
  \newcommand*\lineheight[1]{\fontsize{\fsize}{#1\fsize}\selectfont}%
  \ifx\svgwidth\undefined%
    \setlength{\unitlength}{\textwidth}
    \ifx\svgscale\undefined%
      \relax%
    \else%
      \setlength{\unitlength}{\unitlength * \real{\svgscale}}%
    \fi%
  \else%
    \setlength{\unitlength}{\svgwidth}%
  \fi%
  \global\let\svgwidth\undefined%
  \global\let\svgscale\undefined%
  \makeatother%
  \begin{picture}(1,0.69565217)%
    \lineheight{1}%
    \setlength\tabcolsep{0pt}%
    \put(0,0){\includegraphics[width=\unitlength,page=1]{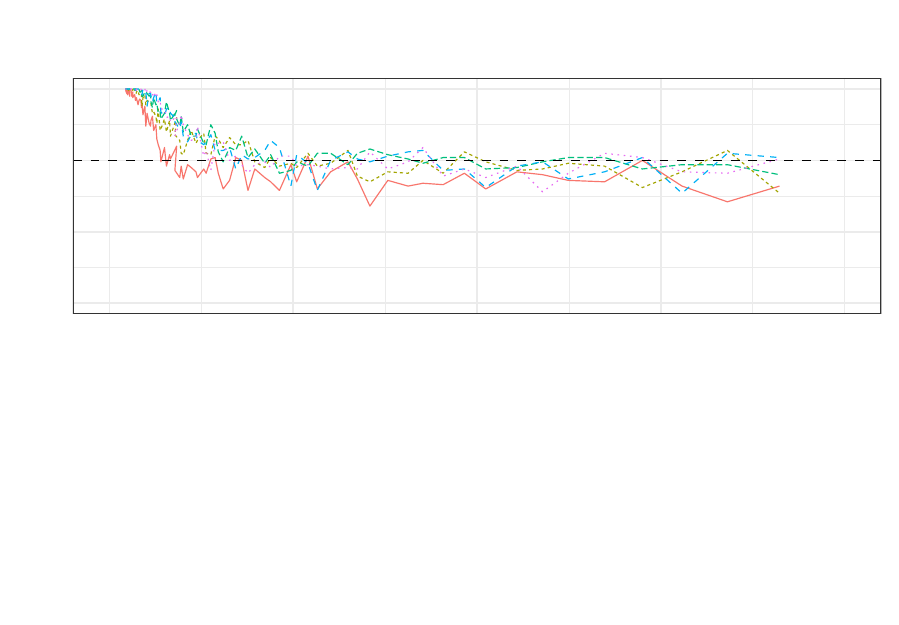}}%
    \put(0.07243478,0.35281738){\color[rgb]{0.30196078,0.30196078,0.30196078}\makebox(0,0)[rt]{\lineheight{1.25}\smash{\begin{tabular}[t]{r}0.85\end{tabular}}}}%
    \put(0.07243478,0.43222608){\color[rgb]{0.30196078,0.30196078,0.30196078}\makebox(0,0)[rt]{\lineheight{1.25}\smash{\begin{tabular}[t]{r}0.90\end{tabular}}}}%
    \put(0.07243478,0.51163479){\color[rgb]{0.30196078,0.30196078,0.30196078}\makebox(0,0)[rt]{\lineheight{1.25}\smash{\begin{tabular}[t]{r}0.95\end{tabular}}}}%
    \put(0.07243478,0.59104348){\color[rgb]{0.30196078,0.30196078,0.30196078}\makebox(0,0)[rt]{\lineheight{1.25}\smash{\begin{tabular}[t]{r}1.00\end{tabular}}}}%
    \put(0,0){\includegraphics[width=\unitlength,page=2]{plot_ci_fourier_svg-tex.pdf}}%
    \put(0.07243478,0.07172172){\color[rgb]{0.30196078,0.30196078,0.30196078}\makebox(0,0)[rt]{\lineheight{1.25}\smash{\begin{tabular}[t]{r}0.85\end{tabular}}}}%
    \put(0.07243478,0.15113042){\color[rgb]{0.30196078,0.30196078,0.30196078}\makebox(0,0)[rt]{\lineheight{1.25}\smash{\begin{tabular}[t]{r}0.90\end{tabular}}}}%
    \put(0.07243478,0.23053913){\color[rgb]{0.30196078,0.30196078,0.30196078}\makebox(0,0)[rt]{\lineheight{1.25}\smash{\begin{tabular}[t]{r}0.95\end{tabular}}}}%
    \put(0.07243478,0.30994783){\color[rgb]{0.30196078,0.30196078,0.30196078}\makebox(0,0)[rt]{\lineheight{1.25}\smash{\begin{tabular}[t]{r}1.00\end{tabular}}}}%
    \put(0,0){\includegraphics[width=\unitlength,page=3]{plot_ci_fourier_svg-tex.pdf}}%
    \put(0.12191304,0.04575652){\color[rgb]{0.30196078,0.30196078,0.30196078}\makebox(0,0)[t]{\lineheight{1.25}\smash{\begin{tabular}[t]{c}0.00\end{tabular}}}}%
    \put(0.32645218,0.04575652){\color[rgb]{0.30196078,0.30196078,0.30196078}\makebox(0,0)[t]{\lineheight{1.25}\smash{\begin{tabular}[t]{c}0.25\end{tabular}}}}%
    \put(0.53097391,0.04575652){\color[rgb]{0.30196078,0.30196078,0.30196078}\makebox(0,0)[t]{\lineheight{1.25}\smash{\begin{tabular}[t]{c}0.50\end{tabular}}}}%
    \put(0.73551307,0.04575652){\color[rgb]{0.30196078,0.30196078,0.30196078}\makebox(0,0)[t]{\lineheight{1.25}\smash{\begin{tabular}[t]{c}0.75\end{tabular}}}}%
    \put(0.94003481,0.04575652){\color[rgb]{0.30196078,0.30196078,0.30196078}\makebox(0,0)[t]{\lineheight{1.25}\smash{\begin{tabular}[t]{c}1.00\end{tabular}}}}%
    \put(0.53097391,0.02365218){\makebox(0,0)[t]{\lineheight{1.25}\smash{\begin{tabular}[t]{c}Distance Correlation\end{tabular}}}}%
    \put(0.03271304,0.33685217){\rotatebox{90}{\makebox(0,0)[t]{\lineheight{1.25}\smash{\begin{tabular}[t]{c}Covering Rate\end{tabular}}}}}%
    \put(0,0){\includegraphics[width=\unitlength,page=4]{plot_ci_fourier_svg-tex.pdf}}%
    \put(0.32132173,0.6452){\makebox(0,0)[lt]{\lineheight{1.25}\smash{\begin{tabular}[t]{l}$n$\end{tabular}}}}%
    \put(0,0){\includegraphics[width=\unitlength,page=5]{plot_ci_fourier_svg-tex.pdf}}%
    \put(0.40413914,0.64657391){\makebox(0,0)[lt]{\lineheight{1.25}\smash{\begin{tabular}[t]{l}200\end{tabular}}}}%
    \put(0.47951305,0.64657391){\makebox(0,0)[lt]{\lineheight{1.25}\smash{\begin{tabular}[t]{l}400\end{tabular}}}}%
    \put(0.55488695,0.64657391){\makebox(0,0)[lt]{\lineheight{1.25}\smash{\begin{tabular}[t]{l}600\end{tabular}}}}%
    \put(0.6302435,0.64657391){\makebox(0,0)[lt]{\lineheight{1.25}\smash{\begin{tabular}[t]{l}800\end{tabular}}}}%
    \put(0.70561741,0.64657391){\makebox(0,0)[lt]{\lineheight{1.25}\smash{\begin{tabular}[t]{l}1000\end{tabular}}}}%
  \end{picture}%
  \endgroup%

  \caption{\it Empirical covering rates of the confidence interval \eqref{eq:confidence_interval} with $\alpha = 0.95$ and data generated according to Eq.\@ \eqref{eq:fourier_model}. The dashed line represents the nominal level of $95 \%$. Top panel: Full covariance matrix \eqref{eq:full}. Bottom panel: Sparse covariance matrix \eqref{eq:sparse}.}
  \label{fig:ci_fourier}
\end{figure}

The first model under consideration is a $2$-dimensional vector autoregressive (VAR) process, i.e. $(X_k, Y_k)$, $k = 1, \ldots, n$, satisfy the equation
\begin{equation}
  \label{eq:var_model}
  \begin{pmatrix}
    X_n \\ Y_n
  \end{pmatrix} = A \, \begin{pmatrix}
    X_{n-1} \\ Y_{n-1}
  \end{pmatrix} + \varepsilon_n,
\end{equation}
where the $\varepsilon_n$, $n \in \mathbb{N}$, are i.i.d.\@ centred Gaussian random vectors with covariance matrix $\Sigma$.  The matrix $A$ is chosen as
$$
  A = \begin{pmatrix}
    1/2 & 1/5 \\ 1/5 & 1/2
  \end{pmatrix}
$$
and $\Sigma$ is given by
$$
  \Sigma = \begin{pmatrix}
    1 & \rho \\ \rho & 1
  \end{pmatrix}.
$$

In Figure \ref{fig:var_model}  we display the empirical rejection rate of the test \eqref{det5} for the hypotheses \eqref{det2} in the VAR model \eqref{eq:var_model}. The results reflect the qualitative behavior  of the test \eqref{det5} described in Proposition \ref{cor:test_dcov}. In the ``interior'' of the null hypothesis ($\mathrm{dcor}_n(X,Y) < \Delta$) the empirical rejection probabilities are close $0$ and they quickly increase under the alternative ($\mathrm{dcor}_n(X,Y) > \Delta$). At the ``boundary'' of the hypotheses ($\mathrm{dcor}_n(X,Y) = \Delta$) the empirical rejection probabilities are very close to the nominal level $\alpha$ (which is predicted by Proposition \ref{cor:test_dcov} as the limit in this case). The empirical covering rates of the confidence intervals \eqref{eq:confidence_interval} for this model are given in Figure \ref{fig:ci_VAR}. They are close to the nominal covering rate of $95\%$, though some deviation occurs at very small sample sizes ($n = 25, 50$). However, it should be pointed out that even at these small sample sizes, the empirical covering rate almost never drops below $90\%$. All in all, the finite sample performance seems to be satisfactory even for small sample sizes, and very good for medium and large sample sizes.

A  key advantages of testing relevant hypotheses is that one may switch the roles of the null hypothesis and alternative.  More precisely, if one wants to control the probability of an error for deciding in favor of $ \mathrm{dcor}(X,Y)  < \Delta $ one can consider the testing problem  \eqref{det3} and the corresponding test \eqref{det33}.
To illustrate this fact we display in Figure \ref{fig:var_model_alternative}  the rejection rates of the test \eqref{det33}  in the
VAR model \eqref{eq:var_model}. These results correspond to the theoretical findings in Remark \ref{rem2}.  We observe that the type I error is very small under the null hypotheses and that larger sample sizes are required to reject the null hypothesis in  \eqref{det2} with reasonable power.

To investigate the performance of our test in a non-Euclidean setting, we simulate a ten-dimensional VAR process $(C_k)_{k \in \mathbb{N}}$ satisfying
$$
  C_k = C_{k-1}/2 + \eta_k,
$$
where the $\eta_k$ are i.i.d.\@ centred Gaussian random vectors with their covariance matrices $(\sigma_{ij})_{i,j = 1, \ldots, 10}$ given by either
\begin{equation}
  \label{eq:sparse}
  \sigma_{ij} = \begin{cases}
    1 & \quad \textrm{if } i = j, \\ \rho &\quad \textrm{if } |i-j| = 5, \\ 0 &\quad \textrm{otherwise,}
  \end{cases}
\end{equation}
or
\begin{equation}
  \label{eq:full}
  \sigma_{ij} = \rho^{|i-j|}
\end{equation}
for appropriate choices of $\rho$. Writing $C_k^{(j)}$ for the $j$-th coordinate of $C_k$, we then define random functions $X_k$ and $Y_k$ by
\begin{equation}
  \label{eq:fourier_model}
  (X_k, Y_k) : [0,1] \to \mathbb{R}^2, \quad s \mapsto \left(\sum_{j=0}^4 C_k^{(j+1)} \varphi_{j}(s), \sum_{j=0}^4 C_k^{(j+6)} \varphi_{j}(s)\right),
\end{equation}
where $\varphi_0 = 1$, $\varphi_{2j}(s) = \sqrt{2} \sin(2\pi js)$ and $\varphi_{2j-1}(s) = \sqrt{2} \cos(2\pi js)$. The simulation results for our functional data are given in Figures \ref{fig:fourier_modell1} and \ref{fig:fourier_modell_voll2}. In contrast to the two-dimensional VAR-model \eqref{eq:var_model}, we only give the simulation results for $n \geq 200$, since we observed a failure of the proposed method for sample sizes much smaller than this. We suspect that this is simply due to the more complex structure of the functional data, and based on this effect we recommend against employing the self-normalizing procedure for functional data for samples smaller than $n = 200$ (which in any case does not seem to be an unreasonably large sample size). The simulation results confirm our theoretical findings in the case of functional data. In most scenarios the nominal level $\alpha=5\%$ is well approximated at the boundary of the hypotheses ($\mathrm{dcor}_n(X,Y) = \Delta$). In the interior of the null hypothesis ($\mathrm{dcor}_n(X,Y) <  \Delta$) the type I error is smaller than the nominal level, while the rejection probabilities quickly increase  with increasing dependence measured by $\mathrm{dcor}_n(X,Y) > \Delta $. The covering rates of the confidence intervals in this model are given in Figure \ref{fig:ci_fourier}. They are overall still satisfactory, although the case $n = 200$ in the sparse covariance matrix variant drops below the nominal covering rate of $95\%$ in many instances. Furthermore, we can observe an interesting phenomenon; namely, that the covering rate approaches $100\%$ as the value of the distance correlation approaches $0$, indicating a failure of the underlying asymptotic theory in the edge case $\mathrm{dcor}(X,Y) = 0$. This is consistent with the different limiting behaviour of the distance covariance under the hypothesis of perfect independence and in particular the different rate of convergence; cf.\@ Theorem \ref{thm:independence_weak_convergence}.

\section{Conclusions}
Distance correlation has long been accepted as a powerful tool to measure dependence between random objects $X$ and $Y$. Two of its most appealing properties are its generality -- it can be used  data in separable metric spaces -- and the fact that it perfectly characterizes independence. However, in classical hypothesis testing, one does not care about a small amount of dependence and is looking to detect only what we call `practically significant' dependence. It can therefore be desirable to test the relevant hypothesis $\mathrm{dcov}(X,Y) \leq \Delta$ instead of the classical hypothesis $\mathrm{dcor}(X,Y) = 0$. Alternatively, one can avoid  hypothesis testing  and construct confidence intervals for $\mathrm{dcor}(X,Y)$. In this article, we have presented a self-normalization approach which can be used for both of these objectives. Our methods work not only for i.i.d.\@ data but also for time series under the well-established $\beta$-mixing condition, and their  finite sample performance is convincing. For finite-dimensional data, it even worked reliably for extremely small sample sizes ($n \geq 25$); for infinite-dimensional data, we recommend using samples of at least medium size (in our simulations, $n \geq 200$ was sufficient). The procedure does not require any resampling, which makes its computational complexity far superior to, say, bootstrap-based methods, which for data with serial dependence appear to be the only real competitors to our procedure in the context of confidence intervals.

There are several possible avenues for generalization of our results, most of which stem from the computational advantage that the self-normalization brings when compared with resampling methods. For instance, \cite{betkendehlingkroll:2022} and \cite{chu:2023} construct tests for independence of entire time series based on the distance covariance combined with bootstrap methods. Translating this problem into a relevant hypothesis framework and applying the self-normalization strategy considered here could bring huge improvements in terms of computational complexity. Similarly, \cite{davis_et_al:2018} provide a test for serial dependence based on so-called cross-distance covariance function, i.e.\@ an distance covariance based analogue of the usual cross-covariance function \citep[see also][]{zhou:2012}. To determine critical values, they need to resort to computationally expensive resampling methods. Again, the switch to a relevant hypothesis setting in combination with a self-normalization approach could be very useful here. Moreover, it would be desirable to obtain results similar to ours in the case of long range dependent data  \citep[see][]{Pipiras_Taqqu_2017}.

A further interesting and promising direction for future research is the  extension of the developed methodology to decide which of the covariates ${X^{(1)}, \ldots,X^{(d)}}$ has a practically  relevant impact  on the random variable $Y$. Observing the weak convergence in \eqref{det8}, it is easy to see that the $p$-value of the test \eqref{det5} is given by $ p_i(\Delta) := \mathbb{P} \big  ( W > ({ \mathrm{dcor}_n(X^{(i)},Y) - \Delta})/{V_{n,\mathrm{dcor}}} \big )$, $i=1, \ldots , d$, where the random variable $W$ is defined in \eqref{eq:definition_W}. Therefore, common concepts of multiple testing
\citep[see, for example][]{Benjamini1995,BenjaminiYekutieli2001}
can easily be adapted if a  threshold $\Delta$ for the multiple hypotheses can be specified (in fact it possible to use different thresholds for in the different hypotheses). However, although the $p$-values are increasing functions of the threshold,  extending the discussion in Remark \ref{rem:visual_inspection}(a) is a more delicate problem and depends intrinsically on the multiple testing procedure under consideration. For example, the Benjamini–Yekutieli procedure considers the ordered $p$-values $p_{(1)}(\Delta) \leq \ldots \leq p_{(d)}(\Delta)  $ and declares the  $p$-values $p_{(1)}(\Delta) \leq \ldots \leq p_{(i^*)}(\Delta)  $ as significant, where $i^* \in \{1, \ldots , d\} $ is the largest integer  satisfying
\begin{align}
  \label{conclusion}
  p_{(i^*)} (\Delta)  \leq \tfrac{i^*}{d} \Big (\sum_{j=1}^d\tfrac{1}{j} \Big )^{-1} \alpha.
\end{align}
The parameter $\alpha$ is called the false discovery rate and is fixed in advance. In this situation, there exists a trade-off between $\Delta$ and $i^*$.  For example, decreasing $\Delta$ will increase $i^*$. On the other hand, if one fixes $i^*$, we can find a minimum $\hat \Delta_{\alpha,i^*} $ such that \eqref{conclusion} holds. The interpretation of such a data adaptive threshold is not easy. A discussion of the relations between $\alpha$, $i^*$ and $\Delta$ for different multiple testing problems is deferred to future research.

\section{Proofs}
\label{sec:proofs}

\subsection{Preliminaries and Proofs of Theorem \ref{thm:dcov_prozesskonvergenz} and Corollary \ref{cor:dcor_processkonvergenz}}
Throughout this section, we use the notations $Z := (X,Y)$, $Z_n := (X_n, Y_n)$ and $\hat{\mathbb{P}}_n$, $\hat{\mathbb{P}}_n^X$ and $\hat{\mathbb{P}}_n^Y$ for the empirical versions of the distributions $\mathbb{P}^{X,Y}$,$\mathbb{P}^{Y}$ and  $\mathbb{P}^{Y}$, respectively. For an index set $T$, $\ell^\infty(T)$ denotes the space of bounded functions on $T$ equipped with the supremum norm, which we sometimes denote by $\|\cdot\|_\infty$. The symbol $\mathcal{O}_{\mathbb{P}}$ denotes boundedness in probability, and analogous definitions hold for $\mathcal{O}_{a.s.}$ (almost sure boundedness) and other Landau symbols such as $o_{\mathbb{P}}$ or $o_{a.s.}$. $\mathcal{L}(U)$ denotes the distribution of a random variable $U$. Finally, $\|\cdot\|_2$ is the usual Euclidean norm on $\mathbb{R}^d$, and the symbols $\land$ and $\lor$ may be used to denote the minimum and maximum, respectively, of two numbers.

\begin{lemma}
  \label{lem:landau_claim}
  Let $(S, \|\cdot\|)$ be a normed space and $f,g : \mathbb{N} \to S$ two functions. Suppose that $\|g(n)\| > 0$ for all $n \in \mathbb{N}$ and that $(\|g(n)\|)_{n \in \mathbb{N}}$ is a monotonically increasing sequence. Then the two following statements are equivalent:
  \begin{enumerate}
    \item $\|f(n)\| = \mathcal{O}\left(\|g(n)\|\right)$,
    \item $\max_{1 \leq k \leq n} \|f(k)\| = \mathcal{O}\left(\|g(n)\|\right)$.
  \end{enumerate}
\end{lemma}
\begin{proof}
  The implication from $(ii)$ to $(i)$ is obvious, so we prove the reverse.

  Suppose that $(i)$ holds. Then, by definition of the Landau symbol $\mathcal{O}$, there exist a constant $c_0 > 0$ and some threshold $N \in \mathbb{N}$ such that $\|f(n)\| \leq c_0 \,\|g(n)\|$ for all $n > N$. For any $1 \leq k \leq N$, let us define $c_k := \|f(k)\|/\|g(k)\|$ and $C := \max_{0 \leq k \leq N} c_k$. These quantities are well-defined because we assumed that $\|g(n)\|$ is strictly positive for all $n \in \mathbb{N}$. By construction, we have
  \begin{equation}
    \label{eq:landau1}
    \|f(n)\| \leq C \, \|g(n)\|
  \end{equation}
  for all $n \in \mathbb{N}$. Let $k_n \in \{1, \ldots, n\}$ be such that $\|f(k_n)\| = \max_{1 \leq k \leq n} \|f(k)\|$. Then,
  $$
    \max_{1 \leq k \leq n} \|f(k)\| = \|f(k_n)\| \leq C \, \|g(k_n)\| \leq C \, \|g(n)\|,
  $$
  where we have used the definition of $k_n$ in the first equality, \eqref{eq:landau1} in the first inequality and the fact that the sequence $(\|g(n)\|)_{n \in \mathbb{N}}$ is isotone by assumption in the second inequality. This is exactly the defining inequality of $(ii)$, which completes our proof.
\end{proof}

\begin{lemma}
  \label{lem:deltan_growthrate}
  Let $W = (W(\lambda))_{\lambda \geq 0}$ be a Brownian motion. Let $\Delta_n = (\Delta_n(\lambda))_{0 \leq \lambda \leq 1}$ be the element of $\ell^\infty[0,1]$ defined by
  $$
    \Delta_n(\lambda) := W(n\lambda) - W(\lfloor n\lambda\rfloor).
  $$
  Then $\|\Delta_n\|_{\infty} = \sup_{0 \leq \lambda \leq 1} |\Delta_n(\lambda)|
    = o_\mathbb{P}\left(n^\alpha\right)$ for any $\alpha > 0$.
\end{lemma}
\begin{proof}
  Fix some arbitrary $\alpha > 0$. For any $s \geq 0$, define $W^{(\lfloor n\lambda\rfloor)}(s) = W(\lfloor n\lambda \rfloor + s) - W(\lfloor n\lambda \rfloor)$. By standard arguments, one can show that $W^{(\lfloor n\lambda\rfloor)} := \left(W^{(\lfloor n\lambda\rfloor)}(s)\right)_{s \geq 0}$ is a Brownian motion. Because $|n\lambda - \lfloor n\lambda\rfloor| < 1$, there is some $t = t(n, \lambda) \in [0,1]$ such that $W^{(\lfloor n\lambda\rfloor)}(t) = \Delta_n(\lambda)$. Thus,
  $$
    \left|\Delta_n(\lambda)\right|
    \leq \sup_{0 \leq s \leq 1} \left|W^{(\lfloor n\lambda\rfloor)}(s)\right|,
  $$
  and so
  \begin{equation}
    \label{eq:deltan_sup_ungleichung}
    \|\Delta_n\|_{\infty} = \sup_{0 \leq \lambda \leq 1} |\Delta_n(\lambda)| \leq \sup_{0 \leq \lambda \leq 1}\sup_{0 \leq s \leq 1} \left|W^{(\lfloor n\lambda\rfloor)}(s)\right| = \max_{0 \leq k \leq n} \sup_{0 \leq s \leq 1} \left| W^{(k)}(s)\right|.
  \end{equation}
  By Lemma 16 in \cite{freedman:brownian_motion}, it holds for any Brownian motion $B$ that
  \begin{equation}
    \label{eq:freedman_ungleichung}
    \mathbb{P}\left(\sup_{0 \leq s \leq K} \left|B(s)\right| \geq b\right) \leq 4 \mathbb{P}\left(B(1) \geq b K^{-1/2}\right)
  \end{equation}
  for all $b, K > 0$.  Combining Eqs.\@ \eqref{eq:deltan_sup_ungleichung} and \eqref{eq:freedman_ungleichung} yields
  for any  $\varepsilon > 0$
  \begin{align*}
    \mathbb{P}\left(\|\Delta_n\|_{\infty} \geq \varepsilon n^\alpha\right) & \leq \mathbb{P}\left(\max_{0 \leq k \leq n} \sup_{0 \leq s \leq 1} \left| W^{(k)}(s)\right| \geq \varepsilon\, n^\alpha\right) \\
                                                                           & \leq \sum_{k=0}^n \mathbb{P}\left(\sup_{0 \leq s \leq 1} \left| W^{(k)}(s)\right| \geq \varepsilon\, n^\alpha\right)           \\
                                                                           & \leq 4\sum_{k=0}^n \mathbb{P}\left(W^{(k)}(1) \geq \varepsilon\, n^\alpha\right)                                               \\
                                                                           & = 4(n+1)\mathbb{P}\left(Z \geq \varepsilon\, n^\alpha\right)
  \end{align*}
  for a standard normally distributed random variable $Z$. Applying the tail bound
  $$
    \mathbb{P}(Z > t) \leq \exp (-t^2/2)/(\sqrt{2\pi}t)
  $$
  for all $t \geq 0$ gives
  \begin{align*}
    \mathbb{P}\left(\|\Delta_n\|_{\infty} \geq \varepsilon\, n^\alpha\right)
     & \leq c n^{1-\alpha} \exp (-(\varepsilon^2/2) n^{2\alpha} )  = o(1)
  \end{align*}
  for some constant $c > 0$ depending on $\varepsilon$, which yields
  $$
    \mathbb{P}\Big (\big \|n^{-\alpha}\Delta_n\big \|_{\infty} \geq \varepsilon\Big ) \xrightarrow[n \to \infty]{} 0,
  $$
  and  proves our claim.
\end{proof}

The following lemma is  a  simple consequence of the much stronger Theorem 4 in \cite{kuelbs_philipp:1980}. It is valid for $\alpha$-mixing processes, which is a weaker assumption than absolute regularity. The $\alpha$-mixing coefficient of two $\sigma$-algebras $\mathcal{A}$ and $\mathcal{B}$ is defined by
$$
  \alpha(\mathcal{A}, \mathcal{B}) = \sup_{A \in \mathcal{A}, B \in \mathcal{B}} |\mathbb{P}(A \cap B) - \mathbb{P}(A)\mathbb{P}(B)|,
$$
and a stochastic process $(\xi_k)_{k \in \mathbb{N}}$ is called $\alpha$-mixing or strongly mixing if
$$
  \alpha(n) = \sup_{j \in \mathbb{N}}\alpha\left(\mathcal{F}_1^j, \mathcal{F}_{j+n}^\infty\right) \xrightarrow[n \to \infty]{} 0,
$$
where $\mathcal{F}_i^j$ is the $\sigma$-algebra generated by $\xi_i, \ldots, \xi_j$. We again refer to \cite{bradley:2007} for more information.
\begin{lemma}
  \label{lem:kuelbs_philipp}
  Let $(\xi_n)_{n \in \mathbb{N}}$ be a strictly stationary and strongly mixing sequence of centred \mbox{$\mathbb{R}^d$-valued} random variables whose $(2+\varepsilon)$-moments are uniformly bounded for some $\varepsilon > 0$. Suppose that the mixing coefficients $(\alpha(n))_{n \in \mathbb{N}}$ satisfy $\alpha(n) = \mathcal{O}\left(n^{-r}\right)$ with $r = (1+\delta)(1+2/\varepsilon)$ for some $\delta > 0$. Define the partial sum process $S_n = (S_n(\lambda))_{0 \leq \lambda \leq 1}$ by
  $$
    S_n(\lambda) = \frac{1}{\sqrt{n}} \sum_{i=1}^{\lfloor n\lambda \rfloor} \xi_i.
  $$
  Then there exist independent standard Brownian motions $B_i = (B_i(\lambda))_{0 \leq \lambda \leq 1}$, $1 \leq i \leq d$, such that
  $$
    S_n \rightsquigarrow \Gamma^\frac{1}{2}
    \begin{pmatrix}
      B_1 \\ \vdots \\ B_d
    \end{pmatrix}
  $$
  in $(\ell^\infty[0,1])^d$, and the $ d\times d $ matrix $\Gamma = (\gamma_{ij})_{1 \leq i,j \leq d}$ is given by
  \begin{equation}
    \label{eq:gamma_identity}
    \gamma_{ij} = \mathbb{E}\left[\xi_{1i}\xi_{1j}\right] + \sum_{k=2}^\infty \mathbb{E}\left[\xi_{ki}\xi_{1j}\right] + \sum_{k=2}^\infty \mathbb{E}\left[\xi_{1i}\xi_{kj}\right] < \infty,
  \end{equation}
  with $\xi_{ki}$ denoting the $i$-th coordinate of the vector $\xi_k$.
\end{lemma}
\begin{proof}
  Without loss of generality assume that the $(2 +\varepsilon)$-moments of the random variables $\xi_n$ are bounded by $1$. If this is not the case, we can achieve it by rescaling the $\xi_n$.

  By Theorem 4 in \cite{kuelbs_philipp:1980}, there exist (possibly on a richer probability space) independent Brownian motions $W_i = (W_i(\lambda))_{0 \leq \lambda < \infty}$, $1 \leq i \leq d$, such that
  \begin{equation}
    \label{eq:kuelbs_philipp_invarianzprinzip}
    \sum_{i=1}^n \xi_i - \Gamma^{1/2} \begin{pmatrix}
      W_1(n) \\ \vdots \\ W_d(n)
    \end{pmatrix}
    = \mathcal{O}\left(n^{\frac{1}{2} - \gamma}\right)
  \end{equation}
  almost surely for some $\gamma > 0$ depending only on $\delta, \varepsilon$ and $d$. By Lemma \ref{lem:landau_claim}, Eq.\@ \eqref{eq:kuelbs_philipp_invarianzprinzip} is equivalent to
  $$
    \max_{1 \leq k \leq n} \left| \sum_{i=1}^k \xi_i - \Gamma^{1/2} \begin{pmatrix}
      W_1(k) \\ \vdots \\ W_d(k)
    \end{pmatrix}\right|
    = \mathcal{O}\left(n^{\frac{1}{2} - \gamma}\right)
  $$
  almost surely. Thus, for any $0 \leq \lambda \leq 1$,
  $$
    \left|\sum_{i=1}^{\lfloor n \lambda \rfloor} \xi_i - \Gamma^{1/2} \begin{pmatrix}
      W_1(\lfloor n \lambda \rfloor) \\ \vdots \\ W_d(\lfloor n \lambda \rfloor)
    \end{pmatrix}\right| \leq
    \max_{1 \leq k \leq n} \left|\sum_{i=1}^k \xi_i - \Gamma^{1/2} \begin{pmatrix}
      W_1(k) \\ \vdots \\ W_d(k)
    \end{pmatrix}\right|
    = \mathcal{O}\left(n^{\frac{1}{2} - \gamma}\right)
  $$
  almost surely, and so
  \begin{equation}
    \label{eq:kuelbs_philipp_konsequenz}
    \sup_{0 \leq \lambda \leq 1}\left|\frac{1}{\sqrt{n}}\sum_{i=1}^{\lfloor n \lambda \rfloor} \xi_i - \Gamma^{1/2} \frac{1}{\sqrt{n}}\begin{pmatrix}
      W_1(\lfloor n \lambda \rfloor) \\ \vdots \\ W_d(\lfloor n \lambda \rfloor)
    \end{pmatrix}\right|
    = \mathcal{O}\left(n^{- \gamma}\right)
  \end{equation}
  almost surely. For ease of notation, let us write $W = (W_1, \ldots, W_d)^\top$. Then we can write $W(\lfloor n\lambda \rfloor) = W(n\lambda) + \Delta_n(\lambda)$, where $\Delta_n(\lambda) = W(\lfloor n\lambda \rfloor) - W(n\lambda)$. Denote by $\Delta^{(1)}_n, \ldots, \Delta^{(d)}_n$ the coordinate functions of $\Delta_n$. Then $\|\Delta^{(i)}_n\|_\infty = o_\mathbb{P}(\sqrt{n})$ for all $1 \leq i \leq d$ by Lemma \ref{lem:deltan_growthrate}, and so
  $$
    \max_{1 \leq i \leq d} \big \|\Delta^{(i)}_n \big \|_{\infty} = o_\mathbb{P}(\sqrt{n}).
  $$
  For any $n \in \mathbb{N}$ and any $1 \leq i \leq d$, the rescaled Brownian motion $(W_i(n\lambda)/\sqrt{n})_{0 \leq \lambda < \infty}$ is again a standard Brownian motion, and so the entire process $(W(n\lambda)/\sqrt{n})_{0 \leq \lambda < \infty}$ is equal in distribution to $B := (B_1, \ldots, B_d)$, where the $B_i$ are independent standard Brownian motions. Thus,
  $$
    \frac{1}{\sqrt{n}} (W(\lfloor n\lambda \rfloor))_{0 \leq \lambda < \infty} = \frac{1}{\sqrt{n}} (W(n\lambda))_{0 \leq \lambda < \infty} + \frac{\Delta_n}{\sqrt{n}} \rightsquigarrow B.
  $$
  Using \eqref{eq:kuelbs_philipp_konsequenz}, we therefore get
  $$
    S_n = \Gamma^\frac{1}{2} \frac{1}{\sqrt{n}} (W(\lfloor n\lambda \rfloor))_{0 \leq \lambda \leq 1} + \mathcal{O}_{a.s.}\left(n^{-\gamma}\right) \rightsquigarrow\Gamma^\frac{1}{2} B.
  $$
\end{proof}

The empirical distance covariance $\mathrm{dcov}(\hat{\mathbb{P}}_n)$ defined in \eqref{det3}
is a $V$-statistic with kernel function $h'$ given in \eqref{eq:kern_h'}. Let
\begin{equation}
  \label{eq:kernel_h}
  h(z_1, \ldots, z_6) := \sum_{\sigma \in \mathfrak{S}_6} h'(z_{\sigma(1)}, \ldots, z_{\sigma(6)}),
\end{equation}
be the symmetrisation of $h'$, where $\mathfrak{S}_6$ denotes the symmetric group of order $6$. Then, as any $V$-statistic with kernel $g'$ is equal to the $V$-statistic with kernel $g$, where $g$ is the symmetrisation of $g'$, the empirical distance covariance may also be expressed as a $V$-statistic with kernel $h$.

We shall require the so-called Hoeffding decomposition. If $U_1, \ldots, U_n$ are observations from a stationary process with marginal distribution $\xi$ and $g$ is a symmetric kernel of order $m$, then the $V$-statistic with kernel $g$ based on $U_1, \ldots, U_n$ allows for a representation of the form
\begin{equation}
  \label{eq:hoeffding_decomposition}
  V_g(U_1, \ldots, U_n) = \sum_{i=0}^m {m \choose i} V_n^{(i)}(g; \xi),
\end{equation}
where each $V_n^{(i)}(g;\xi)$ is the $V$-statistic with kernel function
$$
  g_i(u_1, \ldots, u_i;\xi) = \sum_{k=0}^i {i \choose k} (-1)^{i-k} {\bf g}_k(u_1, \ldots, u_k; \xi),
$$
with
$$
  {\bf g}_k(u_1, \ldots, u_k;\xi) = \int g(u_1, \ldots, u_m) ~\mathrm{d}\xi^{m-k}(u_{k+1}, \ldots, u_m).
$$
In particular, $V^{(0)}(g;\xi)$ is equal to $\int g ~\mathrm{d}\xi^m$. The Hoeffding decomposition is a widely used tool in the theory of $U$- and $V$-statistics; for this particular version, see for instance \cite{denker_keller:1983}. Its name goes back to \cite{hoeffding:1948} who first used it for $U$-statistics. The advantage of this representation is that the $g_i$ are degenerate kernel functions, i.e.\@ it holds that $\mathbb{E}[g_i(U_1, u_2, \ldots, u_i)] = 0$ $\xi^{i-1}$-almost surely.

The first part of the following lemma is essentially Lemma 3 in \cite{arcones:1998}, but stated for $V$-statistics instead of $U$-statistics.

\begin{lemma}
  \label{lem:arcones_lemma}
  Let $(U_i)_{i \in \mathbb{N}}$ be a strictly stationary and absolutely regular process and let $g$ be a symmetric and degenerate kernel of order $m$. Assume that, for some $p > 2$ and some $M$ uniform in $i_1, \ldots, i_m$,
  $$
    \mathbb{E}\left[|g(U_{i_1}, \ldots, U_{i_m})|^p\right] < M < \infty.
  $$
  Then it holds that
  \begin{equation}
    \label{eq:arcones_claim}
    \mathbb{E}\left[|V_g(U_1, \ldots, U_n)|^2\right] \lesssim M^2 n^{-m} \left\{1 + \sum_{d=1}^n d^{m-1} \beta(d)^{(p-2)/p}\right\}
  \end{equation}
  and
  \begin{equation}
    \label{eq:increment_claim}
    \mathbb{E}\left[|V_g(U_1, \ldots, U_n) - V_g(U_1, \ldots, U_{n-1})|^2\right] \lesssim M^2 n^{-(m+1)} \left\{1 + \sum_{d=1}^n d^{m-1} \beta(d)^{(p-2)/p}\right\},
  \end{equation}
  where in both instances the symbol $\lesssim$ is hiding a constant depending only on $m$. If $\beta(n) = \mathcal{O}\left(n^{-r}\right)$ for some $r > mp/(p-2)$, then the sums in the upper bounds converge for $n \to \infty$.
\end{lemma}
\begin{proof}
  For Eq.\@ \eqref{eq:arcones_claim}, see Lemma 3 in \cite{arcones:1998} or Lemma 2.3 in \cite{kroll:2023}. We sketch the proof of Eq.\@ \eqref{eq:increment_claim}, which is also very similar to Lemma 3 in \cite{arcones:1998}. Throughout this proof, all constants hidden by the symbol $\lesssim$ will either be universal or only depend on $m$. Observe that
  \begin{align}
    \begin{split}
      \label{eq:v_diffs}
       & \mathbb{E}\left[\left(\sum_{1 \leq i_1, \ldots, i_{m} \leq n} g(U_{i_1}, \ldots, U_{i_m}) - \sum_{1 \leq i_1, \ldots, i_{m} \leq n-1} g(U_{i_1}, \ldots, U_{i_m})\right)^2\right]     \\
       & =\mathbb{E}\left[\left(\sum_{\emptyset \neq A \subseteq \{1, \ldots, m\}}\sum_{1 \leq i_1, \ldots, i_{m} \leq n : \{j ~|~ i_j = n\} = A} g(U_{i_1}, \ldots, U_{i_m}) \right)^2\right] \\
       & \lesssim \sum_{a = 1}^m \mathbb{E}\left[\left(\sum_{1 \leq i_1, \ldots, i_{m-a} \leq n-1} g(U_{i_1}, \ldots, U_{i_{m-a}}, U_n, \ldots, U_n)\right)^2\right].
    \end{split}
  \end{align}
  Write
  \begin{align*}
     & \mathbb{E}\left[\left(\sum_{1 \leq i_1, \ldots, i_{m-a} \leq n-1} g(U_{i_1}, \ldots, U_{i_{m-a}}, U_n, \ldots, U_n)\right)^2\right]                                                                \\
     & = \sum_{1 \leq i_1, \ldots, i_{2(m-a)} \leq n-1} \mathbb{E}\left[g(U_{i_1}, \ldots, U_{i_{m-a}}, U_n, \ldots, U_n)g(U_{i_{m-a+1}}, \ldots, U_{i_{2(m-a)}}, U_n, \ldots, U_n)\right]                \\
     & \lesssim \sum_{1 \leq i_1 \leq \ldots \leq i_{2(m-a)} \leq n-1} \mathbb{E}\left[g(U_{i_1}, \ldots, U_{i_{m-a}}, U_n, \ldots, U_n)g(U_{i_{m-a+1}}, \ldots, U_{i_{2(m-a)}}, U_n, \ldots, U_n)\right]
  \end{align*}
  From here, we can use essentially the same combinatorial arguments as in the proof of Lemma 3 in \cite{arcones:1998} to obtain the bound
  \begin{align*}
     & \mathbb{E}\left[\left(\sum_{1 \leq i_1, \ldots, i_{m-a} \leq n-1} g(U_{i_1}, \ldots, U_{i_{m-a}}, U_n, \ldots, U_n)\right)^2\right] \\
     & \quad \lesssim M^2 n^{m-a} \left\{1 + \sum_{d=1}^n d^{m-1} \beta(d)^{(p-2)/p}\right\},
  \end{align*}
  which, together with Eq.\@ \eqref{eq:v_diffs} yields
  \begin{align}
    \begin{split}
      \label{eq:sum_increment_bound}
       & \mathbb{E}\left[\left(\sum_{1 \leq i_1, \ldots, i_{m} \leq n} g(U_{i_1}, \ldots, U_{i_m}) - \sum_{1 \leq i_1, \ldots, i_{m} \leq n-1} g(U_{i_1}, \ldots, U_{i_m})\right)^2\right] \\
       & \lesssim M^2 n^{m-1} \left\{1 + \sum_{d=1}^n d^{m-1} \beta(d)^{(p-2)/p}\right\}.
    \end{split}
  \end{align}
  Finally,
  \begin{align*}
     & \mathbb{E}\left[|V_g(U_1, \ldots, U_n) - V_g(U_1, \ldots, U_{n-1})|^2\right]                                                                                                                      \\
     & \lesssim n^{-2m}\mathbb{E}\left[\left(\sum_{1 \leq i_1, \ldots, i_{m} \leq n} g(U_{i_1}, \ldots, U_{i_m}) - \sum_{1 \leq i_1, \ldots, i_{m} \leq n-1} g(U_{i_1}, \ldots, U_{i_m})\right)^2\right] \\
     & \qquad+ \left[\left(\frac{n}{n-1}\right)^m - 1\right]^2 \mathbb{E}\left[|V_g(U_1, \ldots, U_{n-1})|^2\right].
  \end{align*}
  Since $0 \leq [n/(n-1)]^m - 1 \lesssim (n-1)^{-1}$, the right-hand side can further be bounded by
  \begin{align*}
     & M^2 \left[n^{m-1-2m} + (n-1)^{-m-2}\right] \left\{1 + \sum_{d=1}^n d^{m-1} \beta(d)^{(p-2)/p}\right\} \\
     & \lesssim M^2 n^{-(m+1)} \left\{1 + \sum_{d=1}^n d^{m-1} \beta(d)^{(p-2)/p}\right\}
  \end{align*}
  due to Eqs.\@ \eqref{eq:arcones_claim} and \eqref{eq:sum_increment_bound}, which proves our claim.
\end{proof}

\begin{corollary}
  \label{cor:hoeffding_zerlegung_bestimmt_asymptotik}
  Let $(U_i)_{i \in \mathbb{N}}$ be a strictly stationary and absolutely regular process, and let $g$ be a symmetric kernel (not necessarily degenerate) of order $m$. Assume that for some $p > 2$ it holds that
  $$
    \mathbb{E}\left[|g(U_{i_1}, \ldots, U_{i_m})|^p\right] < M < \infty
  $$
  for some $M$ uniform in $i_1, \ldots, i_m$. Assume that the mixing coefficients satisfy $\beta(n) = \mathcal{O}\left(n^{-r}\right)$ for some $r > mp/(p-2)$. Let $V_g(\lambda)$ be the V-statistic with kernel $g$ based on the data $U_1, \ldots, U_{\lfloor n\lambda\rfloor}$. Then, writing $\xi = \mathcal{L}(U_1)$, it holds that
  $$
    \mathbb{E}\left[\sup_{0 \leq \lambda \leq 1} \left|\left(\frac{\lfloor n\lambda \rfloor}{n}\right)^c\left(V_g(\lambda) - \int g~\mathrm{d}\xi^m - {m \choose c} V_{\lfloor n\lambda\rfloor}^{(c)}(g;\xi)\right)\right|^2\right] = o\left(n^{-c}\right),
  $$
  for any $1 \leq c \leq m$ with the property that $g_{c'}$ is equal to $0$ almost surely for all $0 < c' < c$. The constant involved in the Landau symbol only depends on the constant $M$, the degrees $m$ and $c$, and the mixing rate of the sample generating process.
\end{corollary}
\begin{proof}
  Throughout the proof, constants hidden by the symbol $\lesssim$ will either be universal or depend only on the constant $M$, the degrees $m$ and $c$ and the sum $\sum_{d=1}^\infty d^{m-1} \beta(d)^{(p-2)/p}$. Assume without loss of generality that $m \geq 2$, since otherwise the integrand in the expected value is $0$, and so the claim is trivial. Using the Hoeffding decomposition \eqref{eq:hoeffding_decomposition}, we have
  \begin{align*}
     & \mathbb{E}\left[\sup_{0 \leq \lambda \leq 1} \left|\left(\frac{\lfloor n\lambda \rfloor}{n}\right)^c\left(V_g(\lambda) - \int g ~\mathrm{d}\xi^m - {m \choose c} V_{\lfloor n\lambda\rfloor}^{(c)}(g;\xi)\right)\right|^2\right] \\
     & \quad = \mathbb{E}\left[\sup_{0 \leq \lambda \leq 1}\left|\left(\frac{\lfloor n\lambda \rfloor}{n}\right)^c\sum_{i=c+1}^m {m \choose i} V_{\lfloor n\lambda\rfloor}^{(i)}(g;\xi)\right|^2\right]                                 \\
     & \quad \leq (m-c) \sum_{i=c+1}^m {m \choose i} \mathbb{E}\left[\sup_{0 \leq \lambda \leq 1}\left|\left(\frac{\lfloor n\lambda \rfloor}{n}\right)^cV_{\lfloor n\lambda\rfloor}^{(i)}(g;\xi)\right|^2\right].
  \end{align*}
  By Lemma \ref{lem:arcones_lemma}, we have for any $k = 1, \ldots, n$
  $$
    \mathbb{E}\left[\left|\left(\frac{k}{n}\right)^cV_k^{(i)}(g;\xi)\right|^2\right] \leq \gamma_i \frac{k^{2c-i}}{n^{2c}} \leq \gamma_i n^{-c-1},
  $$
  where $\gamma_i$ is a constant depending only on $M$, $i$ and the mixing rate of the sample generating process. By taking the maximum over all $\gamma_{c+1}, \ldots, \gamma_m$, we can make this bound depend only on $M$, $m$ and the mixing rate. Hence,
  \begin{equation}
    \label{eq:Vnk_bound}
    \mathbb{E}\left[\left|n^{c/2}\left(\frac{k}{n}\right)^c V_k^{(i)}(g;\xi)\right|^2\right]  \lesssim n^{-1},
  \end{equation}
  for any fixed $k = 1, \ldots, n$ and all $i \geq c+1$. Now fix some $i = c+1, \ldots, m$ and choose $r \in (0,1)$ such that $(1-r)c = r$, or equivalently $c = r/(1-r)$. Such a parameter $r$ exists since $r/(1-r) \uparrow \infty$ as $r \uparrow 1$. For $k = 0, \ldots, n$, define
  $$
    W_n\left(\frac{k}{n}\right) = \begin{cases}
      0                                                  & \quad \textrm{if } k \leq n^r, \\
      n^{c/2}\left(\frac{k}{n}\right)^c V_k^{(i)}(g;\xi) & \quad \textrm{if } k > n^r.
    \end{cases}
  $$
  Then $W_n((k+1)/n) - W_n(k/n) = 0$ for $k < n^r$, and for $k \geq n^r$ we have
  \begin{align*}
    \left|W_n\left(\frac{k+1}{n}\right) - W_n\left(\frac{k}{n}\right)\right| & \leq n^{c/2} \left|\left(\frac{k+1}{n}\right)^c - \left(\frac{k}{n}\right)^c\right| \left|V_{k+1}^{(i)}(g;\xi)\right|       \\
                                                                             & \qquad + n^{c/2} \left(\frac{k}{n}\right)^c \left|V_{k+1}^{(i)}(g;\xi) - V_{k}^{(i)}(g;\xi)\right|                          \\
                                                                             & \lesssim n^{c/2 - 1}\left|V_{k+1}^{(i)}(g;\xi)\right| + n^{-c/2} k^c \left|V_{k+1}^{(i)}(g;\xi) - V_{k}^{(i)}(g;\xi)\right|
  \end{align*}
  since $[(k+1)/n]^c - (k/n)^c \lesssim n^{-1}$. By taking second moments and using Lemma \ref{lem:arcones_lemma}, we get
  \begin{align}
    \begin{split}
      \label{eq:increments_Wn}
      \mathbb{E}\left[\left|W_n\left(\frac{k+1}{n}\right) - W_n\left(\frac{k}{n}\right)\right|^2\right] & \lesssim n^{c-2} (k+1)^{-i} + n^{-c} k^{2c} (k+1)^{-(i+1)} \\
                                                                                                        & \lesssim n^{c-2} k^{-(c+1)} + n^{-c}k^{c- 2}               \\
                                                                                                        & \lesssim n^{c-2}n^{-r(c+1)} + n^{-2}                       \\
                                                                                                        & \lesssim n^{(1-r)c - r - 2} + n^{-2}                       \\
                                                                                                        & \lesssim n^{-2},
    \end{split}
  \end{align}
  for $k \geq n^r$, since $i \geq c+1$ and $(1-r)c - r = 0$. Now extend $W_n$ to a process indexed in $[0,1]$ by linear interpolation of the values $W_n(k/n)$, $k = 0, \ldots, n$. Then $W_n(\lambda) = 0$ for $0 \leq \lambda \leq n^{r-1}$, and so
  \begin{align*}
     & \mathbb{E}\left[\sup_{0 \leq \lambda \leq n^{r-1}} \left|n^{c/2}\left(\frac{\lfloor n\lambda \rfloor}{n}\right)^c V_{\lfloor n\lambda \rfloor}^{(i)}(g;\xi) - W_n(\lambda)\right|^2\right] \\
     & = \mathbb{E}\left[\sup_{0 \leq \lambda \leq n^{r-1}}\left|n^{c/2}\left(\frac{\lfloor n\lambda \rfloor}{n}\right)^c V_{\lfloor n\lambda \rfloor}^{(i)}(g;\xi)\right|^2\right]               \\
     & \lesssim n \cdot n^{r-1} n^{-1}                                                                                                                                                            \\
     & = n^{r-1}
  \end{align*}
  by Eq.\@ \eqref{eq:Vnk_bound}. On the other hand,
  \begin{align*}
     & \mathbb{E}\left[\sup_{n^{r-1} \leq \lambda \leq 1} \left|n^{c/2}\left(\frac{\lfloor n\lambda \rfloor}{n}\right)^c V_{\lfloor n\lambda \rfloor}^{(i)}(g;\xi) - W_n(\lambda)\right|^2\right] \\
     & \leq \mathbb{E}\left[\max_{n^r \leq k \leq n}\left|W_n\left(\frac{k+1}{n}\right) - W_n\left(\frac{k}{n}\right)\right|^2\right]                                                             \\
     & \leq n \max_{n^r \leq k \leq n}\mathbb{E}\left[\left|W_n\left(\frac{k+1}{n}\right) - W_n\left(\frac{k}{n}\right)\right|^2\right]                                                           \\
     & \lesssim n^{-1}
  \end{align*}
  by Eq.\@ \eqref{eq:increments_Wn}. Therefore,
  \begin{equation}
    \label{eq:Vn_Wn_approximation}
    \mathbb{E}\left[\sup_{0\leq \lambda \leq 1} \left|n^{c/2}\left(\frac{\lfloor n\lambda \rfloor}{n}\right)^c V_{\lfloor n\lambda \rfloor}^{(i)}(g;\xi) - W_n(\lambda)\right|^2\right] \lesssim n^{r-1} + n^{-1} \xrightarrow[n \to \infty]{} 0.
  \end{equation}
  Let us now investigate the increments $W_n(\lambda_1) - W_n(\lambda_2)$ for arbitrary $\lambda_1, \lambda_2$. Recall that we constructed $W_n$ through linear interpolation; in particular, this means that
  $$
    W_n(\lambda) = W_n\left(\frac{k}{n}\right) + n\left(\lambda - \frac{k}{n}\right)\left[W_n\left(\frac{k+1}{n}\right) - W_n\left(\frac{k}{n}\right)\right]
  $$
  if $k/n \leq \lambda < (k+1)/n$. Pick $0 \leq \lambda_1 \leq \lambda_2 \leq 1$. Assume first that $|\lambda_1 - \lambda_2| \leq 1/n$. Then $\lambda_1, \lambda_2$ lie either in the same interval or in two adjacent intervals. Let us consider the first case, i.e.\@ assume that there is some integer $k$ with $k/n \leq \lambda_1 \leq \lambda_2 < (k+1)/n$. Then
  $$
    \mathbb{E}\left[|W_n(\lambda_1) - W_n(\lambda_2)|^2\right] = n^2|\lambda_1 - \lambda_2|^2 \mathbb{E}\left[\left|W_n\left(\frac{k+1}{n}\right) - W_n\left(\frac{k}{n}\right)\right|^2\right] \lesssim |\lambda_1 - \lambda_2|^2
  $$
  by Eq.\@ \eqref{eq:increments_Wn}. Now assume that $|\lambda_1 - \lambda_2| \leq 1/n$ and $k/n \leq \lambda_1 < (k+1)/n \leq \lambda_2 < (k+2)/n$ for some integer $k$. Then
  \begin{align*}
    W_n(\lambda_2) - W_n(\lambda_1) & = W_n\left(\frac{k+1}{n}\right) + n\left(\lambda_2 - \frac{k+1}{n}\right)\left[W_n\left(\frac{k+2}{n}\right) - W_n\left(\frac{k+1}{n}\right)\right] \\
                                    & \quad - W_n\left(\frac{k}{n}\right) - n\left(\lambda_1 - \frac{k}{n}\right)\left[W_n\left(\frac{k+1}{n}\right) - W_n\left(\frac{k}{n}\right)\right] \\
                                    & = n\left(\lambda_2 - \frac{k+1}{n}\right)\left[W_n\left(\frac{k+2}{n}\right) - W_n\left(\frac{k+1}{n}\right)\right]                                 \\
                                    & \quad - n\left(\lambda_1 - \frac{k+1}{n}\right)\left[W_n\left(\frac{k+1}{n}\right) - W_n\left(\frac{k}{n}\right)\right].
  \end{align*}
  Take second moments and apply Eq.\@ \eqref{eq:increments_Wn} to see that
  \begin{align*}
    \mathbb{E}\left[|W_n(\lambda_1) - W_n(\lambda_2)|^2\right] & \leq n^2\left|\lambda_2 - \frac{k+1}{n}\right|^2 \mathbb{E}\left[\left|W_n\left(\frac{k+2}{n}\right) - W_n\left(\frac{k+1}{n}\right)\right|^2\right] \\
                                                               & \quad + n^2\left|\lambda_1 - \frac{k+1}{n}\right|\mathbb{E}\left[\left|W_n\left(\frac{k+1}{n}\right) - W_n\left(\frac{k}{n}\right)\right|^2\right]   \\
                                                               & \lesssim \left|\lambda_2 - \frac{k+1}{n}\right|^2 + \left|\lambda_1 - \frac{k+1}{n}\right|^2                                                         \\
                                                               & \leq 2 \left|\lambda_1 - \lambda_2\right|^2,
  \end{align*}
  since $\lambda_1 \leq (k+1)/n \leq \lambda_2$. Finally, assume that $|\lambda_1 - \lambda_2| > 1/n$. Then there exist integers $k_1 < k_2$ such that $k_1/n \leq \lambda_1 < (k_1+1)/n \leq k_2/n \leq \lambda_2 < (k_2+1)/n$. By the previous bounds for $|\lambda_1 - \lambda_2| \leq 1/n$ and Eq.\@ \eqref{eq:increments_Wn}, we get
  \begin{align*}
    \mathbb{E}\left[|W_n(\lambda_1) - W_n(\lambda_2)|^2\right] & \lesssim \left|\lambda_1 - \frac{k_1}{n}\right|^2 + \left|\lambda_2 - \frac{k_2 + 1}{n}\right|^2 + \mathbb{E}\left[\left|W_n\left(\frac{k_1}{n}\right) - W_n\left(\frac{k_2+1}{n}\right)\right|^2\right] \\
                                                               & \leq 2 |\lambda_1 - \lambda_2|^2 + (k_2 - k_1) \sum_{l = k_1}^{k_2} \mathbb{E}\left[\left|W_n\left(\frac{l}{n}\right) - W_n\left(\frac{l+1}{n}\right)\right|^2\right]                                    \\
                                                               & \lesssim 2|\lambda_1 - \lambda_2|^2 + \frac{(k_2 - k_1)^2}{n^2}                                                                                                                                          \\
                                                               & \leq 4|\lambda_1 - \lambda_2|^2.
  \end{align*}
  Collecting all three cases, we see that
  $$
    \mathbb{E}\left[|W_n(\lambda_1) - W_n(\lambda_2)|^2\right] \lesssim |\lambda_1 - \lambda_2|^2
  $$
  for any $\lambda_1, \lambda_2 \in [0,1]$. On the other hand, it is not difficult to see that the expected value on the left-hand side can also be bounded by $n^{-1}$ as a consequence of Eq.\@ \eqref{eq:Vnk_bound}. Therefore,
  $$
    \mathbb{E}\left[|W_n(\lambda_1) - W_n(\lambda_2)|^2\right] \lesssim d_n^2(\lambda_1, \lambda_2),
  $$
  where $d_n(\lambda_1, \lambda_2) = |\lambda_1 - \lambda_2| \land n^{-1/2}$. $d_n$ defines a semimetric on $[0,1]$ with respect to which the diameter of $[0,1]$ is equal to $n^{-1/2}$. If $D(\varepsilon, d_n)$ denotes the maximum number of $\varepsilon$-separated (with respect to $d_n$) points in $[0,1]$, then $D(\varepsilon, d_n) \lesssim \varepsilon^{-1}$ for any $0 < \varepsilon < n^{-1/2}$. Hence, by Corollary 2.2.5 in \cite{wellner_vandervaart:1996}, and because $W_n(0) = 0$,
  $$
    \left\|\sup_{0 \leq \lambda \leq 1} \left|W_n(\lambda)\right|\right\|_{L_2} \lesssim \int_0^{1/\sqrt{n}} \varepsilon^{-1/2} ~\mathrm{d}\varepsilon \lesssim n^{-1/4}.
  $$
  Squaring both sides and combining this with Eq.\@ \eqref{eq:Vn_Wn_approximation} proves our claim.
\end{proof}

\begin{proof}[Proof of Theorem \ref{thm:dcov_prozesskonvergenz}]
  Let us first assume that $\mathrm{dcov}(X,Y) = 0$. Using the definition of $h$ in Eq.\@ \eqref{eq:kernel_h}, Minkowski's inequality and Hölder's inequality, one can show that
  \begin{equation}
    \label{eq:kernel_h_moment_bound}
    \|h(Z_{i_1}, \ldots, Z_{i_6})\|_{L_p} \leq C\|X_1\|_{L_{2p}} \|Y_1\|_{L_{2p}}
  \end{equation}
  for some universal constant $C > 0$ and all $p \geq 1$. Thus, $h$ has uniformly bounded $(2+\varepsilon/2)$-moments. By Corollary \ref{cor:hoeffding_zerlegung_bestimmt_asymptotik} with $c = 1$ and $m = 6$, it then holds that
  $$
    \frac{\lfloor n\lambda \rfloor}{n}\left[\mathrm{dcov}(\hat{\mathbb{P}}_{\lfloor n\lambda \rfloor}) - \mathrm{dcov}(\mathbb{P}^{X,Y})\right] = \frac{1}{n} \sum_{i=1}^{\lfloor n\lambda \rfloor} h_1(Z_i;\mathbb{P}^{X,Y}) + R_{\lfloor n\lambda \rfloor}(\mathbb{P}^{X,Y}),
  $$
  where
  $$
    \sup_{0 \leq \lambda \leq 1} \left|R_{\lfloor n\lambda \rfloor}(\mathbb{P}^{X,Y})\right| = o_\mathbb{P}\left(n^{-1/2}\right).
  $$
  Let $\mathbb{P}^{X,X}$ and $\mathbb{P}^{Y,Y}$ denote the distributions of $(X,X)$ and $(Y,Y)$, respectively, then analogous identities hold for $\mathrm{dcov}_{\lfloor n\lambda \rfloor}(X,X) - \mathrm{dcov}(X,X)$ and $\mathrm{dcov}_{\lfloor n\lambda \rfloor}(Y,Y) - \mathrm{dcov}(Y,Y)$. Thus,
  \begin{align*}
    X_n(\lambda) - X(\lambda) & = \frac{{\lfloor n\lambda \rfloor}}{n}
    \begin{pmatrix}
      \mathrm{dcov}_{\lfloor n\lambda \rfloor}(X,Y) - \mathrm{dcov}(X,Y) \\ \mathrm{dcov}_{\lfloor n\lambda \rfloor}(X,X) - \mathrm{dcov}(X,X) \\ \mathrm{dcov}_{\lfloor n\lambda \rfloor}(Y,Y) - \mathrm{dcov}(Y,Y)
    \end{pmatrix}
    - \left(\lambda - \frac{{\lfloor n\lambda \rfloor}}{n}\right)
    \begin{pmatrix}
      \mathrm{dcov}(X,Y) \\ \mathrm{dcov}(X,X) \\ \mathrm{dcov}(Y,Y)
    \end{pmatrix}                                    \\
                              & = \frac{1}{n} \sum_{i=1}^{\lfloor n\lambda \rfloor}
    \begin{pmatrix}
      h_1(Z_i;\mathbb{P}^{X,Y}) \\ h_1((X_i, X_i);\mathbb{P}^{X,X}) \\ h_1((Y_i, Y_i);\mathbb{P}^{Y,Y})
    \end{pmatrix}
    +
    \begin{pmatrix}
      R_{\lfloor n\lambda \rfloor}(\mathbb{P}^{X,Y}) \\ R_{\lfloor n\lambda \rfloor}(\mathbb{P}^{X,X}) \\ R_{\lfloor n\lambda \rfloor}(\mathbb{P}^{Y,Y})
    \end{pmatrix}
    \\
                              & \qquad - \left(\lambda - \frac{{\lfloor n\lambda \rfloor}}{n}\right)
    \begin{pmatrix}
      \mathrm{dcov}(X,Y) \\ \mathrm{dcov}(X,X) \\ \mathrm{dcov}(Y,Y)
    \end{pmatrix}                                    \\
                              & = \frac{1}{n} \sum_{i=1}^{\lfloor n\lambda \rfloor}
    \begin{pmatrix}
      h_1(Z_i;\mathbb{P}^{X,Y}) \\ h_1((X_i, X_i);\mathbb{P}^{X,X}) \\ h_1((Y_i, Y_i);\mathbb{P}^{Y,Y})
    \end{pmatrix}
    + R_n^*(\lambda),
  \end{align*}
  with $\sup_{0 \leq \lambda}\|R_n^*(\lambda)\|_2 = o_\mathbb{P}(n^{-1/2})$. By Lemma \ref{lem:kuelbs_philipp}, we therefore get
  $$
    \sqrt{n}(X_n - X)  \rightsquigarrow \Gamma^\frac{1}{2} \begin{pmatrix}
      B_1 \\ B_2 \\ B_3
    \end{pmatrix},
  $$
  where $B_1, B_2, B_3$ are independent standard Brownian motions. The covariance matrix $\Gamma$ is described in the statement of Lemma \ref{lem:kuelbs_philipp} with
  \begin{equation}
    \label{eq:xi_identity}
    \xi_i = \begin{pmatrix}
      h_1(Z_i;\mathbb{P}^{X,Y}) \\ h_1((X_i, X_i);\mathbb{P}^{X,X}) \\ h_1((Y_i, Y_i);\mathbb{P}^{Y,Y})
    \end{pmatrix}.
  \end{equation}

  This proves the general convergence claim. Let us now investigate the structure of the matrix $\Gamma$ if $\mathrm{dcov}(X,Y) = 0$. Using the definition of the kernel $h$, it is not difficult to show that $h_1(\cdot \,;\mathbb{P}^{X,Y})$ is equal to $0$ $\mathbb{P}^{X,Y}$-almost surely if and only if $\mathrm{dcov}(X,Y) = 0$. Therefore, the first coordinates of the special random vectors $\xi_i$ in Eq.\@ \eqref{eq:xi_identity} are $0$ almost surely. On the other hand, the formula for the coordinates of $\Gamma$ in Eq.\@ \eqref{eq:gamma_identity} can be rewritten as
  \begin{equation}
    \label{eq:gamma_formel_alternativ}
    \gamma_{ij} = \lim_{n \to \infty} \mathrm{Cov}\left(\frac{1}{\sqrt{n}} \sum_{k=1}^n \xi_{ki}, \frac{1}{\sqrt{n}} \sum_{k=1}^n \xi_{kj}\right).
  \end{equation}
  This can be verified either by working out this limiting covariance directly and verifying that it is in fact equal to the expression in Eq.\@ \eqref{eq:gamma_identity}, or alternatively by applying a standard central limit theorem for mixing data to the random vectors $S_n(1)$, $n \in \mathbb{N}$, which results in a limiting covariance matrix as specified by Eq.\@ \eqref{eq:gamma_formel_alternativ}. But since the projection of $S_n$ onto $S_n(1)$ is continuous, this covariance matrix must be exactly $\Gamma$. No matter how we arrive at the conclusion, it is now clear from Eq.\@ \eqref{eq:gamma_formel_alternativ} that $\gamma_{ij} = 0$ if $i = 1$ or $j = 1$, since the random vectors $\xi_i$ have $0$ as their first coordinates almost surely.
\end{proof}
\begin{remark}

  An immediate consequence of the last part of the proof of Theorem \ref{thm:dcov_prozesskonvergenz} is that $\gamma_{11} = 0$ if $\mathrm{dcov}(X,Y) = 0$. This is really an if and only if statement, i.e.\@ it holds that $\gamma_{11} > 0$ if $\mathrm{dcov}(X,Y) > 0$. We will show this in the proof of Theorem \ref{thm:schwache_konvergenz_dcov}. While we do not need it, it is also true that $\gamma_{22}, \gamma_{33} > 0$ under the assumptions of the theorem, since $\mathrm{dcov}(X,X), \mathrm{dcov}(Y,Y) > 0$ by Proposition 2.3 in \cite{lyons:2013}. This implies that $h_1(\cdot\, ; \mathbb{P}^{X,X})$ and $h_1(\cdot\, ;\mathbb{P}^{Y,Y})$ do not completely vanish on their respective measure's support, and from here one can proceed as for $\gamma_{11}$ in the proof of Theorem \ref{thm:schwache_konvergenz_dcov} below.
\end{remark}

For the sake of convenience, we recall here the definition of Hadamard differentiability, which we take from Chapter 3.9 in \cite{wellner_vandervaart:1996}. Let $D$ and $E$ be two normed spaces and $D_\phi, D_0$ subsets of $D$. A map $\phi : D_\phi \to E$ is called Hadamard differentiable at $\theta \in D_\phi$ tangentially to $D_0$ if there exists a linear and continuous map $\phi_\theta' : D_0 \to E$ such that $t_n^{-1}[\phi(\theta + t_n h_n) - \phi(\theta)] \to \phi_\theta'(h)$ as $n \to \infty$ for all $h \in D_0$, all real-valued sequences $t_n \to 0$ and all $D$-valued sequences $h_n$, $n \in \mathbb{N}$, with the property that $h_n \to h$ and $\theta + t_n h_n \in D_\phi$ for all $n$.

\begin{lemma}
  \label{lem:hadamard_diffbar}
  Fix some $\varepsilon > 0$ and consider the subset $D_T \subseteq (\ell^\infty[0,1])^3$ consisting of all $g = (g_1, g_2, g_3)$ such that
  $$
    \varepsilon \leq \inf_{0 \leq \lambda \leq 1} \lambda^{-1} g_j(\lambda) \leq \sup_{0 \leq \lambda \leq 1} \lambda^{-1} g_j(\lambda) < \infty
  $$
  for $j = 2,3$, and
  $$
    0 \leq \inf_{0 \leq \lambda \leq 1} \lambda^{-1} g_1(\lambda) \leq \sup_{0 \leq \lambda \leq 1} \lambda^{-1} g_1(\lambda) \leq \varepsilon^{-1}.
  $$
  For any fixed $r > 1/2$, define the function
  \begin{align*}
    T : D_T & \to \ell^\infty[0,1],                                                                         \\
    g       & \mapsto \left[\lambda \mapsto \lambda^r g_1(\lambda)/\sqrt{g_2(\lambda) g_3(\lambda)}\right].
  \end{align*}
  Then $T$ is Hadamard-differentiable everywhere on $D_T$ tangentially to
  $$
    C = \left\{g = (g_1, g_2, g_3) \in (\ell^\infty[0,1])^3 ~\Big|~  \|g(\lambda)\|_2 = \mathcal{O}\left(\sqrt{\lambda \log (1/\lambda)}\right) \textrm{ for } \lambda \downarrow 0\right\},
  $$
  and its Hadamard derivative is given by
  $$
    T_f'(h)[\lambda] = \lambda^r\left\langle \left(\frac{1}{\sqrt{f_2 f_3}}, - \frac{f_1}{2 \sqrt{f_2^3 f_3}}, - \frac{f_1}{2 \sqrt{f_2 f_3^3}}\right)^\top(\lambda),
    h(\lambda) \right\rangle.
  $$
\end{lemma}
\begin{proof}
  The map $T$ is a composition of several easier maps, and one can establish Hadamard differentiability of $T$ via the chain rule for Hadamard derivatives \citep[Lemma 3.9.3 in][]{wellner_vandervaart:1996}. For instance, one can write
  $$
    T = I \circ A,
  $$
  where, with $L_r : (\ell^\infty[0,1])^3 \to \ell^\infty[0,1]$ denoting the constant function $L_r : f \mapsto (\lambda \mapsto \lambda^r)$, $I$ is given by
  \begin{align*}
    I : D_I    & \to \ell^\infty[0,1], \\
    (g_1, g_2) & \mapsto L_rg_1/g_2.
  \end{align*}
  $D_I \subseteq (\ell^\infty[0,1])^2$ is defined just as $D_T$ but with two coordinate functions instead of three. More precisely, let $D_I$ consist of all functions of the form $\pi_{1,2} \circ g$ for some $g \in D_T$, where $\pi_{1,2}$ denotes the projection onto the first two coordinates. Finally, define
  \begin{align*}
    A : D_T & \to (\ell^\infty[0,1])^2,      \\
    g       & \mapsto (g_1, \sqrt{g_2 g_3}).
  \end{align*}
  It is not hard to verify that $A$ is Hadamard differentiable everywhere on $D_T$ with derivative
  $$
    A'_g(h) = \begin{pmatrix}
      h_1 \\
      [g_3 h_2 + g_2 h_3]/[2\sqrt{g_2g_3}]
    \end{pmatrix}.
  $$
  Now fix any $g \in D_I$, and let $t_n$ be a real-valued sequence converging to $0$ and $h_n \in (\ell^\infty[0,1])^2$ a sequence converging uniformly to some $h \in C$ such that $g + t_n h_n \in D_I$ for all $n \in \mathbb{N}$. Then
  $$
    \frac{I(g + t_n h_n) - I(g)}{t_n} = L_rt_n^{-1}\left[\frac{g_1 + t_n h_{n,1}}{g_2 + t_n h_{n,2}} - \frac{g_1}{g_2}\right] \xrightarrow[n \to \infty]{} L_rg_2^{-2}[g_2 h_1 - g_1 h_2] = I_g'(h)
  $$
  uniformly on $[0, 1]$. $I_g$ is obviously linear and continuous with respect to uniform convergence. Because the norms of the sample paths of $h \in C$ are of order $\mathcal{O}\left(\sqrt{\lambda \log(1/\lambda)}\right)$ by construction of the set $C$, we have
  \begin{align*}
    \left|I_g'(h)[\lambda]\right| & \leq L_r(\lambda) g_2^{-2}(\lambda) \|h(\lambda)\|_2\left\|\begin{pmatrix}
                                                                                                 g_2(\lambda) \\ -g_1(\lambda)
                                                                                               \end{pmatrix}\right\|_2 \\
                                  & \leq C_{\varepsilon,h} \lambda^r \lambda^{-2} \sqrt{\lambda \log(1/\lambda)} \lambda    \\
                                  & = \lambda^{r-1/2} \sqrt{\log(1/\lambda)}
  \end{align*}
  for some constant $C_{\varepsilon, h}$ depending only on $\varepsilon$ and $h$. The right-hand side is bounded uniformly in $\lambda \in [0,1]$, and so $I_g'(h) \in \ell^\infty[0,1]$. $I$ is therefore Hadamard-differentiable everywhere on $D_I$ tangentially to $C$. Furthermore, it is easy to see that $A(D_T) \subseteq D_I$ and $A_g'(C) \subseteq C$. By the chain rule for Hadamard derivatives (cited above), $T = I \circ A$ is Hadamard-differentiable everywhere on $D_T$ tangentially to $C$ with derivative
  \begin{align*}
    T_f'(h) = I'_{A(f)} \circ A_f'(h) & = L_r \cdot (f_2 f_3)^{-1} \left[\sqrt{f_2 f_3} h_1 - f_1 \left(f_3 h_2 + f_2 h_3\right) / \left(2 \sqrt{f_2 f_3}\right)\right] \\
                                      & = L_r \cdot \left(\frac{1}{\sqrt{f_2 f_3}} h_1 - \frac{f_1}{2\sqrt{f_2^3 f_3}} h_2 - \frac{f_1}{2\sqrt{f_2 f_3^2}}h_3\right).
  \end{align*}
\end{proof}

\begin{lemma}
  \label{lem:dcor_hilfslemma}
  Define the process $X^*_n \in (\ell^\infty[0,1])^3$ by
  $$
    X_n^*(\lambda) = X_n(\lambda)
  $$
  if $\mathrm{dcov}_{\lfloor n\lambda\rfloor}(X,X) \land \mathrm{dcov}_{\lfloor n\lambda\rfloor}(Y,Y) > [\mathrm{dcov}(X,X) \land \mathrm{dcov}(Y,Y)]/2$, and
  $$
    X_n^*(\lambda) = \frac{\lfloor n\lambda \rfloor}{n} \begin{pmatrix}
      \mathrm{dcov}(X,Y) \\ \mathrm{dcov}(X,X) \\ \mathrm{dcov}(Y,Y)
    \end{pmatrix}
  $$
  otherwise. Then
  $$
    \sup_{0 \leq \lambda \leq 1} \sqrt{n}\left\|X_n^*(\lambda) - X_n(\lambda)\right\|_2 \xrightarrow[n \to \infty]{a.s.} 0,
  $$
  and
  $$
    \sup_{0 \leq \lambda \leq 1} \sqrt{n}\left|\lambda^r\mathrm{dcor}_{\lfloor n \lambda \rfloor}(X,Y) - T(X_n^*)(\lambda)\right| \xrightarrow[n \to \infty]{a.s.} 0,
  $$
  for any $r > 1/2$, where $T$ is the map from Lemma \ref{lem:hadamard_diffbar}.
\end{lemma}
\begin{proof}
  Since $\mathrm{dcov}_n(X,X) \to \mathrm{dcov}(X,X)$ and $\mathrm{dcov}_n(Y,Y) \to \mathrm{dcov}(Y,Y)$ almost surely \citep[Theorem 1 in][]{kroll:2022}, there exists almost surely some $K \in \mathbb{N}$ such that $\mathrm{dcov}_k(X,X) > \mathrm{dcov}(X,X)/2$ and $\mathrm{dcov}_k(Y,Y) > \mathrm{dcov}(Y,Y)/2$ for all $k > K$. Therefore,
  \begin{align*}
    \sup_{0 \leq \lambda \leq 1} \sqrt{n}\left\|X_n^*(\lambda) - X_n(\lambda)\right\|_2 = \max_{k = 1, \dots, K} \sqrt{n}\left\|\frac{k}{n}\begin{pmatrix}
                                                                                                                                             \mathrm{dcov}_k(X,Y) - \mathrm{dcov}(X,Y) \\ \mathrm{dcov}_k(X,X) - \mathrm{dcov}(X,X) \\ \mathrm{dcov}_k(Y,Y) - \mathrm{dcov}(Y,Y)
                                                                                                                                           \end{pmatrix}\right\|_2
    \xrightarrow[n \to \infty]{a.s.} 0.
  \end{align*}
  For the difference of the distance correlations, we note that
  $$
    T(X_n^*)(\lambda) = \begin{cases}
      \lambda^r \mathrm{dcor}_{\lfloor n\lambda\rfloor}(X,Y) & \quad \textrm{if } X_n^*(\lambda) = X_n(\lambda), \\
      \lambda^r\mathrm{dcor}(X,Y)                            & \quad \textrm{otherwise}.
    \end{cases}
  $$
  Therefore, with the same almost surely finite $K \in \mathbb{N}$ as before,
  \begin{align*}
    \sup_{0 \leq \lambda \leq 1} \sqrt{n}\left|\lambda^r\mathrm{dcor}_{\lfloor n \lambda \rfloor}(X,Y) - T(X_n^*)(\lambda)\right| & \leq K^r n^{1/2 - r} \max_{k=1, \ldots, K} \left|\mathrm{dcor}_k(X,Y) - \mathrm{dcor}(X,Y)\right| \\
                                                                                                                                  & \xrightarrow[n \to \infty]{a.s.} 0.
  \end{align*}
\end{proof}

\begin{proof}[Proof of Corollary \ref{cor:dcor_processkonvergenz}]
  We use the notation $X_n, X, X_n^*$ and $T$ as defined in Theorem \ref{thm:dcov_prozesskonvergenz} and Lemmas \ref{lem:hadamard_diffbar} and \ref{lem:dcor_hilfslemma}. By Theorem \ref{thm:dcov_prozesskonvergenz} and Lemma \ref{lem:dcor_hilfslemma}, we have $\sqrt{n}(X_n^* - X) \rightsquigarrow \Gamma^{1/2} W$. By Lemma \ref{lem:hadamard_diffbar} and the functional Delta method \citep[Theorem 3.9.4 in][]{wellner_vandervaart:1996}, we have $\sqrt{n}[T(X_n^*) - T(X)] \rightsquigarrow T_X'(\Gamma^{1/2} W)$. But since $T(X) = Z$, Lemma \ref{lem:dcor_hilfslemma} then gives us
  $$
    \sqrt{n}(Z_n - Z) = \sqrt{n}[T(X_n^*) - T(X)] + o_{a.s.}(1) \rightsquigarrow T_X'\left(\Gamma^{1/2}W\right) = L.
  $$
  The additional statement that $L$ is tight, mean-zero and Gaussian follows from the continuity and linearity of $T_X'$.
\end{proof}

\subsection{Proofs of Theorems \ref{thm:schwache_konvergenz_dcov} and \ref{thm:schwache_konvergenz_dcor}}
\begin{proof}[Proof of Theorem \ref{thm:schwache_konvergenz_dcor}]
  We use the notation of Corollary \ref{cor:dcor_processkonvergenz} with the choice $r = 2$. By that corollary, $\sqrt{n}(Z_n - Z) \rightsquigarrow L$, and
  \begin{align*}
    L(\lambda) & = \lambda \left\langle \begin{pmatrix}
                                          \frac{1}{\sqrt{\mathrm{dcov}(X,X) \mathrm{dcov}(Y,Y)}} \\ - \frac{\mathrm{dcov}(X,Y)}{2\sqrt{\mathrm{dcov}(X,X)^3\mathrm{dcov}(Y,Y)}} \\ - \frac{\mathrm{dcov}(X,Y)}{2\sqrt{\mathrm{dcov}(X,X)\mathrm{dcov}(Y,Y)^3}}
                                        \end{pmatrix},
    \Gamma^\frac{1}{2} W(\lambda)\right\rangle                                                                                                                                                                         \\
               & = \lambda \left\langle \Gamma^\frac{1}{2}\begin{pmatrix}
                                                            \frac{1}{\sqrt{\mathrm{dcov}(X,X) \mathrm{dcov}(Y,Y)}} \\ - \frac{\mathrm{dcov}(X,Y)}{2\sqrt{\mathrm{dcov}(X,X)^3\mathrm{dcov}(Y,Y)}} \\ - \frac{\mathrm{dcov}(X,Y)}{2\sqrt{\mathrm{dcov}(X,X)\mathrm{dcov}(Y,Y)^3}}
                                                          \end{pmatrix},
    W(\lambda)\right\rangle
  \end{align*}
  where we have used the self-adjointness of the matrix $\Gamma^\frac{1}{2}$. Using the notation
  $$
    \alpha = \alpha(\mathbb{P}^{X,Y}) := \Gamma^\frac{1}{2}\begin{pmatrix}
      \frac{1}{\sqrt{\mathrm{dcov}(X,X) \mathrm{dcov}(Y,Y)}} \\ - \frac{\mathrm{dcov}(X,Y)}{2\sqrt{\mathrm{dcov}(X,X)^3\mathrm{dcov}(Y,Y)}} \\ - \frac{\mathrm{dcov}(X,Y)}{2\sqrt{\mathrm{dcov}(X,X)\mathrm{dcov}(Y,Y)^3}}
    \end{pmatrix}
  $$
  the above identity can be written as $L(\lambda) = \lambda \langle\alpha, W(\lambda)\rangle$. $\left\langle \alpha, W\right\rangle$ is a centred Gaussian process with covariance function
  \begin{align*}
    \mathrm{Cov}\left(\left\langle \alpha, W(\lambda_1)\right\rangle, \left\langle \alpha, W(\lambda_2)\right\rangle\right) & = \mathrm{Cov}\left(\sum_{i=1}^3 \alpha_i W_i(\lambda_1), \sum_{j=1}^3 \alpha_j W_j(\lambda_2)\right) \\
                                                                                                                            & = \sum_{i=1}^3 \alpha_i^2 \mathrm{Cov}\left(W_i(\lambda_1), W_i(\lambda_2)\right)                     \\
                                                                                                                            & = \|\alpha\|_2^2 (\lambda_1 \land \lambda_2).
  \end{align*}
  Thus, $\left\langle \alpha, W\right\rangle$ is equal in distribution to $\|\alpha\|_2 B$ for a standard Brownian motion $B = (B(\lambda))_{\lambda \geq 0}$. Consequently,
  \begin{equation}
    \label{eq:Zn_konvergenz}
    \sqrt{n}(Z_n - Z) \rightsquigarrow \left[\|\alpha\|_2 \lambda B(\lambda)\right]_{0 \leq \lambda \leq 1}.
  \end{equation}
  Furthermore, we have the identity
  \begin{align*}
     & \sqrt{n}\left(\lambda \,\mathrm{dcor}_{\lfloor n\lambda \rfloor}(X,Y) - \lambda\, \mathrm{dcor}_n(X,Y)\right)                                                                                   \\
     & \quad= \sqrt{n}\left(\left[\lambda \,\mathrm{dcor}_{\lfloor n\lambda \rfloor}(X,Y) - \lambda\, \mathrm{dcor}(X,Y)\right] - \lambda\left[ \mathrm{dcor}_n(X,Y) -\mathrm{dcor}(X,Y)\right]\right) \\
     & \quad = \sqrt{n}[Z_n(\lambda) - Z(\lambda)] - \lambda \sqrt{n}[Z_n(1) - Z(1)].
  \end{align*}
  Define the maps $\phi : \ell^\infty[0,1] \to \ell^\infty[0,1]$ and $\psi : \ell^\infty[0,1] \to \mathbb{R}^2$ by
  $$
    \phi(S)[\lambda] = S(\lambda) - \lambda S(1)
  $$
  and, with $\|\cdot\|_{L_2(\gamma)}$ denoting the $L_2$ norm with respect to $\gamma$,
  $$
    \psi(S) = \begin{pmatrix}
      S(1) \\
      \|\phi(S)\|_{L_2(\gamma)}.
    \end{pmatrix}
  $$
  Then $\phi$ and therefore $\psi$ are continuous, and it holds by Eq.\@ \eqref{eq:Zn_konvergenz} and the continuous mapping theorem \citep[Theorem 1.3.6 in][]{wellner_vandervaart:1996} that
  $$
    \sqrt{n}\begin{pmatrix}
      \mathrm{dcor}_n(X,Y) - \mathrm{dcor}(X,Y) \\ V_{n,\mathrm{dcor}}
    \end{pmatrix} = \psi(\sqrt{n}[Z_n - Z]) \xrightarrow[n \to \infty]{\mathcal{D}} \psi\left[\left[\|\alpha\|_2 \lambda B(\lambda)\right]_{0 \leq \lambda \leq 1}\right],
  $$
  which proves Theorem \ref{thm:schwache_konvergenz_dcor} with $\tau = \|\alpha\|_2$.
\end{proof}

\begin{proof}[Proof of Theorem \ref{thm:schwache_konvergenz_dcov}]
  By Theorem \ref{thm:dcov_prozesskonvergenz} and using the fact that
  $$
    \left|\frac{\lfloor n\lambda \rfloor}{n} \mathrm{dcov}_{\lfloor n\lambda\rfloor}(X,Y) - \lambda \mathrm{dcov}_{\lfloor n\lambda \rfloor}(X,Y)\right| \leq \frac{1}{n} \mathrm{dcov}_{\lfloor n\lambda\rfloor}(X,Y) = \mathcal{O}_{a.s.}\left(\frac{1}{n}\right)
  $$
  due to the almost sure convergence of $\mathrm{dcov}_n(X,Y)$ \citep[Theorem 1 in][]{kroll:2022}, we get
  \begin{equation}
    \label{eq:dcov_konvergenz}
    \sqrt{n}[\lambda\,\mathrm{dcov}_{\lfloor n\lambda\rfloor}(X,Y) - \lambda\,\mathrm{dcov}(X,Y)]_{0 \leq \lambda \leq 1} \rightsquigarrow \langle \gamma_{\mathrm{row}, 1}^\top, W\rangle \overset{\mathcal{D}}{=} \|\gamma_{\mathrm{row}, 1}\|_2  B,
  \end{equation}
  where $\gamma_{\mathrm{row}, 1}$ is the first row of the covariance matrix $\Gamma^{1/2}$, $B$ is a standard Brownian motion and the equality sign overset with the symbol $\mathcal{D}$ signifies equality in distribution. But the first row of the covariance matrix $\Gamma = (\gamma_{ij})_{1 \leq i,j \leq 3}$ contains at least one non-zero entry, namely
  $$
    \gamma_{11} = \mathbb{E}\left[h_1\left(Z_1 ; \mathbb{P}^{X,Y}\right)\right] + 2 \sum_{k=2}^\infty \mathbb{E}\left[h_1\left(Z_1 ; \mathbb{P}^{X,Y}\right)h_1\left(Z_k ; \mathbb{P}^{X,Y}\right)\right] > 0,
  $$
  where the specific form of $\gamma_{11}$ is due to Eqs.\@ \eqref{eq:gamma_identity} and \eqref{eq:xi_identity}, and the fact that it is strictly greater than $0$ is a consequence of Corollary 10.8 (II) in \cite{bradley:2007}. This also implies that the first row of the square root $\Gamma^{1/2}$ cannot consist of only $0$-entries. Therefore $\|\gamma_{\mathrm{row}, 1}\|_2 > 0$, and with this the claim of Theorem \ref{thm:schwache_konvergenz_dcov} follows from Eq.\@ \eqref{eq:dcov_konvergenz} and the continuous mapping theorem just like in the proof of Theorem \ref{thm:schwache_konvergenz_dcor}.
\end{proof}


\subsection{Proof of Proposition \ref{cor:test_dcov}}
We have the identity
$$
  \frac{ \mathrm{dcor}_n(X,Y) - \Delta}{V_{n,\mathrm{dcor}}} = \frac{ \mathrm{dcor}_n(X,Y) - \mathrm{dcor}(X,Y)}{V_{n,\mathrm{dcor}}} + \frac{\sqrt{n}(\mathrm{dcor}(X,Y) - \Delta)}{\sqrt{n} V_{n,\mathrm{dcor}}}.
$$
By Theorem \ref{thm:schwache_konvergenz_dcor} and the continuous mapping theorem, we have
$$
  \frac{ \mathrm{dcor}_n(X,Y) - \mathrm{dcor}(X,Y)}{V_{n,\mathrm{dcor}}} \xrightarrow[n \to \infty]{\mathcal{D}} W.
$$
Thus, by the continuous mapping theorem and Slutsky's lemma,
$$
  \frac{ \mathrm{dcor}_n(X,Y) - \Delta}{V_{n,\mathrm{dcor}}} \xrightarrow[n \to \infty]{\mathcal{D}} \begin{cases}-\infty \qquad &\textrm{if } \mathrm{dcor}(X,Y) < \Delta, \\ W \qquad &\textrm{if } \mathrm{dcor}(X,Y) = \Delta, \\ +\infty \qquad &\textrm{if } \mathrm{dcor}(X,Y) > \Delta.\end{cases}
$$

\subsection{Proof of Theorem \ref{thm:independence_weak_convergence}}

It is not complicated (but tedious) to verify that, due to the independence between $X$ and $Y$, the kernel function $h$ is $1$-degenerate and that the function $h_2 = h_2(\cdot, \cdot \, ; \mathbb{P}^{X,Y})$ is the kernel of the second-order term in the Hoeffding decomposition of $h$ as explained in Eq.\@ \eqref{eq:hoeffding_decomposition}. Just as in the proof of Theorem \ref{thm:dcov_prozesskonvergenz}, we can show that Eq.\@ \eqref{eq:kernel_h_moment_bound} holds. By Corollary \ref{cor:hoeffding_zerlegung_bestimmt_asymptotik}, we then get
\begin{equation}
  \label{eq:independence_identity_second_order}
  \mathrm{dcov}_{\lfloor n \lambda \rfloor}(X,Y) = \frac{15}{(\lfloor n\lambda \rfloor)^{2}} \sum_{1 \leq i,j \leq \lfloor n\lambda \rfloor} h_2\left(Z_i, Z_j ; \mathbb{P}^{X,Y}\right) + R_n(\lambda),
\end{equation}
where
\begin{equation}
  \label{eq:remainder_Rn}
  \sup_{0 \leq \lambda \leq 1} |R_n(\lambda)| = o_\mathbb{P}\left(n^{-1}\right)
\end{equation}
The moments $\mathbb{E}\left[|h_2(Z_i,Z_j)|^{2+\varepsilon/2}\right]$ are bounded uniformly in $i,j$ as a consequence of Eq.\@ \eqref{eq:kernel_h_moment_bound} and Jensen's inequality. Furthermore, by Proposition 3.5 in \cite{lyons:2013}, we have the identity
\begin{align*}
  h_2\left(z_i, z_j ; \mathbb{P}^{X,Y}\right) & = c \left\langle \hat{\phi}(x_i), \hat{\phi}(x_j)\right\rangle_{H_1}\left\langle \hat{\psi}(y_i), \hat{\psi}(y_j)\right\rangle_{H_2} \\
                                              & = c\left\langle \hat{\phi}(x_i) \otimes \hat{\psi}(y_i), \hat{\phi}(x_j)\otimes \hat{\psi}(y_j)\right\rangle_{H_1 \otimes H_2},
\end{align*}
where $c > 0$ is some universal constant, $\hat{\phi} : \mathcal{X} \to H_1$ and $\hat{\psi} : \mathcal{Y} \to H_2$ are certain embeddings into two separable Hilbert spaces $H_1$ and $H_2$, and $H_1 \otimes H_2$ denotes the tensor product Hilbert space of $H_1$ and $H_2$. The details behind all of these constructions can be found in Section 3 in \cite{lyons:2013}. For our purposes, it suffices to note that this identity of $h_2$ as an inner product in some Hilbert space also means that it is non-negative definite, i.e.\@
$$
  \sum_{i,j=1}^m \alpha_i \alpha_j h_2\left(z_i, z_j ; \mathbb{P}^{X,Y}\right) \geq 0
$$
for all $m \in \mathbb{N}$, $\alpha_1, \ldots, \alpha_m \in \mathbb{R}$ and $z_1, \ldots, z_m \in \mathcal{X} \times \mathcal{Y}$. By the Moore-Aronszajn theorem \citep[Theorem 2 in][]{aronszajn:1943}, there is a unique Reproducing Kernel Hilbert Space $H$ associated with $h_2$. Furthermore, $h_2$ is obviously continuous, and we have already seen that it has finite moments on the diagonal. By Lemma 2.3 and Corollary 3.5 in \cite{steinwart_scovel:2012}, we therefore get the following Mercer representation of $h_2$: For any fixed $z$ in the support of $\mathbb{P}^{X,Y}$, it holds for all $z' \in \mathcal{X} \times \mathcal{Y}$ that
\begin{equation}
  \label{eq:mercer}
  h_2\left(z,z' ; \mathbb{P}^{X,Y}\right) = \sum_{k=1}^\infty \mu_k \varphi_k(z) \varphi_k(z'),
\end{equation}
where the objects $(\mu_k, \varphi_k)$ are as in the statement of our theorem, and the series on the right-hand side converges absolutely. In the following, we assume that $\mu_k > 0$ for all $k \in \mathbb{N}$; the case where $\mu_k > 0$ for only finitely many $k \in \mathbb{N}$ can be covered by essentially the same arguments. $\mathcal{X}$ and $\mathcal{Y}$ are separable metric spaces, and hence second-countable Hausdorff spaces. Products of second-countable Hausdorff spaces are again second-countable and Hausdorff, which in particular implies $\mathbb{P}^{X,Y}(\mathcal{X} \times \mathcal{Y} \setminus S) = 0$, where $S$ denotes the support of $\mathbb{P}^{X,Y}$. This implies that Eq.\@ \eqref{eq:mercer} holds for $\mathbb{P}^{X,Y} \otimes \mathbb{P}^{X,Y}$-almost all points $(z,z')$.

By Eqs.\@ \eqref{eq:independence_identity_second_order} and \eqref{eq:mercer}, we get
\begin{align}
  \begin{split}
    \label{eq:nQn_identity_1}
    nQ_n(\lambda) & = \frac{15}{n} \sum_{1 \leq i,j \leq \lfloor n\lambda \rfloor} \sum_{k=1}^\infty \mu_k \varphi_k(Z_i)\varphi_k(Z_j) + n R_n(\lambda)      \\
                  & = 15  \sum_{k=1}^\infty \frac{\mu_k}{n}\sum_{1 \leq i,j \leq \lfloor n\lambda \rfloor}\mu_k \varphi_k(Z_i)\varphi_k(Z_j) + n R_n(\lambda) \\
                  & = 15 \sum_{k=1}^\infty \mu_k \left(\frac{1}{\sqrt{n}} \sum_{i=1}^{\lfloor n\lambda \rfloor}\varphi_k(Z_i)\right)^2 + nR_n(\lambda)
  \end{split}
\end{align}
almost surely. Now define the separable Hilbert space
$$
  \ell_2(\mu) = \left\{(x_k)_{k \in \mathbb{N}} ~\Big|~\sum_{k=1}^\infty \mu_k x_k^2 < \infty\right\}
$$
equipped with the inner product
$$
  \langle x,y\rangle_{\ell_2(\mu)} = \sum_{k=1}^\infty \mu_k x_k y_k,
$$
and write
\begin{align*}
  f_i =                    & (\varphi_k(Z_i))_{k \in \mathbb{N}},                                                                        \\
  \zeta_n(\lambda)         & = \frac{1}{\sqrt{n}} \sum_{i=1}^{\lfloor n\lambda \rfloor} f_i,                                             \\
  \tilde{\zeta}_n(\lambda) & = \zeta_n(\lambda) + \frac{n\lambda - \lfloor n\lambda \rfloor}{\sqrt{n}} f_{\lfloor n\lambda \rfloor + 1}.
\end{align*}
The $f_i$ are strictly stationary and almost surely $\ell_2(\mu)$-valued by Eq.\@ \eqref{eq:mercer}, and Eq.\@ \eqref{eq:nQn_identity_1} can be restated as
\begin{equation}
  \label{eq:nQn_identity_3}
  nQ_n(\lambda) = 15 \|\zeta_n(\lambda)\|_{\ell_2(\mu)}^2 + nR_n(\lambda).
\end{equation}
Furthermore, by the degeneracy of the kernel $h_2$, it holds for any $k \in \mathbb{N}$ that
$$
  \mathbb{E}\left[\varphi_k(Z_1)\right] = \mu_k^{-1} \int \mathbb{E} \left[h_2\left(Z_1, z ; \mathbb{P}^{X,Y}\right)\right] \varphi_k(z) ~\mathrm{d}\mathbb{P}^{X,Y}(z)= 0,
$$
which also implies that
$$
  \mathbb{E} \langle f_1, x\rangle_{\ell_2(\mu)} = \mathbb{E}\left[\sum_{k=1}^\infty \mu_k x_k \varphi_k(Z_1)\right] = \sum_{k=1}^\infty \mu_k x_k \mathbb{E}\varphi_k(Z_1) = 0
$$
for any $x \in \ell_2(\mu)$, i.e.\@ the $f_i$ are centred at expectation \citep[in the general sense of Hilbert space valued random variables; cf.\@ Definition 2.3 in][]{kuo:1975}. Furthermore, for any $p > 0$,
$$
  \mathbb{E}\left[\|f_1\|_{\ell_2(\mu)}^p\right] = \mathbb{E}\left[\left(\sum_{k=1}^\infty \mu_k \varphi_k^2(Z_1)\right)^{p/2}\right] = \mathbb{E}\left[h_2\left(Z_1, Z_1 ; \mathbb{P}^{X,Y}\right)^{p/2}\right]
$$
by Eq.\@ \eqref{eq:mercer}, and so the $f_i$ have finite $(4+\varepsilon)$-moments. Therefore,
\begin{equation}
  \label{eq:zeta_n_convergence}
  \tilde{\zeta}_n \rightsquigarrow W = (W_k)_{k \in \mathbb{N}}
\end{equation}
by Theorem 11 a) in \cite{dehling:1983}, where the convergence holds in the space of all continuous functions from $[0,1]$ to $\ell_2(\mu)$, equipped with the supremum norm. The process $(W_k)_{k \in \mathbb{N}}$ is that from the statement of our theorem. Furthermore,
\begin{equation}
  \label{eq:diff_normal_continuous}
  \mathbb{E}\left[\sup_{0 \leq \lambda \leq 1}\|\zeta_n - \tilde{\zeta}_n\|_{\ell_2(\mu)}^{2+\varepsilon}\right] \leq \sum_{i=1}^n \mathbb{E}\left[\|n^{-1/2} f_i\|_{\ell_2(\mu)}^{2+\varepsilon}\right] = n^{-\varepsilon/2} \mathbb{E}\left[\|f_1\|_{\ell_2(\mu)}^{2+\varepsilon}\right] \xrightarrow[n \to \infty]{} 0.
\end{equation}
It follows from Eqs.\@ \eqref{eq:remainder_Rn} and \eqref{eq:nQn_identity_3} through \eqref{eq:diff_normal_continuous}, combined with the continuous mapping theorem \citep[Theorem 1.3.6 in][]{wellner_vandervaart:1996} that
$$
  nQ_n \rightsquigarrow 15 \|W\|_{\ell_2(\mu)}^2 = 15 \sum_{k=1}^\infty \mu_k W_k^2.
$$
This concludes the proof.

\section*{Acknowledgements}
Financial support by the Deutsche Forschungsgemeinschaft (DFG, German Research Foundation; Project-ID 520388526;  TRR 391:  Spatio-temporal Statistics for the Transition of Energy and Transport) is gratefully acknowledged.

\bibliographystyle{abbrvnat}
\bibliography{rel_hyp_dcov}
\end{document}